\documentclass[11pt,a4paper]{article}
\usepackage[english]{babel}
\usepackage{times}
\usepackage{graphicx}
\usepackage{amscd}
\usepackage{amsmath}
\usepackage{amsfonts}
\usepackage{amssymb}
\usepackage{amsthm}
\usepackage{latexsym}
\newtheorem{theorem}{Theorem}

\newtheorem{corollary}[theorem]{Corollary}

\newtheorem{problem}[theorem]{Problem}
\newtheorem{remark}[theorem]{Remark}

\def\neweq#1{\begin{equation}\label{#1}}
\def\endeq{\end{equation}}
\def\eq#1{(\ref{#1})}

\def\endproof{\hfill $\Box$}

\renewcommand{\arraystretch}{1.5}
\newcommand{\R}{\mathbb{R}}
\newcommand{\N}{\mathbb{N}}
\newcommand{\eps}{\varepsilon}
\newcommand{\E}{{\cal E}}
\begin{document}

\title{Nonlinearity in oscillating bridges}

\author{Filippo GAZZOLA\\
{\small Dipartimento di Matematica del Politecnico, Piazza L. da Vinci 32 - 20133 Milano (Italy)}}

\date{}
\maketitle
\begin{abstract}
We first recall several historical oscillating bridges that, in some cases, led to collapses. Some of them are quite recent and show that,
nowadays, oscillations in suspension bridges are not yet well understood. Next, we survey some attempts to model bridges with differential equations. Although
these equations arise from quite different scientific communities, they display some common features. One of them, which we believe to be incorrect, is
the acceptance of the linear Hooke law in elasticity. This law should be used only in presence of small deviations from equilibrium, a situation
which does not occur in strongly oscillating bridges. Then we discuss a couple of recent models whose solutions exhibit self-excited oscillations,
the phenomenon visible in real bridges. This suggests a different point of view in modeling equations and gives a strong hint how to modify the existing models
in order to obtain a reliable theory. The purpose of this paper is precisely to highlight the necessity of revisiting classical models, to introduce reliable models,
and to indicate the steps we believe necessary to reach this target.\par\noindent
\end{abstract}

\tableofcontents

\vfill\eject

\section{Introduction}

The story of bridges is full of many dramatic events, such as uncontrolled oscillations which, in some cases, led to collapses.
To get into the problem, we invite the reader to have a look at the videos \cite{assago,london,tacoma,volgograd}.
These failures have to be attributed to the action of external forces, such as the wind or traffic loads, or to macroscopic mistakes in the projects.
{From} a theoretical point of view, there is no satisfactory mathematical model which, up to nowadays, perfectly describes the complex behavior of bridges.
And the lack of a reliable analytical model precludes precise studies both from numerical and engineering points of views.\par
The main purpose of the present paper is to show the necessity of revisiting existing models since they fail to describe the behavior of real bridges. We will
explain which are the weaknesses of the so far considered equations and suggest some possible improvements according to the fundamental rules of classical
mechanics. Only with some nonlinearity and with a sufficiently large number of degrees of freedom several behaviors may be modeled.
We do not claim to have a perfect model, we just wish to indicate the way to reach it. Much more work is needed and we explain what we believe to
be the next steps.\par
We first survey and discuss some historical events, we recall what is known in elasticity theory, and we describe in full detail the existing models.
With this database at hand, our purpose is to analyse the oscillating behavior of certain bridges, to determine the causes of oscillations,
and to give an explanation to the possible appearance of different kinds of oscillations, such as torsional oscillations. Due to the lateral sustaining cables,
suspension bridges most emphasise these oscillations which, however, also appear in other kinds of bridges: for instance, light pedestrian
bridges display similar behaviors even if their mechanical description is much simpler.\par
According to \cite{goldstein}, chaos is a disordered and unpredictable behavior of solutions in a dynamical system.
With this characterization, there is no doubt that chaos is somehow present in the disordered and unpredictable oscillations of bridges. From
\cite[Section 11.7]{goldstein} we recall a general principle (GP) of classical mechanics:
\begin{center}
\begin{minipage}{162mm}
{\bf (GP)} {\em The minimal requirements for a system of first-order equations to exhibit chaos is that they be nonlinear and have at least three variables.}
\end{minipage}
\end{center}
This principle suggests that
\begin{center}
{\bf any model aiming to describe oscillating bridges should be nonlinear and with enough degrees of freedom.}
\end{center}

Most of the mathematical models existing in literature fail to satisfy (GP) and, therefore, must be accordingly modified. We suggest possible modifications
of the corresponding differential equations and we believe that, if solved, this would lead to a better understanding of the underlying phenomena
and, perhaps, to several practical actions for the plans of future bridges, as well as remedial measures for existing structures. In particular, one of
the major scopes of this paper is to convince the reader that linear theories are not suitable for the study of bridges oscillations whereas,
although they are certainly too naive, some recent nonlinear models do display self-excited oscillations as visible in bridges.\par
In Section \ref{story}, we collect a number of historical events and observations about bridges, both suspended and not.
A complete story of bridges is far beyond the scopes of the present paper and the choice of events is mainly motivated by the phenomena that they displayed.
The description of the events is accompanied by comments of engineers and of witnesses, and by possible theoretical explanations of the observed phenomena.
The described events are then used in order to figure out a common behavior of oscillating bridges; in particular, it appears that the requirements of (GP) must
be satisfied. Recent events testify that the problems of controlling and forecasting bridges oscillations is still unsolved.\par
In Section \ref{howto}, we discuss several equations appearing in literature as models for oscillating bridges. Most of them use in some point the well-known
linear Hooke law (${\cal LHL}$ in the sequel) of elasticity. This is what we believe to be a major weakness, but not the only one, of all these models.
This is also the opinion of McKenna \cite[p.16]{mckmonth}:
\begin{center}
\begin{minipage}{162mm}
{\em We doubt that a bridge oscillating up and down by about 10 meters every 4 seconds obeys Hooke's law.}
\end{minipage}
\end{center}
{From} \cite{britannica}, we recall what is known as ${\cal LHL}$.
\begin{center}
\begin{minipage}{162mm}
{\em The linear Hooke law (${\bf {\cal LHL}}$) of elasticity, discovered by the English scientist Robert Hooke in 1660, states that for relatively small
deformations of an object, the displacement or size of the deformation is directly proportional to the deforming force or load. Under these conditions
the object returns to its original shape and size upon removal of the load. ... At relatively large values of applied force, the deformation
of the elastic material is often larger than expected on the basis of ${\cal LHL}$, even though the material remains elastic and returns to
its original shape and size after removal of the force. ${\cal LHL}$ describes the elastic properties of materials only in the range in which
the force and displacement are proportional.}
\end{minipage}
\end{center}
Hence, by no means one should use ${\cal LHL}$ in presence of large deformations. In such case, the restoring elastic force $f$ is ``more than linear''.
Instead of having the usual form $f(s)=ks$, where $s$ is the displacement from equilibrium and $k>0$ depends on the elasticity of the deformed material,
it has an additional superlinear term $\varphi(s)$ which becomes negligible for small displacements $s$. More precisely,
$$f(s)=ks+\varphi(s)\qquad\mbox{with}\qquad\lim_{s\to0}\frac{\varphi(s)}{s}=0\ .$$
The simplest example of such term is $\varphi(s)=\eps s^p$ with $\eps>0$ and $p>1$; this superlinear term may become arbitrarily small
for $\eps$ small and/or $p$ large. Therefore, the parameters $\eps$ and $p$, which do exist, may be chosen in such a way to describe with
a better precision the elastic behavior of a material when large displacements are involved. As we shall see, this apparently harmless
and tiny nonlinear perturbation has devastative effects on the models and, moreover, it is amazingly useful to display self-excited oscillations
as the ones visible in real bridges. On the contrary, linear models prevent to view the real phenomena which occur in bridges, such as the sudden
increase of the width of their oscillations and the switch to different ones.\par
The necessity of dealing with nonlinear models is by now quite clear also in more general elasticity problems;
from the preface of the book by Ciarlet \cite{ciarletbook}, let us quote
\begin{center}
\begin{minipage}{162mm}
{\em ... it has been increasingly acknowledged that the classical linear equations of elasticity, whose mathematical theory is now firmly
established, have a limited range of applicability, outside of which they should be replaced by genuine nonlinear equations that they in effect approximate.}
\end{minipage}
\end{center}

In order to model bridges, the most natural way is to view the roadway as a thin narrow rectangular plate. In Section \ref{elasticity}, we quote several
references which show that classical linear elastic models for thin plates do not describe with a sufficient accuracy large deflections of a plate. But even linear
theories present considerable difficulties and a further possibility is to view the bridge as a one dimensional beam; this model is much simpler but, of course, it
prevents the appearance of possible torsional oscillations. This is the main difficulty in modeling bridges: find simple models which, however, display the same
phenomenon visible in real bridges.\par
In Section \ref{models} we survey a number of equations arising from different scientific communities. The first equations are based on engineering models
and mainly focus the attention on quantitative aspects such as the exact values of the parameters involved. Some other equations are more related
to physical models and aim to describe in full details all the energies involved. Finally, some of the equations are purely mathematical models
aiming to reach a prototype equation and proving some qualitative behavior. All these models have to face a delicate choice: either consider uncoupled
behaviors between vertical and torsional oscillations of the roadway or simplify the model by decoupling these two phenomena. In the former case,
the equations have many degrees of freedom and become terribly complicated: hence, very few results can be obtained. In the latter case, the model
fails to satisfy the requirements of (GP) and appears too far from the real world.\par
As a compromise between these two choices, in Section \ref{blup} we recall the model introduced in \cite{gazpav,gazpav3} which describes vertical
oscillations and torsional oscillations of the roadway within the same simplified beam equation. The solution to the equation
exhibits self-excited oscillations quite similar to those observed in suspension bridges. We do not believe that the simple equation considered models
the complex behavior of bridges but we do believe that it displays the same phenomena as in more complicated models closer related to bridges.
In particular, finite time blow up occurs with wide oscillations. These phenomena are typical of differential equations of at least
fourth order since they do not occur in lower order equations, see \cite{gazpav}. We also show that the same phenomenon is visible in a $2\times2$ system of
nonlinear ODE's of second order related to a system suggested by McKenna \cite{mckmonth}.\par
Putting all together, in Section \ref{afford} we afford an explanation in terms of the energies involved.
Starting from a survey of milestone historical sources \cite{bleich,tac2}, we attempt a qualitative but detailed energy balance and we attribute the appearance
of torsional oscillations in bridges to some ``hidden'' elastic energy which is not visible since it involves second order derivatives of the displacement of the
bridge: this suggests an analogy between bridges oscillations and a ball bouncing on the floor. The discovery of the phenomenon usually called in literature
{\em flutter speed} has to be attributed to Bleich \cite{bleichsolo}; in our opinion, the flutter speed should be seen as a {\bf critical energy threshold} which,
if exceeded, gives rise to uncontrolled phenomena such as torsional oscillations. We give some hints on how to determine the critical energy threshold,
according to some eigenvalue problems whose
eigenfunctions describe the oscillating modes of the roadway.\par
In bridges one should always expect vertical oscillations and, in case they become very large, also torsional oscillations; in order to display
the possible transition between these two kinds of oscillations, in Section \ref{newmodel} we suggest a new equation as a model for suspension
bridges, see \eq{truebeam}. With all the results and observations at hand, in Section \ref{possibleTacoma} we also attempt
a detailed description of what happened on November 10, 1940, the day when the Tacoma Narrows Bridge collapsed.
As far as we are aware a universally accepted explanation of this collapse in not yet available. Our explanation fits with all
the material developed in the present paper. This allows us to suggest a couple of precautions when planning future bridges, see Section \ref{howplan}.\par
We recently had the pleasure to participate to a conference on bridge maintenance, safety and management, see \cite{iabmas}.
There were engineers from all over the world, the atmosphere was very enjoyable
and the problems discussed were extremely interesting. And there was a large number of basic questions still
unsolved, most of the results and projects had some percentage of incertitude. Many talks were devoted to suggest new equations to model the
studied phenomena and to forecast the impact of new structural issues: even apparently simple problems are still unsolved.
We believe this should be a strong motivation for many mathematicians (from mathematical physics, analysis, numerics) to get interested in
bridges modeling, experiments, and performances. Throughout the paper we suggest a number of open problems which, if solved, could be a good starting point
to reach a deeper understanding of oscillations in bridges.

\section{What has been observed in bridges}\label{story}

A simplified picture of a suspension bridge can be sketched as in Figure \ref{67}
\begin{figure}[ht]
\begin{center}
{\includegraphics[height=32mm, width=72mm]{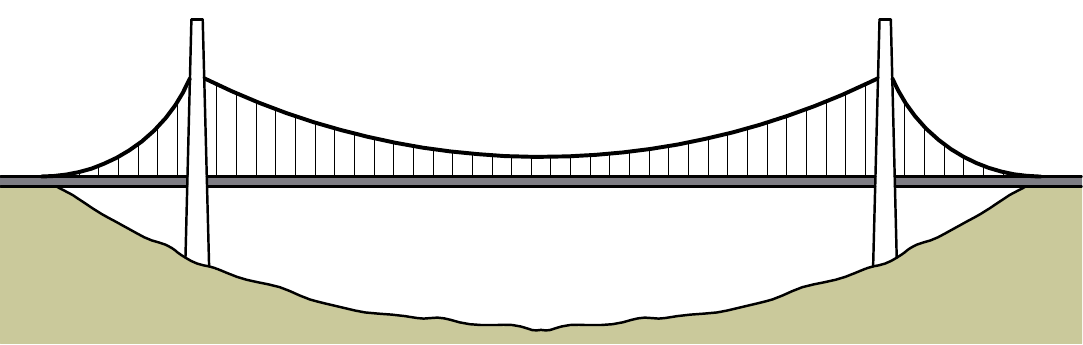}}\qquad\qquad
{\includegraphics[height=32mm, width=72mm]{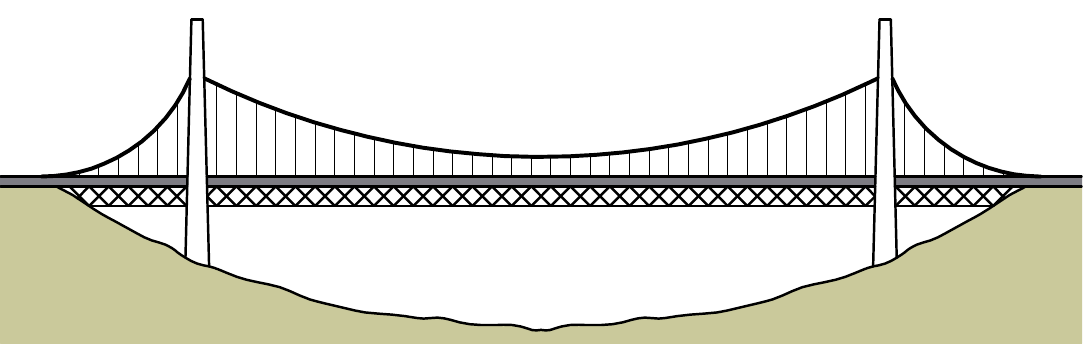}}
\caption{Suspension bridges without girder and with girder.}\label{67}
\end{center}
\end{figure}
where one sees the difference between the elastic structure of a bridge without girder and the more stiff structure of a bridge with girder.\par
Although the first project of a suspension bridge is due to the Italian engineer Verantius around 1615, see \cite{veranzio} and \cite[p.7]{navier2}
or \cite[p.16]{kawada2}, the first suspension bridges were built only about two centuries later in Great Britain. According to \cite{bender},
\begin{center}
\begin{minipage}{162mm}
{\em The invention of the suspension bridges by Sir Samuel Brown sprung from the sight of a spider's web hanging across the path of the inventor, observed
on a morning's walk, when his mind was occupied with the idea of bridging the Tweed.}
\end{minipage}
\end{center}
Samuel Brown (1776-1852) was an early pioneer of suspension bridge design and construction. He is best known for the Union Bridge of 1820, the first vehicular
suspension bridge in Britain.\par
An event deserving mention is certainly the inauguration of the Menai Straits Bridge, in 1826. The project of the bridge was due to Thomas Telford
and the opening of the bridge is considered as the beginning of a new science nowadays known as ``Structural Engineering''. The construction
of this bridge had a huge impact in the English society,
a group of engineers founded the ``Institution of Civil Engineers'' and Telford was elected the first president of this association. In 1839 the Menai
Bridge collapsed due to a hurricane. In that occasion, unexpected oscillations appeared; Provis \cite{provis} provided the following description:
\begin{center}
\begin{minipage}{162mm}
{\em ... the character of the motion of the platform was not that of a simple undulation, as had been anticipated, but the movement of the
undulatory wave was oblique, both with respect to the lines of the bearers, and to the general direction of the bridge.}
\end{minipage}
\end{center}

Also the Broughton Suspension Bridge was built in 1826. It collapsed in 1831 due to mechanical resonance induced by troops marching over the bridge
in step. A bolt in one of the stay-chains snapped, causing the bridge to collapse at one end, throwing about 40 men into the river. As a consequence
of the incident, the British Army issued an order that troops should ``break step'' when crossing a bridge. These two pioneering bridges already
show how the wind and/or traffic loads, both vehicles and pedestrians, play a crucial negative role in the bridge stability.\par
A further event deserving to be mentioned is the collapse of the Brighton
Chain Pier, built in 1823. It collapsed a first time in 1833, it was rebuilt and partially destroyed once again in 1836. Both the
collapses are attributed to violent windstorms. For the second collapse a witness, William Reid, reported valuable
observations and sketched a picture illustrating the destruction \cite[p.99]{reid}, see Figure \ref{brighton}
\begin{figure}[ht]
\begin{center}
{\includegraphics[height=39mm, width=78mm]{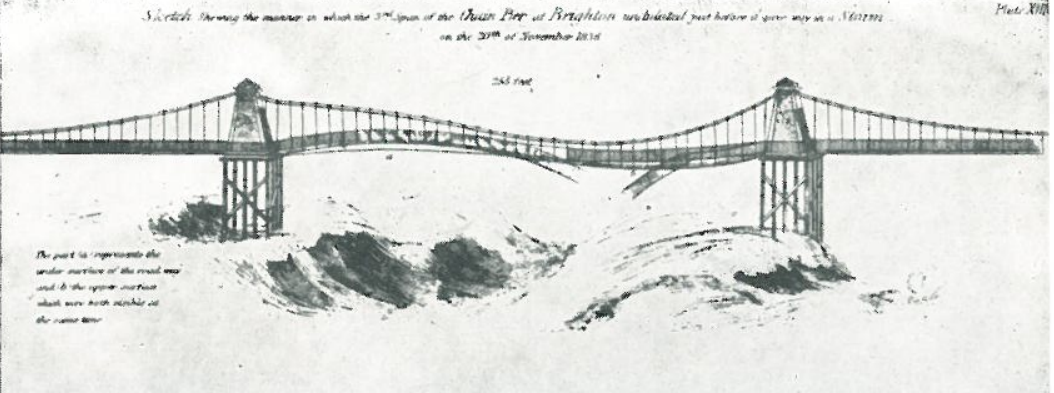}}\quad
{\includegraphics[height=39mm, width=78mm]{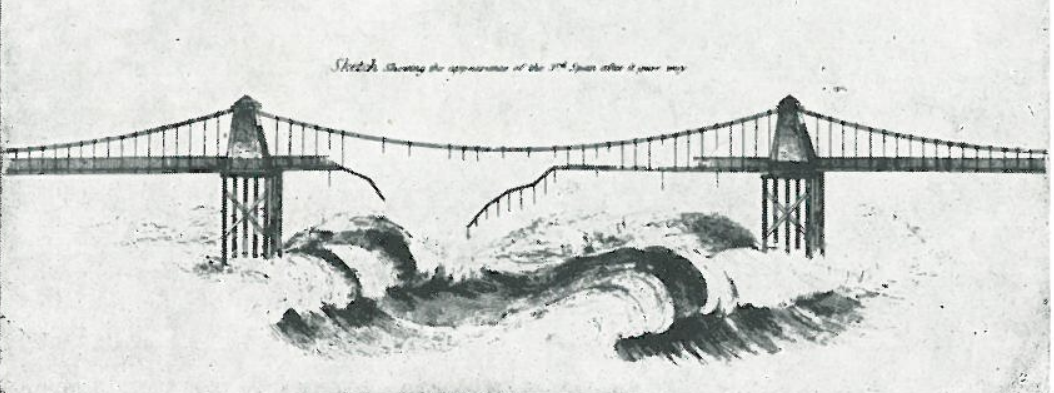}}
\caption{Destruction of the Brighton Chain Pier.}\label{brighton}
\end{center}
\end{figure}
which is taken from \cite{rocard}. This is the first reliable report on oscillations appearing in bridges, the most intriguing part of the report
being \cite{reid,rocard}:
\begin{center}
\begin{minipage}{162mm}
{\em For a considerable time, the undulations of all the spans seemed nearly equal ... but soon after midday the lateral oscillations of the third
span increased to a degree to make it doubtful whether the work could withstand the storm; and soon afterwards the oscillating motion across the
roadway, seemed to the eye to be lost in the undulating one, which in the third span was much greater than in the other three; the undulatory
motion which was along the length of the road is that which is shown in the first sketch; but there was also an oscillating motion of the great chains
across the work, though the one seemed to destroy the other ...}
\end{minipage}
\end{center}
More comments about this collapse are due to Russell \cite{russell}; in particular, he claims that
\begin{center}
\begin{minipage}{162mm}
{\em ... the remedies I have proposed, are those by which such destructive vibrations would have been rendered impossible.}
\end{minipage}
\end{center}
These two comments may have several interpretations. However, what appears absolutely clear is that different kinds of oscillations appeared
(undulations, lateral oscillations, oscillation motion of the great chains) and some of them were considered destructive.
Further details on the Brighton Chair Pier collapse may be found in \cite[pp.4-5]{bleich}.\par
Some decades earlier, at the end of the eighteenth century, the German physicist Ernst Chladni was touring Europe and showing,
among other things, the nodal line patterns of vibrating plates, see Figure \ref{patterns}.
\begin{figure}[ht]
\begin{center}
{\includegraphics[height=71mm, width=65mm]{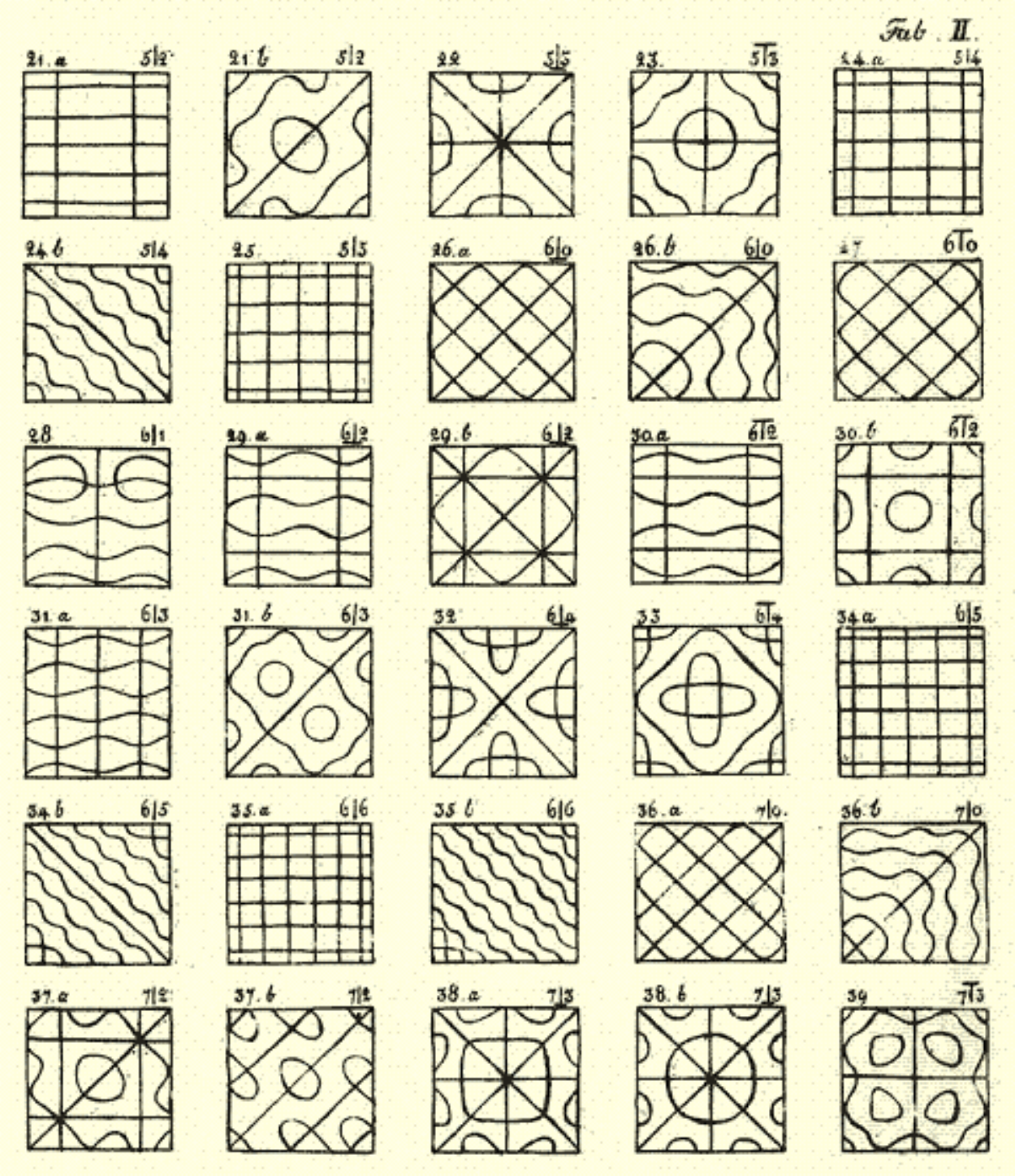}}
\caption{Chladni patterns in a vibrating plate.}\label{patterns}
\end{center}
\end{figure}
Chladni's technique, first published in \cite{chl}, consisted of creating vibrations in a square-shaped metal plate whose surface was covered with light sand.
The plate was bowed until it reached resonance, when the vibration caused the sand to concentrate along the nodal lines of vibrations,
see \cite{chladniexperiment} for the nowadays experiment. This simple but very effective way to display the nodal lines of vibrations was seen by
Navier \cite{navier} as
\begin{center}
\begin{minipage}{162mm}
{\em Les curieuses exp\'eriences de M. Chaldni sur les vibrations des plaques...}
\end{minipage}
\end{center}
It appears quite clearly from Figure \ref{patterns} how complicated may be the vibrations of a thin plate and hence, see Section \ref{elasticity},
of a bridge. And, indeed, the just described events testify that, besides the somehow expected vertical oscillations, also different kinds of oscillations
may appear. For instance, one may have ``an oblique undulatory wave'' or some kind of resonance or the interaction with other structural
components such as the suspension chains. The description of different coexisting forms of oscillations is probably the most important
open problem in suspension bridges.\par
It is not among the scopes of this paper to give the complete story of bridges collapses for which we refer to \cite[Section 1.1]{bleich}, to
\cite[Chapter IV]{rocard}, to \cite{aer,tac1,hayden,ward}, to the recent monographs \cite{akesson,kawada2}, and also to \cite{bridgefailure} for a complete database.
Let us just mention that between 1818 and 1889, ten suspension bridges suffered major damages or collapsed in windstorms, see \cite[Table 1, p.13]{tac1},
which is commented by
\begin{center}
\begin{minipage}{162mm}
{\em An examination of the British press for the 18 years between 1821 and 1839 shows it to be more replete with disastrous news of suspension bridges
troubles than Table 1 reveals, since some of these structures suffered from the wind several times during this period and a number of other
suspension bridges were damaged or destroyed as a result of overloading.}
\end{minipage}
\end{center}

The story of bridges, suspended and not, contains many further dramatic events, an amazing amount of bridges had troubles
for different reasons such as the wind, the traffic loads, or macroscopic mistakes in the project, see e.g.\ \cite{hao,pearson}. Among them,
the most celebrated is certainly the Tacoma Narrows Bridge, collapsed in 1940 just a few months after its opening, both because of the impressive
video \cite{tacoma} and because of the large number of studies that it has inspired starting from the reports \cite{Tacoma1,bleich,tac1,tac3,tac4,tac2,tac5}.\par
Let us recall some observations made on the Tacoma collapse.
Since we were unable to find the Federal Report \cite{Tacoma1} that we repeatedly quote below, we refer to it by trusting the valuable
historical research by Scott \cite{wake} and by McKenna and coauthors, see in particular \cite{mck1,mckmonth,mck,mck4}. A good starting point to describe the
Tacoma collapse is... the Golden Gate Bridge, inaugurated a few years earlier, in 1937. This bridge is usually classified as
``very flexible'' although it is strongly stiffened by a thick girder, see Figure \ref{69}.
\begin{figure}[ht]
\begin{center}
{\includegraphics[height=39mm, width=78mm]{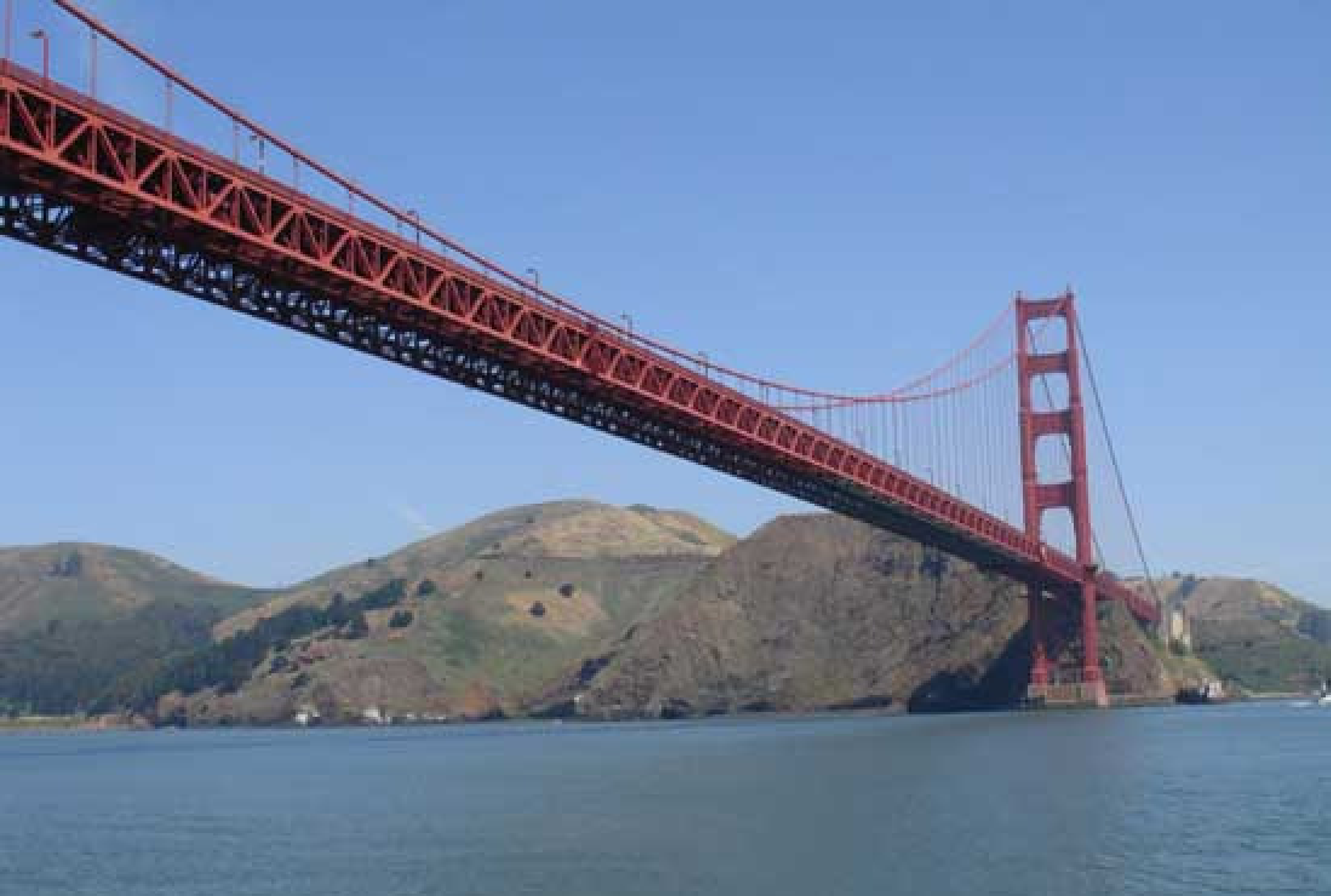}}\quad
{\includegraphics[height=39mm, width=78mm]{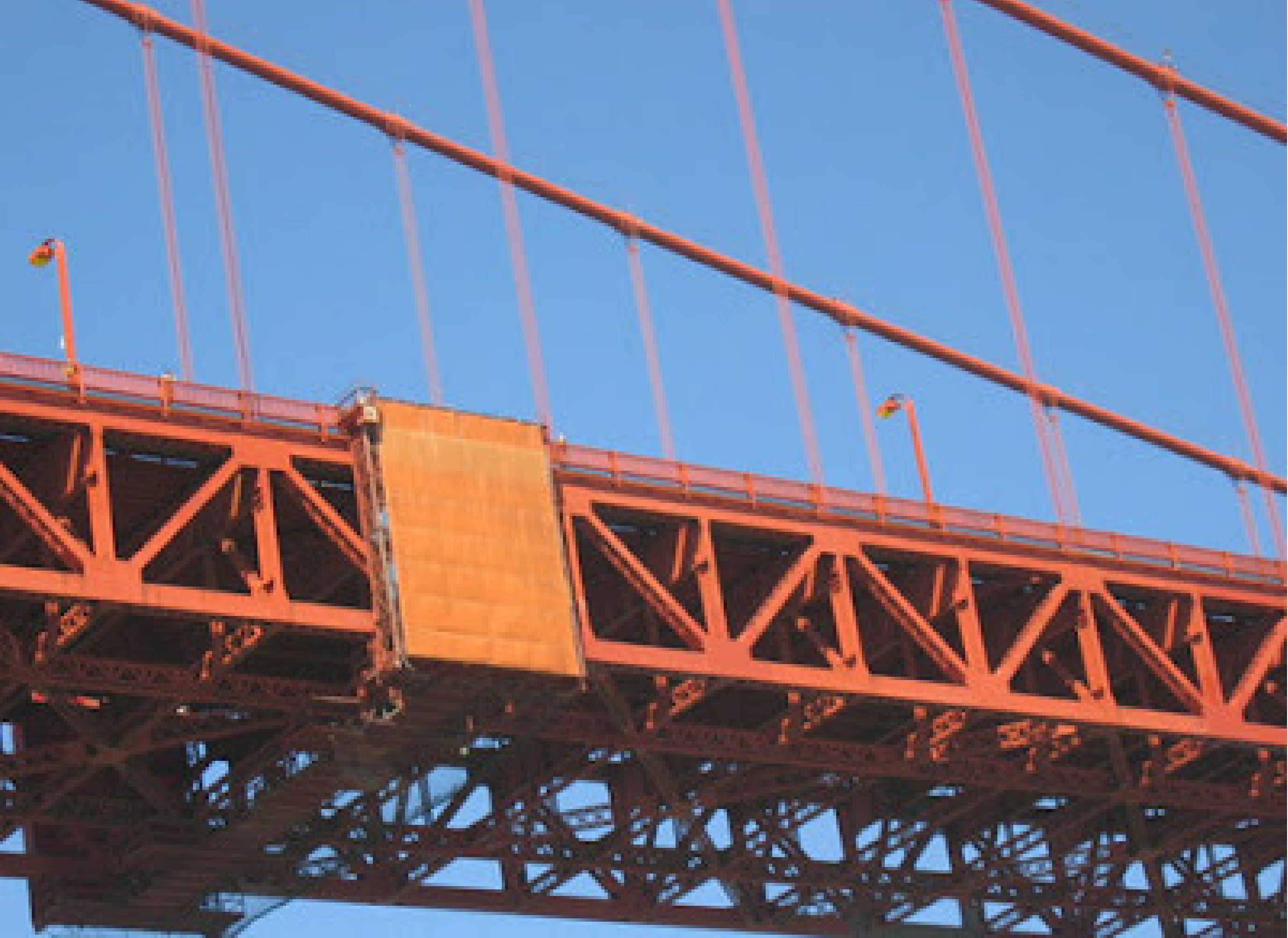}}
\caption{Girder at the Golden Gate Bridge.}\label{69}
\end{center}
\end{figure}
The original roadway was heavy and made with concrete; the weight was reduced in 1986 when a new roadway
was installed, see \cite{perks}. Nowadays, in spite of the girder, the bridge can swing more than an amazing 8 meters and
flex about 3 meters under big loads, which explains why the bridge is classified as very flexible. The huge mass
involved and these large distances from equilibrium explain why ${\cal LHL}$ certainly fails. Due to high winds
around 120 kilometers per hour, the Golden Gate Bridge has been closed, without suffering structural damage, only three times: in 1951, 1982, 1983,
always during the month of December. A further interesting phenomenon is the appearance of traveling waves in 1938:
in \cite[Appendix IX]{Tacoma1} (see also \cite{mck4}), the chief engineer of the Golden Gate Bridge writes
\begin{center}
\begin{minipage}{162mm}
{\em ... I observed that the suspended structure of the bridge was undulating vertically in a wavelike motion of considerable amplitude ...}
\end{minipage}
\end{center}
see also the related detailed description in \cite[Section 1]{mck4}. Hence, one should also expect traveling
waves in bridges, see the sketched representation in the first picture in Figure \ref{nuove}.
\begin{figure}[ht]
\begin{center}
{\includegraphics[height=32mm, width=72mm]{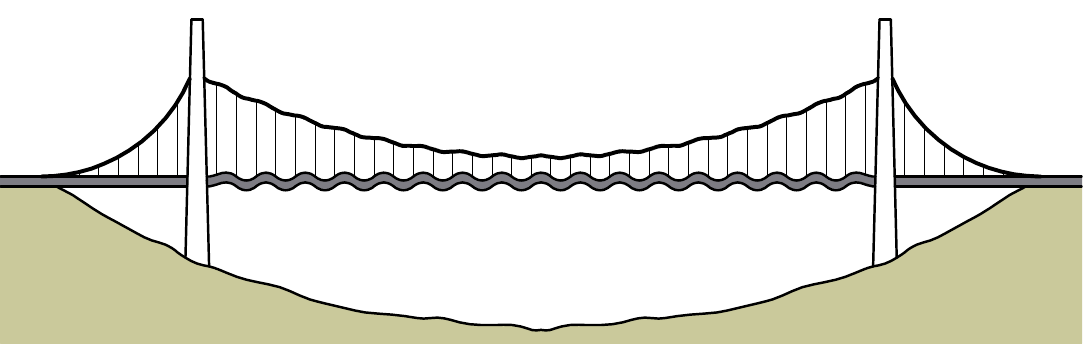}}\qquad\qquad{\includegraphics[height=32mm, width=72mm]{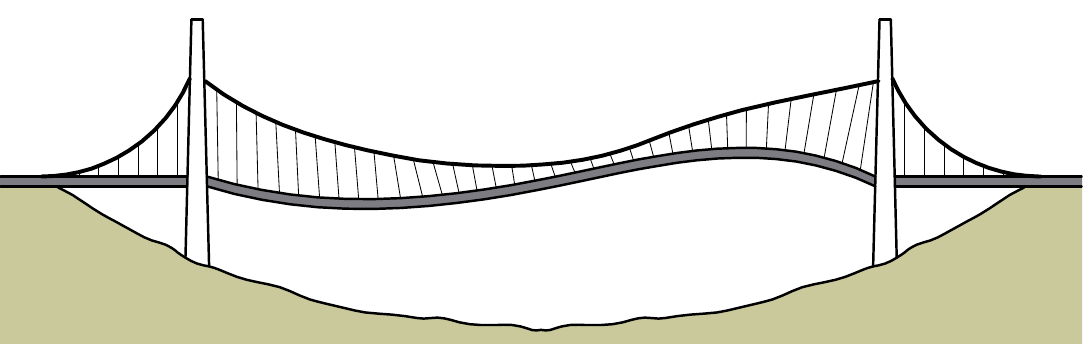}}
\caption{Traveling waves and torsional motion in bridges without girder.}\label{nuove}
\end{center}
\end{figure}
All this may occur also in apparently stiff structures. And in presence of extremely flexible structures, these traveling waves can generate
further dangerous phenomena such as torsional oscillations, see the second picture in Figure \ref{nuove}.\par
When comparing the structure of the Golden Gate Bridge with the one of the original Tacoma Narrows Bridge, one immediately sees a main macroscopic
difference: the thick girder sustaining the bridge, compare Figures \ref{69} and \ref{tacoma12}. The girder gives more
stiffness to the bridge; this is certainly the main reason why in the Golden Gate Bridge no torsional oscillation ever appeared. A further reason is
that larger widths of the roadway seem to prevent torsional oscillations, see \eq{speedflutter} below; from \cite[p.186]{rocard} we quote
\begin{center}
\begin{minipage}{162mm}
{\em ... a bridge twice as wide will have exactly double the critical speed wind.}
\end{minipage}
\end{center}

The Tacoma Bridge was rebuilt with a thick girder acting as a strong stiffening structure, see \cite{tac4}: as mentioned by Scanlan \cite[p.840]{scanlan},
\begin{center}
\begin{minipage}{162mm}
{\em the original bridge was torsionally weak, while the replacement was torsionally stiff.}
\end{minipage}
\end{center}
The replacement of the original bridge opened in 1950, see \cite{tac4} for some remarks on the project, and still stands today as the
westbound lanes of the present-day twin bridge complex, the eastbound lanes opened in 2007.
Figure \ref{tacoma12} - picture by Michael Goff, Oregon Department of Transportation, USA - shows the striking difference between the original
Tacoma Bridge collapsed in 1940 and the twin bridges as they are today.
\begin{figure}[ht]
\begin{center}
{\includegraphics[height=39mm, width=78mm]{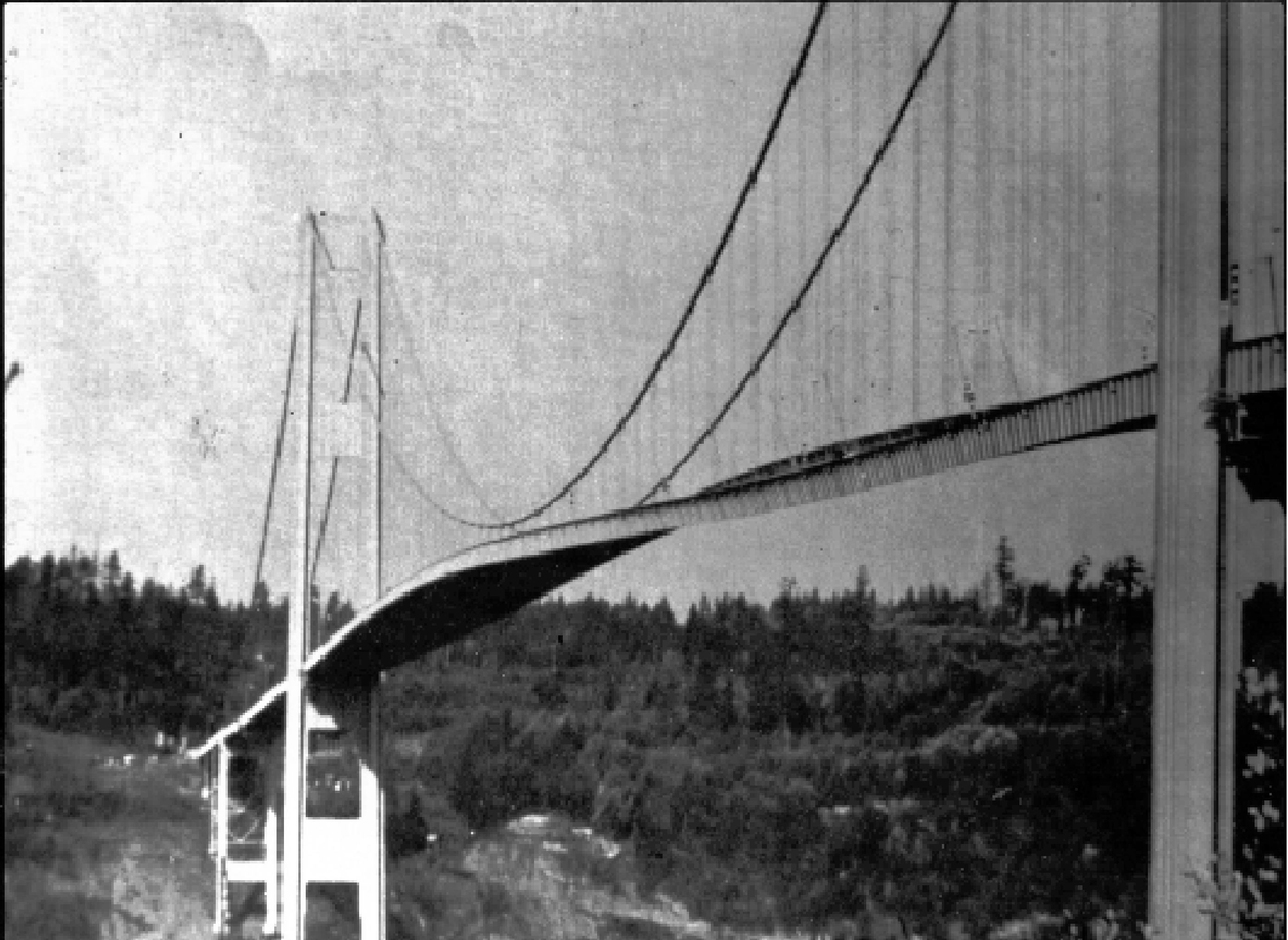}}\quad
{\includegraphics[height=39mm, width=78mm]{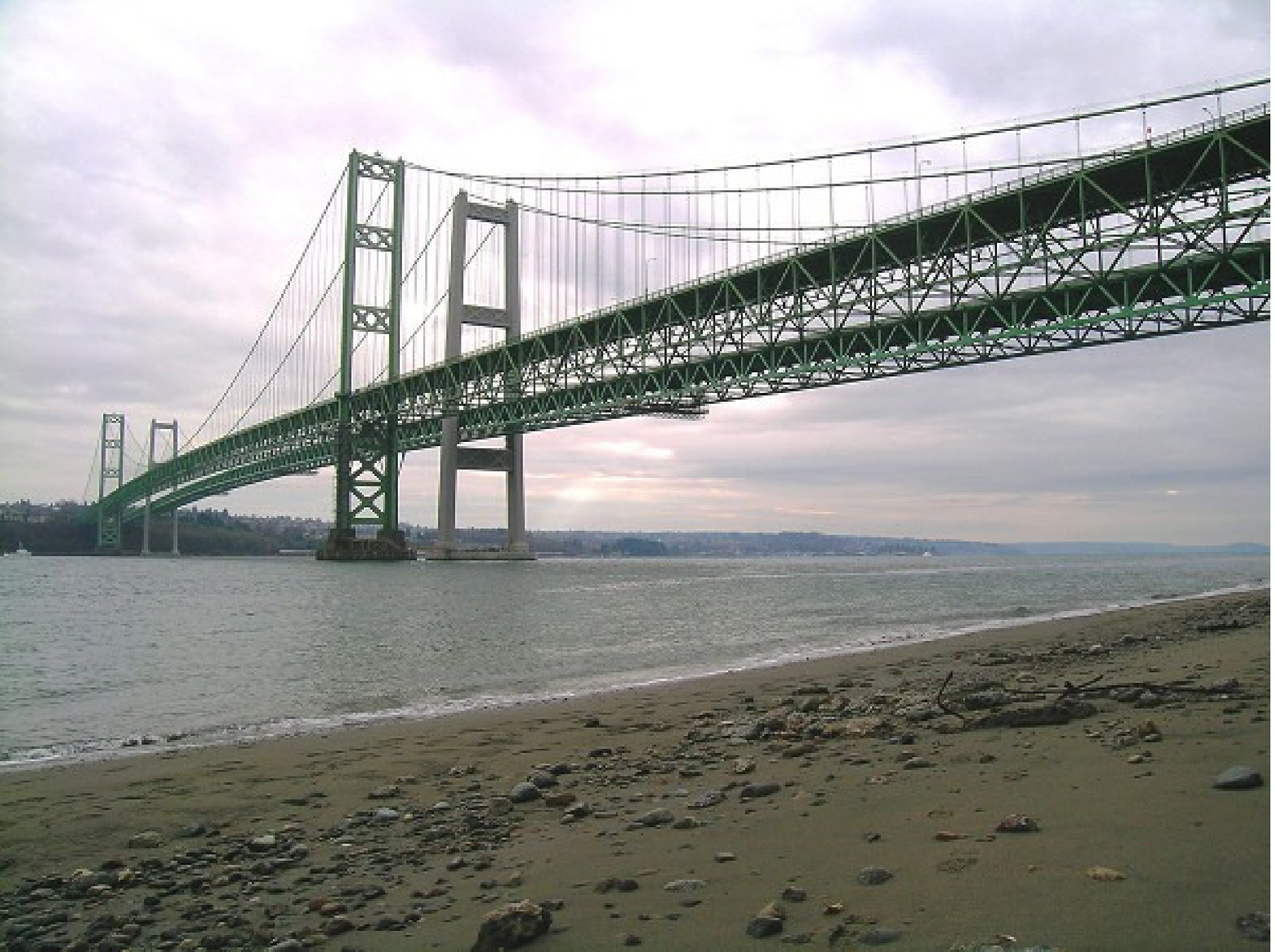}}
\caption{The collapsed Tacoma Bridge and the current twins Tacoma Bridges.}\label{tacoma12}
\end{center}
\end{figure}

Let us go back to the original Tacoma Bridge: even if it was extremely flexible, it is not clear why torsional oscillations appeared.
According to Scanlan \cite[p.841]{scanlan},
\begin{center}
\begin{minipage}{162mm}
{\em ... some of the writings of von K\'arm\'an leave a trail of confusion on this point. ... it can clearly be shown that the rhythm of the failure
(torsion) mode has nothing to do with the natural rhythm of shed vortices following the K\'arm\'an vortex street pattern. ... Others have added to the confusion.
A recent mathematics text, for example, seeking an application for a developed theory of parametric resonance, attempts to explain the Tacoma
Narrows failure through this phenomenon.}
\end{minipage}
\end{center}
Hence, Scanlan discards the possibility of the appearance of von K\'arm\'an vortices and raises doubts on the appearance of resonance which, indeed, is by now
also discarded. Of course, it is reasonable to expect resonance in presence of a single-mode solicitation, such as for the Broughton Bridge. But for the
Tacoma Bridge, Lazer-McKenna \cite[Section 1]{mck1} raise the question
\begin{center}
\begin{minipage}{162mm}
{\em ... the phenomenon of linear resonance is very precise. Could it really be that such precise conditions existed in the middle of the Tacoma Narrows,
in an extremely powerful storm?}
\end{minipage}
\end{center}

So, no plausible explanation is available nowadays. In a letter \cite{farq}, Prof.\ Farquharson claimed that
\begin{center}
\begin{minipage}{162mm}
{\em ... a violent change in the motion was noted. This change appeared to take place without any intermediate stages and with such extreme violence ...
The motion, which a moment before had involved nine or ten waves, had shifted to two.}
\end{minipage}
\end{center}
All this happened under not extremely strong winds, about 80km/h, and under a relatively high frequency of oscillation, about 36cpm, see \cite[p.23]{tac1}.
See \cite[Section 2.3]{mckmonth} for more details and for the conclusion that
\begin{center}
\begin{minipage}{162mm}
{\em there is no consensus on what caused the sudden change to torsional motion.}
\end{minipage}
\end{center}
This is confirmed by the following ambiguous comments taken from \cite[Appendix D]{bleich}:
\begin{center}
\begin{minipage}{162mm}
{\em If vertical and torsional oscillations occur, they must be caused by vertical components of wind forces or by some structural action which derives vertical reactions from
a horizontally acting wind.}
\end{minipage}
\end{center}
This part is continued in \cite{bleich} by stating that there exist references to both alternatives and that
\begin{center}
\begin{minipage}{162mm}
{\em A few instrumental measurements have been made ... which showed the wind varying up to 8 degrees from the horizontal. Such variation from the
horizontal is not the only, and perhaps not the principal source of vertical wind force on a structure.}
\end{minipage}
\end{center}

Besides the lack of consensus on the causes of the switch between vertical and torsional oscillations, all the above comments highlight a strong
instability of the oscillation motion as if, after reaching some critical energy threshold, an impulse (a Dirac delta) generated a new unexpected
motion. Refer to Section \ref{conclusions} for our own interpretation of this phenomenon which is described in \cite{Tacoma1,bleich} (see also
\cite[pp.50-51]{wake}) as:
\begin{center}
\begin{minipage}{162mm}
{\em large vertical oscillations can rapidly change, almost instantaneously, to a torsional oscillation.}
\end{minipage}
\end{center}
We do not completely agree with this description since a careful look at \cite{tacoma} shows that vertical oscillations continue also after
the appearance of torsional oscillations; in the video, one sees that at the beginning of the bridge the street-lamps oscillate in opposition of phase
when compared with the street-lamps at the end of the bridge. So, the phenomenon which occurs may be better described as follows:
\begin{center}
\begin{minipage}{162mm}
{\bf large vertical oscillations can rapidly create, almost instantaneously, additional torsional oscillations.}
\end{minipage}
\end{center}
Roughly speaking, we believe that part of the energy responsible of vertical oscillations switches to another energy which generates torsional oscillations;
the switch occurs without intermediate stages as if an impulse was responsible of it. Our own explanation to this fact is that
\begin{center}
\begin{minipage}{162mm}
{\bf since vertical oscillations cannot be continued too far downwards below the equilibrium position due to the hangers,
when the bridge reaches some limit horizontal position with large kinetic energy, part of the energy transforms
into elastic energy and generates a crossing wave, namely a torsional oscillation.}
\end{minipage}
\end{center}
We make this explanation more precise in Section \ref{energybalance}, after some further observations. In order to explain the ``switch of
oscillations'' several mathematical models were suggested in literature. In next section we survey some of these models which are quite
different from each other although they have some common features.\par
The Deer Isle Bridge, see Figure \ref{deer},
\begin{figure}[ht]
\begin{center}
{\includegraphics[height=39mm, width=78mm]{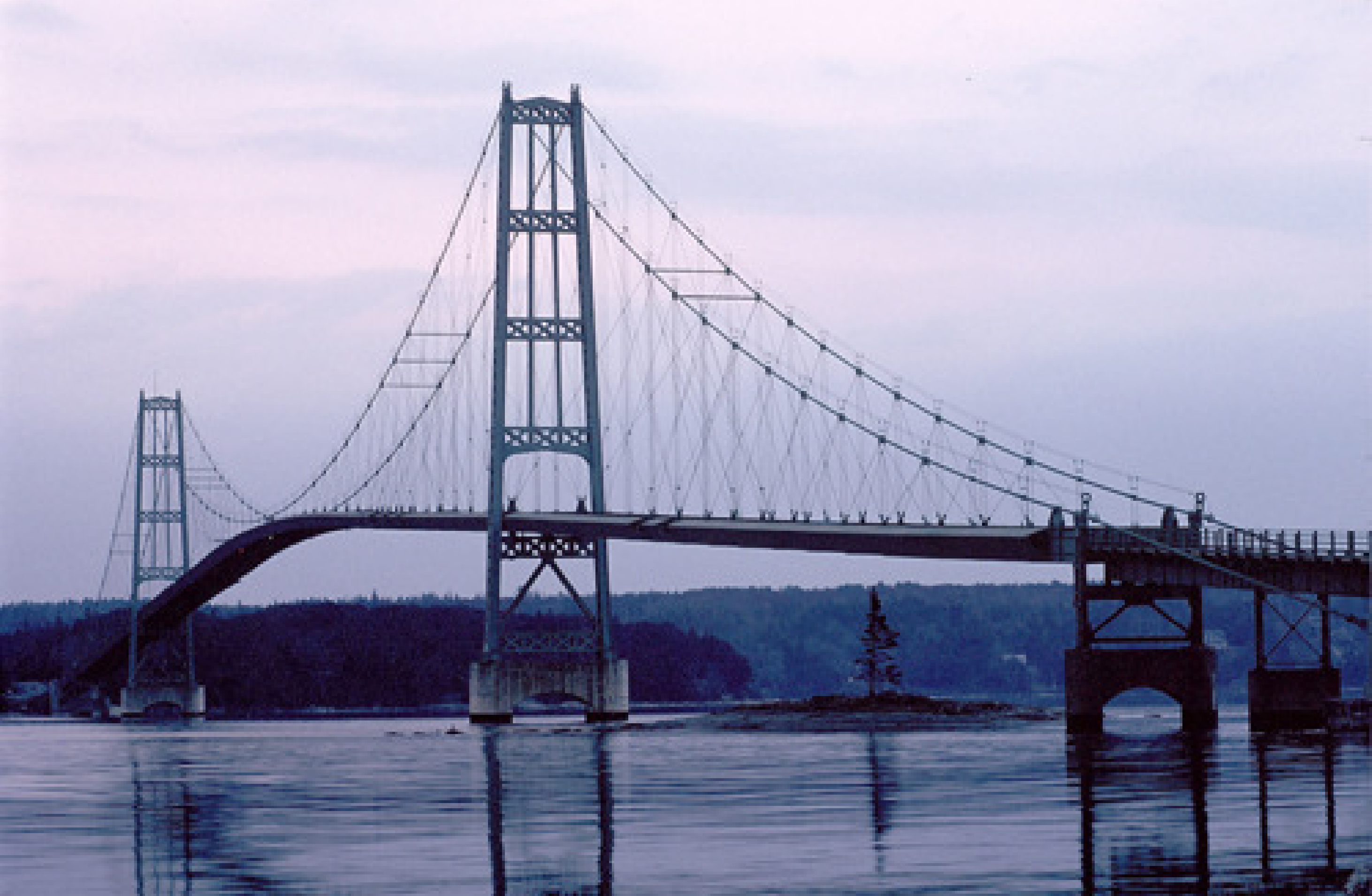}}\quad {\includegraphics[height=39mm, width=78mm]{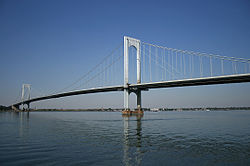}}
\caption{The Deer Isle Bridge (left) and the Bronx-Whitestone Bridge (right).}\label{deer}
\end{center}
\end{figure}
is a suspension bridge in the state of Maine (USA) which encountered wind stability problems similar to those of the original Tacoma Bridge.
Before the bridge was finished, in 1939, the wind induced motion in the relatively lightweight roadway. Diagonal stays running from the sustaining cables to
the stiffening girders on both towers were added to stabilize the bridge. Nevertheless, the oscillations of the roadway during some windstorms in 1942 caused
extensive damage and destroyed some of the stays. At that time everybody had the collapse of the Tacoma Bridge in mind, so that stronger and more extensive
longitudinal and transverse diagonal stays were added. In her report \cite{moran}, Barbara Moran wrote
\begin{center}
\begin{minipage}{162mm}
{\em The Deer Isle Bridge was built at the same time as the Tacoma Narrows, and with virtually the same design. One difference: it still stands.}
\end{minipage}
\end{center}
This shows strong instability: even if two bridges are considered similar they can react differently to external solicitations. Of course, much depends
on what is meant by ``virtually similar''...\par
The Bronx-Whitestone Bridge displayed in Figure \ref{deer}, was built in New York in 1939 and has shown an intermitted tendency to mild vertical motion from the time
the floor system was installed. The reported motions have never been very large, but were noticeable to the traveling public. Several successive steps were
taken to stabilise the structure, see \cite{anon}. Midspan diagonal stays and friction dampers at the towers were first installed; these were later
supplemented by diagonal stayropes from the tower tops to the roadway level. However, even these devices were not entirely adequate and in 1946 the
roadway was stiffened by the addition of truss members mounted above the original plate girders, the latter becoming the lower chords of the trusses
\cite{ammann,pavlo}. This is a typical example of bridge built without considering all the possible external effects, subsequently damped by means of
several unnatural additional components. Our own criticism is that
\begin{center}
\begin{minipage}{162mm}
{\bf instead of just solving the problem, one should understand the problem.}
\end{minipage}
\end{center}
And precisely in order to understand the problem, we described above those events which displayed the pure elastic behavior of bridges.
These were mostly suspension bridges without girders and were free to oscillate. This is a good reason why the Tacoma collapse should be
further studied for deeper knowledge: it displays the pure motion without stiffening constraints which hide the elastic features of bridges.\par
The Tacoma Bridge collapse is just the most celebrated and dramatic evidence of oscillating bridge but bridges oscillations are still
not well understood nowadays. On May 2010, the Russian authorities closed the Volgograd Bridge to all motor traffic due to its
strong vertical oscillations (traveling waves) caused by windy conditions, see \cite{volgograd} for the BBC report and video. Once more, these
oscillations may appear surprising since the Volgograd Bridge is a concrete girder bridge and its stiffness should prevent oscillations. However, it seems
that strong water currents in the Volga river loosened one of the bridge's vertical supports so that the stiffening effect due to the concrete
support was lost and the behavior became more similar to that of a suspension bridge. The bridge remained closed while it was inspected for damage.
As soon as the original effect was restored the bridge reopened for public access. In Figure \ref{volgabridge} the reader finds
pictures of the bridge and of the damped sustaining support.
\begin{figure}[ht]
\begin{center}
{\includegraphics[height=39mm, width=78mm]{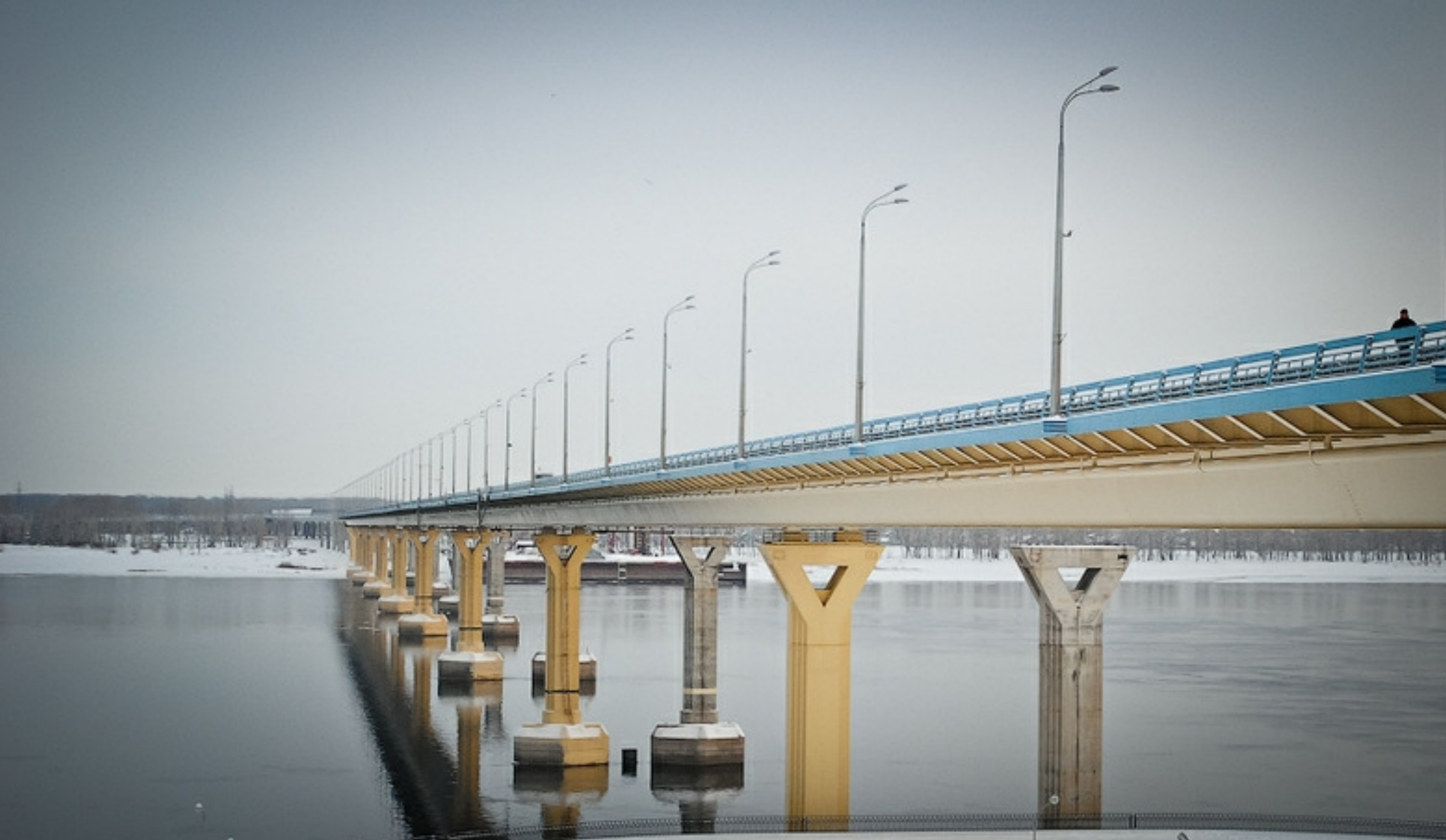}}\quad{\includegraphics[height=39mm, width=78mm]{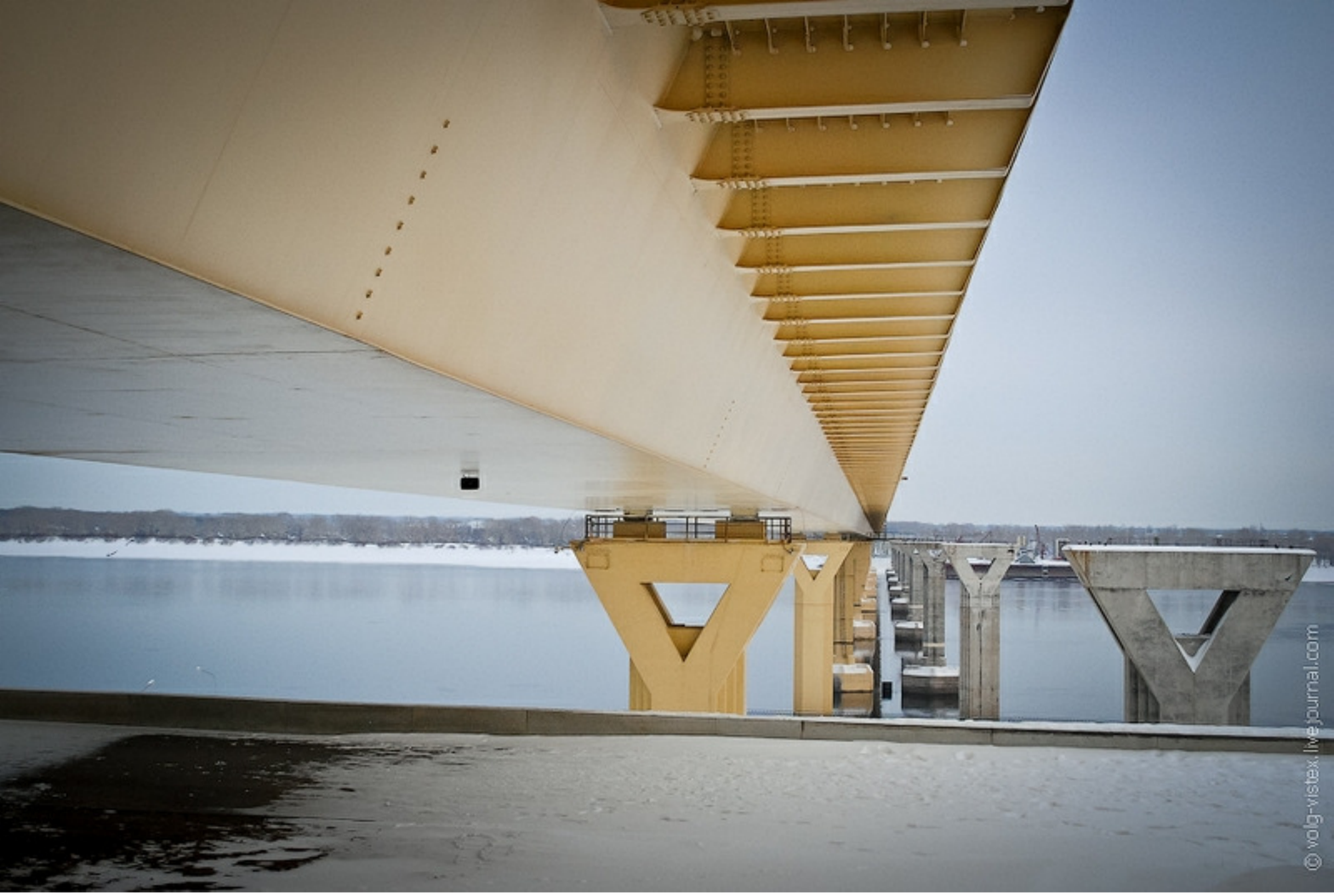}}
\caption{The Volgograd Bridge.}\label{volgabridge}
\end{center}
\end{figure}
These pictures are taken from \cite{volgobridge}, where one can also find full details on the damping system of the bridge. The Volgograd Bridge well
shows how oscillation induced fatigue of the structural members of bridges is a major factor limiting the life of the bridge. In \cite{kawada} one may find
a mathematical analysis and wind tunnel tests for examining oscillations which occur under ``constant low wind'', rather than under violent windstorms:
\begin{center}
\begin{minipage}{162mm}
{\em Limited oscillation could even cause a collapse of light suspension bridges in a reasonably short time.}
\end{minipage}
\end{center}

As already observed, the wind is not the only possible external source which generates bridges oscillations which also appear in pedestrian bridges where
lateral swaying is the counterpart of torsional oscillation.
In June 2000, the very same day when the London Millennium Bridge opened and the crowd streamed on it, the bridge started to sway from side
to side, see \cite{london}. Many pedestrians fell spontaneously into step with the vibrations, thereby amplifying them. According to Sanderson
\cite{sanderson}, the bridge wobble was due to the way people balanced themselves, rather than the timing of their steps. Therefore, the pedestrians
acted as negative dampers, adding energy to the bridge's natural sway. Macdonald \cite[p.1056]{macdonald} explains this phenomenon by writing
\begin{center}
\begin{minipage}{162mm}
{\em ... above a certain critical number of pedestrians, this
negative damping overcomes the positive structural damping, causing the onset of exponentially increasing vibrations.}
\end{minipage}
\end{center}
Although we have some doubts about the real meaning of ``exponentially increasing vibrations'' we have no doubts that this description corresponds to a superlinear
behavior which has also been observed in several further pedestrian bridges, see \cite{franck} and \cite{zivanovic} from which we quote
\begin{center}
\begin{minipage}{162mm}
{\em ... damping usually increases with increasing vibration magnitude due to engagement of additional damping mechanisms.}
\end{minipage}
\end{center}
The Millennium Bridge was made secure by adding some stiffening trusses below the girder, see Figure \ref{LMB} (Photo $\copyright$ Peter Visontay).
\begin{figure}[ht]
\begin{center}
{\includegraphics[height=39mm, width=78mm]{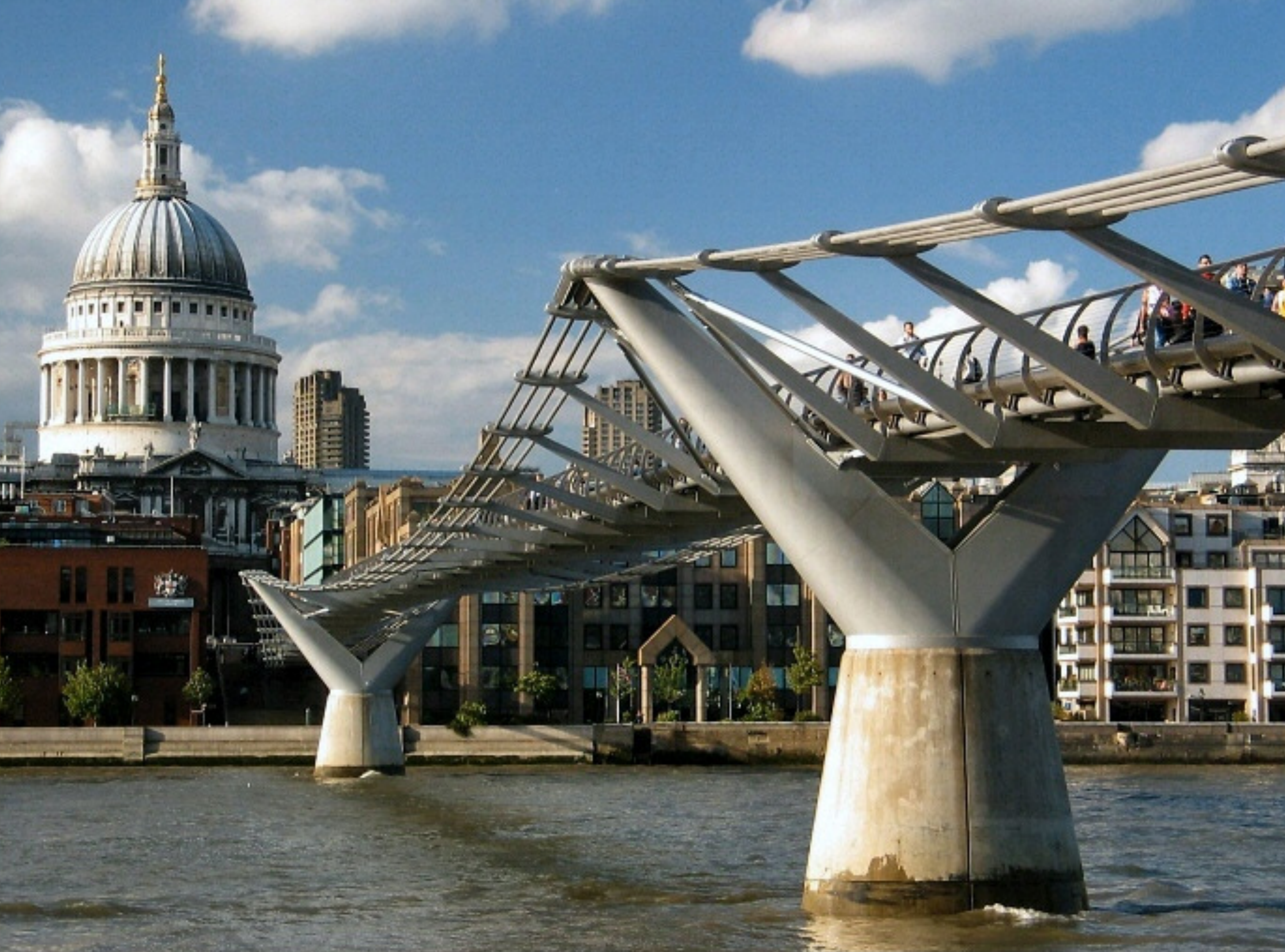}}
\caption{The London Millennium Bridge.}\label{LMB}
\end{center}
\end{figure}
The mathematical explanation of this solution
is that trusses lessen swaying and force the bridge to remain closer to its equilibrium position, that is, closer to a linear behavior
as described by ${\cal LHL}$. Although trusses delay the appearance of the superlinear behavior, they do not solve completely the problem
as one may wonder what would happen if 10.000 elephants would simultaneously walk through the Millennium Bridge... In this respect, let us quote
from \cite[p.13]{tac1} a comment on suspension bridges strengthened by stiffening girders:
\begin{center}
\begin{minipage}{162mm}
{\em That significant motions have not been recorded on most of these bridges is conceivably due to the fact that they have never been subjected
to optimum winds for a sufficient period of time.}
\end{minipage}
\end{center}

Another pedestrian bridge, the Assago Bridge in Milan (310m long), had a similar problem. In February 2011, just after a concert the
publics crossed the bridge and, suddenly, swaying became so violent that people could hardly stand, see \cite{fazzo} and \cite{assago}.
Even worse was the subsequent panic effect when the crowd started running in order to escape from a possible
collapse; this amplified swaying but, quite luckily, nobody was injured. In this case, the project did not take
into account that a large number of people would go through the bridge just after the events; when swaying started there were about 1.200
pedestrians on the footbridge. This problem was solved by adding positive dampers, see \cite{stella}.\par
According to \cite{bridgefailure}, around 400 recorded bridges failed for several different reasons and the ones who failed after year 2000 are more
than 70. Probably, some years after publication of this paper, these numbers will have increased considerably... The database \cite{bridgefailure}
consists mainly of brief descriptions and statistics for each bridge failure: location, number of fatalities/injuries, etc.
rather than in-depth analysis of the cause of the failure for which we refer to the nice book by Akesson \cite{akesson}.\par
As we have seen, the reasons of failures are of different kinds.
Firstly, strong and/or continued winds: these may cause wide vertical oscillations which may switch to different kinds of oscillations. Especially
for suspension bridges the latter phenomenon appears quite evident, due to the many elastic components (cables, hangers, towers, etc.) which appear in it.
A second cause are traffic loads, such as some precise resonance phenomenon, or some unpredictable synchronised behavior, or some unexpected
huge load; these problems are quite common in many different kinds of bridges. Finally, a third cause are mistakes in the project; these are both
theoretical, for instance assuming ${\cal LHL}$, and practical, such as wrong assumptions on the possible maximum external actions.\par
After describing so many disasters, we suggest a joke which may sound as a provocation. Since many bridges projects did not forecast oscillations
it could be more safe to build old-fashioned rock bridges, such as the Roman aqueduct built in Segovia (Spain) during the first century and still
in perfect shape and in use. Of course,
we are not suggesting here to replace the Golden Gate Bridge with a Roman-style bridge! But we do suggest to plan bridges by taking into account
all possible kinds of solicitations. Moreover, we suggest not to hide unsolved problems with some unnatural solutions such as stiff
and heavy girders or more extensive longitudinal and transverse diagonal stays, see Section \ref{howplan} for more suggestions.\par\medskip
Throughout this section we listed a number of historical events about bridges. They taught us the following facts:\par
1. Self-excited oscillations appear in bridges. Often this is somehow unexpected since the project does not take into account
several external strong and/or prolonged effects. And even if expected, oscillations can be much wider than estimated.\par
2. Oscillations can be extenuated by stiffening the structure or by adding positive (and heavy, and expensive) dampers to the structure. However, none
of these solutions can completely prevent oscillations, especially in presence of highly unfavorable events such as strong
and prolonged winds, not necessarily hurricanes, or heavy and synchronised traffic loads. Due to the unnatural stiffness of
the structure, trusses and dampers may cause cracks, see \cite{crack} and references therein; but we leave this problem to engineers... see \cite{kawada2}.\par
3. The oscillations are amplified by an observable superlinear effect. More the bridge is far from its equilibrium
position, more the impact of external forces is relevant. It is by now well understood that suspension bridges behave
nonlinearly, see e.g.\ \cite{brown,lacarbonara2}.\par
4. In extremely flexible bridges, such as the Tacoma Bridge which had no stiffening truss, vertical oscillations can partially switch to torsional oscillations
and even to more complicated oscillations, see the pictures in Figure \ref{kable} which are taken from \cite[p.143]{cable}
\begin{figure}[ht]
\begin{center}
{\includegraphics[height=26mm, width=52mm]{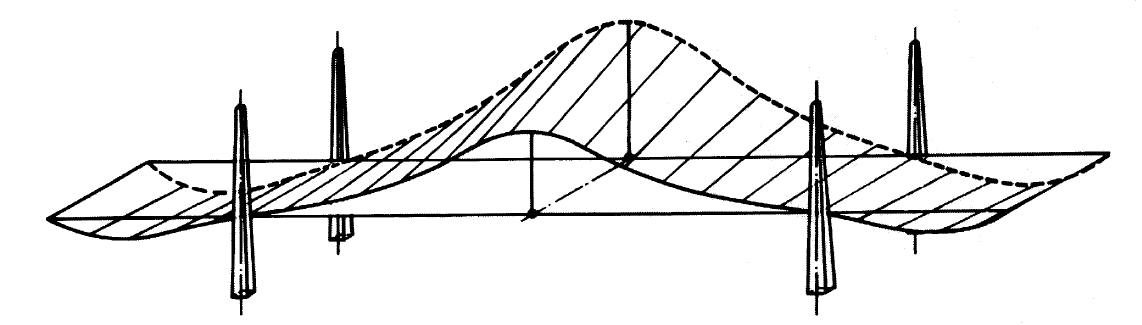}}\quad{\includegraphics[height=26mm, width=52mm]{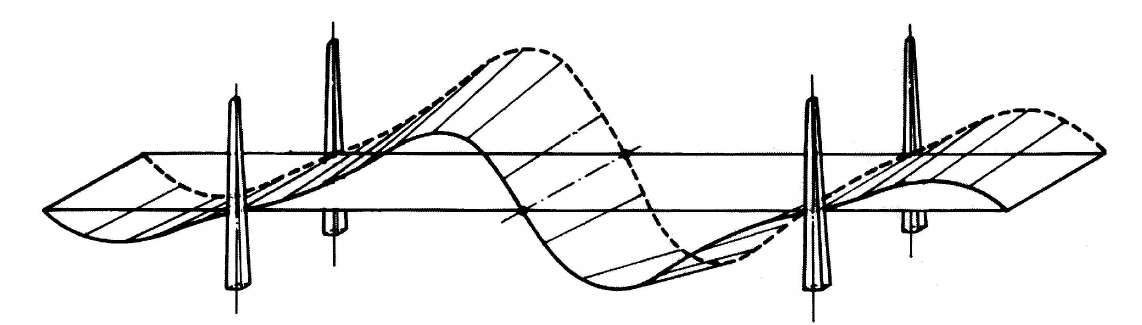}}
\quad{\includegraphics[height=26mm, width=52mm]{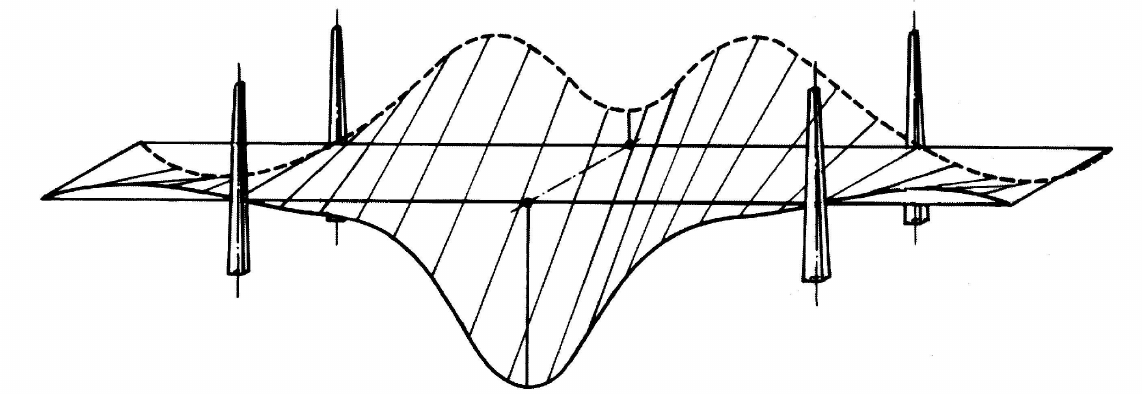}}\\
{\includegraphics[height=26mm, width=52mm]{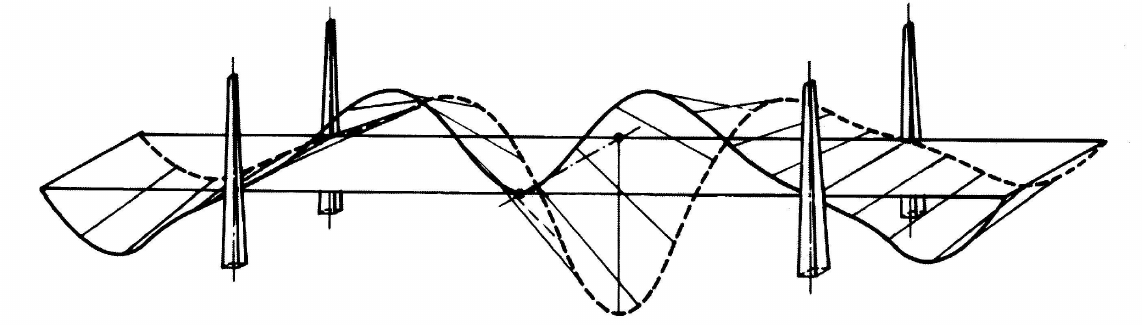}}\quad{\includegraphics[height=26mm, width=52mm]{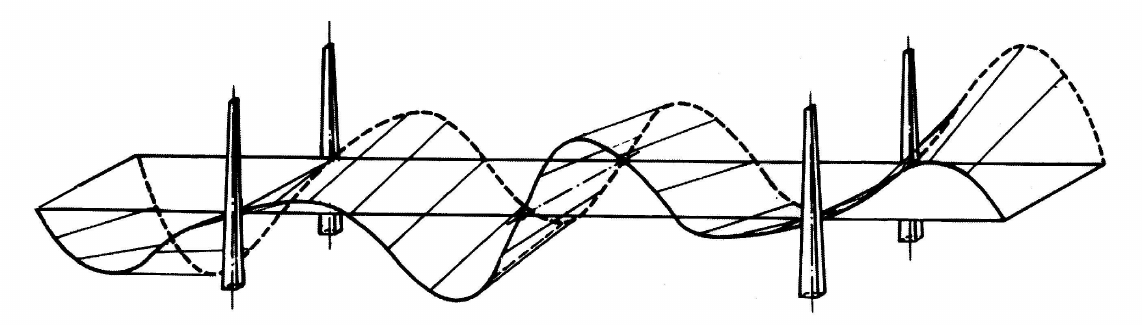}}
\quad{\includegraphics[height=26mm, width=52mm]{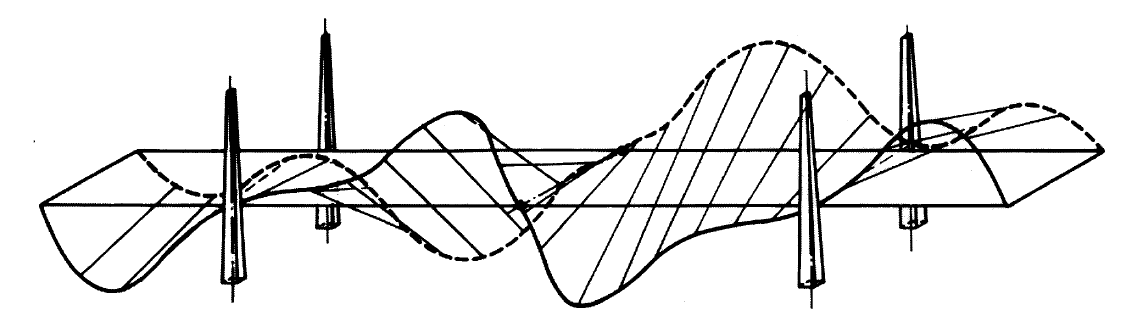}}
\caption{Combined oscillations of a bridge roadway.}\label{kable}
\end{center}
\end{figure}
and see also \cite[pp.94-95]{rocard} and Figure \ref{ghaffer} below.
This occurs when vertical oscillations become too large and, in such case, vertical and torsional oscillations coexist. Up to nowadays,
there is no convincing explanation why the switch occurs.

\section{How to model bridges}\label{howto}

The amazing number of failures described in the previous section shows that the existing theories and models are not adequate to describe the statics and the dynamics
of oscillating bridges. In this section we survey different points of view, different models, and we underline their main weaknesses. We also suggest how to modify them
in order to fulfill the requirements of (GP).

\subsection{A quick overview on elasticity: from linear to semilinear models}\label{elasticity}

A quite natural way to describe the bridge roadway is to view it as a thin rectangular plate. This is also the opinion of Rocard \cite[p.150]{rocard}:
\begin{center}
\begin{minipage}{162mm}
{\em The plate as a model is perfectly correct and corresponds mechanically to a vibrating suspension bridge...}
\end{minipage}
\end{center}

In this case, a commonly adopted theory is the linear one by Kirchhoff-Love \cite{kirchhoff,love}, see also \cite[Section 1.1.2]{gazgruswe}, which we briefly recall.
The bending energy of a plate involves curvatures of the surface. Let $\kappa_1$, $\kappa_2$ denote the principal curvatures of the graph
of a smooth function $u$ representing the deformation of the plate, then a simple model for the bending energy of the deformed plate $\Omega$ is
\neweq{curva}
\mathbb{E}(u)=\int_\Omega\left(\frac{\kappa_1^2}{2}+\frac{\kappa_2^2}{2}+\sigma\kappa_1\kappa_2\right)\, dx_1dx_2
\endeq
where $\sigma$ denotes the Poisson ratio defined by $\sigma=\frac{\lambda}{2\left(\lambda +\mu \right)}$ with the so-called Lam\'e constants $\lambda,\mu $ that depend
on the material. For physical reasons it holds that $\mu >0$ and usually $\lambda \geq 0$ so that $0\le\sigma<\frac{1}{2}$.
In the linear theory of elastic plates, for small deformations $u$ the terms in \eq{curva} are considered to be purely quadratic with respect to the second order
derivatives of $u$. More precisely, for small deformations $u$, one has
$$(\kappa_1+\kappa_2)^2\approx(\Delta u)^2\ ,\quad\kappa_1\kappa_2\approx\det(D^2u)=(u_{x_1x_1}u_{x_2x_2}-u_{x_1x_2}^{2})\ ,$$
and therefore
$$\frac{\kappa_1^2}{2}+\frac{\kappa_2^2}{2}+\sigma\kappa_1\kappa_2\approx\frac{1}{2}(\Delta u)^2+(\sigma-1)\det(D^2u).$$
Then \eq{curva} yields
\neweq{energy-gs}
\mathbb{E}(u)=\int_{\Omega }\left(\frac{1}{2}\left( \Delta u\right) ^{2}+(\sigma-1)\det(D^2u)\right) \, dx_1dx_2\, .
\endeq
Note that for $-1<\sigma<1$ the functional $\mathbb{E}$ is coercive and convex.
This modern variational formulation appears in \cite{Friedrichs}, while a discussion for a boundary value problem for a thin elastic plate in a somehow old
fashioned notation is made by Kirchhoff \cite{kirchhoff}. And precisely the choice of the boundary conditions is quite delicate since it depends on
the physical model considered.\par
Destuynder-Salaun \cite[Section I.2]{destuyndersalaun} describe this modeling by
\begin{center}
\begin{minipage}{162mm}
{\em ... Kirchhoff and Love have suggested to assimilate the plate to a collection of small pieces, each one being articulated with respect
to the other and having a rigid-body behavior. It looks like these articulated wooden snakes that children have as toys. Hence the transverse shear
strain remains zero, while the planar deformation is due to the articulation between small blocks. But this simplified description of a plate movement
can be acceptable only if the components of the stress field can be considered to be negligible.}
\end{minipage}
\end{center}

The above comment says that ${\cal LHL}$ should not be adopted if the components of the stress field are not negligible.
An attempt to deal with large deflections for thin plates is made by Mansfield \cite[Chapters 8-9]{mansfield}. He first considers
approximate methods, then with three classes of asymptotic plate theories: membrane theory, tension field theory, inextensional theory.
Roughly speaking, the three theories may be adopted according to the ratio between the thickness of the plate and the typical planar dimension:
for the first two theories the ratio should be less than $10^{-3}$, whereas for the third theory it should be less than $10^{-2}$.
Since a roadway has a length of the order of 1km, the width of the order of 10m, even for the less stringent inextensional theory the thickness
of the roadway should be less than 10cm which, of course, appears unreasonable. Once more, this means that ${\cal LHL}$ should not be adopted
in bridges. In any case, Mansfield \cite[p.183]{mansfield} writes
\begin{center}
\begin{minipage}{162mm}
{\em The exact large-deflection analysis of plates generally presents considerable difficulties...}
\end{minipage}
\end{center}

Destuynder-Salaun \cite[Section I.2]{destuyndersalaun} also revisit an alternative model due to Naghdi \cite{naghdi} by using a mixed variational
formulation. They refer to \cite{mindlin,reissner1,reissner2} for further details and modifications, and conclude by saying that none between the
Kirchhoff-Love model or one of these alternative models is always better than the others. Moreover, also the definition of the transverse
shear energy is not universally accepted: from \cite[p.149]{destuyndersalaun}, we quote
\begin{center}
\begin{minipage}{162mm}
{\em ... this discussion has been at the origin of a very large number of papers from both mathematicians and engineers. But to our best knowledge,
a convincing justification concerning which one of the two expressions is the more suitable for numerical purpose, has never been formulated in a
convincing manner. This question is nevertheless a fundamental one ...}
\end{minipage}
\end{center}

It is clear that a crucial role is played by the word ``thin''. What is a thin plate? Which width is it allowed to have? If we assume that the
width is zero, a quite unrealistic assumption for bridges, a celebrated two-dimensional equation was suggested by von K\'arm\'an \cite{karman}.
This equation has been widely, and satisfactorily, studied from several mathematical points of view such as existence, regularity, eigenvalue problems,
semilinear versions, see e.g.\ \cite{gazgruswe} for a survey of results. On the other hand, quite often several doubts have been raised on their
physical soundness. For instance, Truesdell \cite[pp.601-602]{truesdell} writes
\begin{center}
\begin{minipage}{162mm}
{\em Being unable to explain just why the von K\'arm\'an theory has always made me feel a little nauseated as well as very slow and stupid, I asked
an expert, Mr. Antman, what was wrong with it. I can do no better than paraphrase what he told me: it relies upon\par
1) ``approximate geometry'', the validity of which is assessable only in terms of some other theory.\par
2) assumptions about the way the stress varies over a cross-section, assumptions that could be justified only in terms of some other theory.\par
3) commitment to some specific linear constitutive relation - linear, that is, in some special measure of strain, while such approximate linearity
should be outcome, not the basis, of a theory.\par
4) neglect of some components of strain - again, something that should be proved mathematically from an overriding, self-consistent theory.\par
5) an apparent confusion of the referential and spatial descriptions - a confusion that is easily justified for the classical linearised
elasticity but here is carried over unquestioned, in contrast with all recent studies of the elasticity of finite deformations.}
\end{minipage}
\end{center}

Truesdell then concludes with a quite eloquent comment:
\begin{center}
\begin{minipage}{162mm}
{\em These objections do not prove that anything is wrong with von K\'arm\'an strange theory. They merely suggest that it would be difficult to prove
that there is anything right about it.}
\end{minipage}
\end{center}
Let us invite the interested reader to have a careful look at the paper by Truesdell \cite{truesdell}; it contains several criticisms
exposed in a highly ironic and exhilarating fashion and, hence, very effective.\par
Classical books for elasticity theory are due to Love \cite{love}, Timoshenko \cite{timoshenko}, Ciarlet \cite{ciarletbook}, Villaggio \cite{villaggio},
see also \cite{nadai,naghdi,timoshenkoplate} for the theory of plates.
Let us also point out a celebrated work by Ball \cite{ball} who was the first analyst to approach the real 3D boundary value problems for
nonlinear elasticity. Further nice attempts to tackle nonlinear elasticity in particular situations were done by Antman \cite{antman1,antman2} who, however,
appears quite skeptic on the possibility to have a general theory:
\begin{center}
\begin{minipage}{162mm}
{\em ... general three-dimensional nonlinear theories have so far proved to be mathematically intractable.}
\end{minipage}
\end{center}

Summarising, what we have seen suggests to conclude this short review about plate models by claiming that classical modeling
of thin plates should be carefully revisited. This suggestion is absolutely not new. In this respect, let us quote a couple of sentences written
by Gurtin \cite{gurtin} about nonlinear elasticity:
\begin{center}
\begin{minipage}{162mm}
{\em Our discussion demonstrates why this theory is far more difficult than most nonlinear theories of mathematical physics. It is hoped that these
notes will convince analysts that nonlinear elasticity is a fertile field in which to work.}
\end{minipage}
\end{center}

Since the previously described Kirchhoff-Love model implicitly assumes ${\cal LHL}$, and since quasilinear equations appear too complicated in order to
give useful information, we intend to add some nonlinearity only in the source $f$ in order to have a semilinear equation, something which appears to be a
good compromise between too poor linear models and too complicated quasilinear models. This compromise is quite common in elasticity,
see e.g.\ \cite[p.322]{ciarletbook} which describes the method of asymptotic expansions for the thickness $\eps$ of a plate as a ``partial linearisation''
\begin{center}
\begin{minipage}{162mm}
{\em ... in that a system of quasilinear partial differential equations, i.e., with nonlinearities in the higher order terms, is replaced as $\eps\to0$
by a system of semilinear partial differential equations, i.e., with nonlinearities only in the lower order terms.}
\end{minipage}
\end{center}

In Section \ref{newmodel}, we suggest a new 2D mathematical model described by a semilinear fourth order wave equation.
Before doing this, in next section we survey some existing models and we suggest some possible variants based on the observations listed in Section \ref{story}.

\subsection{Equations modeling suspension bridges}\label{models}

Although it is oversimplified in several respects, the celebrated report by Navier \cite{navier2} has been for more than one century the only mathematical
treatise of suspension bridges. The second milestone contribution is certainly the monograph by Melan \cite{melan}. After the Tacoma collapse, the engineering
communities felt the necessity to find accurate equations in order to attempt explanations of what had occurred. In this respect, a first source is
certainly the work by Smith-Vincent \cite{tac2} which was written precisely {\em with special reference to the Tacoma Narrows Bridge}.
The bridge is modeled as a one dimensional beam, say the interval $(0,L)$, and in order to obtain an autonomous equation, Smith-Vincent consider
the function $\eta=\eta(x)$ representing the amplitude of the oscillation at the point $x\in(0,L)$. By linearising they obtain a
fourth order linear ODE \cite[(4.2)]{tac2} which can be integrated explicitly. We will not write this equation because we prefer to deal with
the function $v=v(x,t)$ representing the deflection at any point $x\in(0,L)$ and at time $t>0$; roughly speaking, $v(x,t)=\eta(x)\sin(\omega t)$
for some $\omega>0$. In this respect, a slightly better job was done in \cite{bleich} although this book was not very lucky since two of the
authors (McCullogh and Bleich) passed away during its preparation. Equation \cite[(2.7)]{bleich} is precisely \cite[(4.2)]{tac2}; but
\cite[(2.6)]{bleich} considers the deflection $v$ and reads
\neweq{primissima}
m\, v_{tt}+EI\, v_{xxxx}-H_w\, v_{xx}+\frac{w\, h}{H_w}=0\, ,\qquad x\in(0,L)\, ,\ t>0\, ,
\endeq
where $E$ and $I$ are, respectively, the elastic modulus and the moment of inertia of the stiffening girder so that $EI$ is the stiffness of the
girder; moreover, $m$ denotes the mass per unit length, $w=mg$ is the weight which produces a cable stress whose horizontal component is $H_w$,
and $h$ is the increase of $H_w$ as a result of the additional deflection $v$. In particular, this means that $h$ depends on $v$ although
\cite{bleich} does not emphasise this fact and considers $h$ as a constant.\par
An excellent source to derive the equation of vertical oscillations in suspension bridges is \cite[Chapter IV]{rocard} where all the details
are perfectly explained. The author, the French physicist Yves-Andr\'e Rocard (1903-1992), also helped to develop the atomic bomb for France.
Consider again that a long span bridge roadway is a beam of length $L>0$ and that it is oscillating; let $v(x,t)$ denote the vertical
component of the oscillation for $x\in(0,L)$ and $t>0$. The equation derived in \cite[p.132]{rocard} reads
\neweq{flutter}
m\, v_{tt}+EI\, v_{xxxx}-\big(H_w+\gamma v\big)\, v_{xx}+\frac{w\, \gamma}{H_w}v=f(x,t)\, ,\quad x\in(0,L)\, ,\ t>0,
\endeq
where $H_w$, $EI$ and $m$ are as in \eq{primissima}, $\gamma v$ is the variation $h$ of $H_w$ supposed to vary linearly with $v$, and $f$ is
an external forcing term. Note that a nonlinearity appears here in the term $\gamma v v_{xx}$. In fact, \eq{flutter} is closely related to
an equation suggested much earlier by Melan \cite[p.77]{melan} but it has not been subsequently attributed to him.

\begin{problem} {\em Study oscillations and possible blow up in finite time for traveling waves to \eq{flutter} having velocity $c>0$,
$v=v(x,t)=y(x-ct)$ for $x\in\R$ and $t>0$, in the cases where $f\equiv1$ is constant and where $f$ depends superlinearly on $v$.
Putting $\tau=x-ct$ one is led to find solutions to the ODE
$$
EI\, y''''(\tau)-\Big(\gamma y(\tau)+H_w-mc^2\Big)\, y''(\tau)+\frac{w\, \gamma}{H_w}y(\tau)=1\, ,\quad \tau\in\R\, .
$$
By letting $w(\tau)=y(\tau)-\frac{H_w}{w\, \gamma}$ and normalising some constants, we arrive at
\neweq{y4}
w''''(\tau)-\Big(\alpha w(\tau)+\beta\Big)\, w''(\tau)+w(\tau)=0\, ,\quad \tau\in\R\, ,
\endeq
for some $\alpha>0$ and $\beta\in\R$; we expect different behaviors depending on $\alpha$ and $\beta$. It would be interesting to see
if local solutions to \eq{y4} blow up in finite time with wide oscillations. Moreover, one should also consider the more general problem
$$w''''(\tau)-\Big(\alpha w(\tau)+\beta\Big)\, w''(\tau)+f(w(\tau))=0\, ,\quad \tau\in\R\, ,$$
with $f$ being superlinear, for instance $f(s)=s+\eps s^3$ with $\eps>0$ small. Incidentally, we note that such $f$ satisfies
\eq{f} and \eq{fmono}-\eq{f2} below.}\endproof\end{problem}

Let us also mention that Rocard \cite[pp.166-167]{rocard} studies the possibility of simultaneous excitation of different bending and torsional modes
and obtains a coupled system of linear equations of the kind of \eq{flutter}.
With few variants, equations \eq{primissima} and \eq{flutter} seem nowadays to be well-accepted among engineers, see e.g.\ \cite[Section VII.4]{aer};
moreover, quite similar equations are derived to describe related phenomena in
cable-stayed bridges \cite[(1)]{bruno} and in arch bridges traversed by high-speed trains \cite[(14)-(15)]{lacarbonara}.\par
Let $v(x,t)$ and $\theta(x,t)$ denote respectively the vertical and torsional components of the oscillation of the bridge, then the following system
is derived in \cite[(1)-(2)]{como} for the linearised equations of the elastic combined vertical-torsional oscillation motion:
\renewcommand{\arraystretch}{2.8}
\neweq{eqqq}
\left\{\begin{array}{l}
\displaystyle m\, v_{tt}+EI\, v_{xxxx}-H_w\, v_{xx}+\frac{w^2}{H_w^2}\, \frac{EA}{L}\int_0^L v(z,t)\, dz=f(x,t)\\
\displaystyle I_0\, \theta_{tt}+C_1\, \theta_{xxxx}-(C_2+H_w\ell^2)\, \theta_{xx}+\frac{\ell^2w^2}{H_w^2}\, \frac{EA}{L}\int_0^L\theta(z,t)\, dz=g(x,t)\\
x\in(0,L)\, ,\ t>0,
\end{array}\right.
\endeq
\renewcommand{\arraystretch}{1.5}where $m$, $w$, $H_w$ are as in \eq{primissima}, $EI$, $C_1$, $C_2$, $EA$ are respectively the flexural,
warping, torsional, extensional stiffness of the girder, $I_0$ the polar moment of inertia of the girder section,
$2\ell$ the roadway width, $f(x,t)$ and $g(x,t)$ are the lift and the moment for unit girder length of the self-excited forces.
The linearisation here consists in dropping the term $\gamma v v_{xx}$ but a preliminary linearisation was already present in \eq{flutter} in the zero
order term. And the nonlocal linear term $\int_0^L v$, which replaces the zero order term in \eq{flutter},
is obtained by assuming ${\cal LHL}$. The nonlocal term in \eq{eqqq} represents the increment of energy due to the external
wind during a period of time; this will be better explained in Section \ref{energies}.\par
A special mention is deserved by an important paper by Abdel-Ghaffar \cite{abdel} where variational principles are used to obtain the combined
equations of a suspension bridge motion in a fairly general nonlinear form. The effect of coupled vertical-torsional oscillations as well as
cross-distortional of the stiffening structure is clarified by separating them into four different kinds of displacements: the vertical
displacement $v$, the torsional angle $\theta$, the cross section distortional angle $\psi$, the warping displacement $u$, although
$u$ can be expressed in terms of $\theta$ and $\psi$. These displacements are well described in Figure \ref{ghaffer} which is taken from
\cite[Figure 2]{abdel}.
\begin{figure}[ht]
\begin{center}
{\includegraphics[height=47mm, width=81mm]{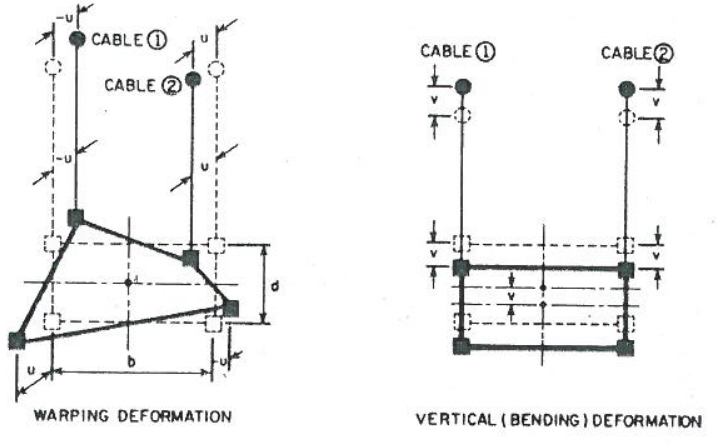}} \ {\includegraphics[height=47mm, width=83mm]{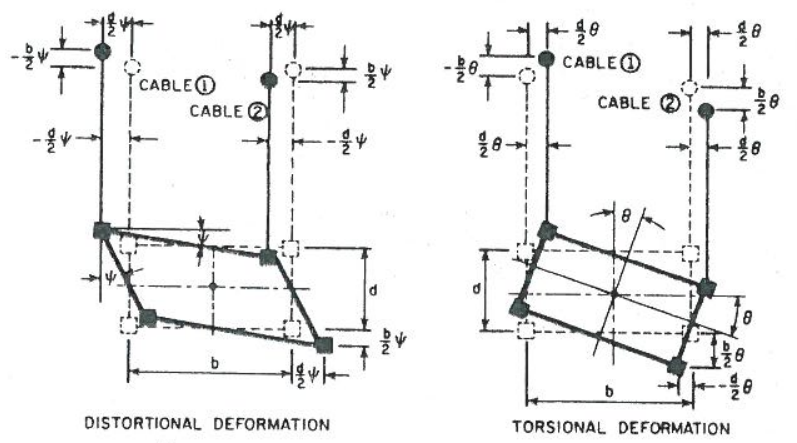}}
\caption{The four different kinds of displacements.}\label{ghaffer}
\end{center}
\end{figure}
A careful analysis of the energies involved is made, reaching up to fifth derivatives in the equations, see \cite[(15)]{abdel}. Higher order
derivatives are then neglected and the following nonlinear system of three PDE's of fourth order in the three unknown displacements $v$, $\theta$,
$\psi$ is obtained, see \cite[(28)-(29)-(30)]{abdel}:
$$\left\{\begin{array}{l}
\frac{w}{g}\, v_{tt}+EI\, v_{xxxx}-\Big(2H_w+H_1(t)+H_2(t)\Big)\, v_{xx}+\frac{b}{2}\, \Big(H_1(t)-H_2(t)\Big)\, (\theta_{xx}+\psi_{xx})\\
\ \ \ +\frac{w}{2H_w}\, \Big(H_1(t)+H_2(t)\Big)-\frac{w_s\, r^2}{g}\, \left(1+\frac{EI}{2G\mu r^2}\right)\, v_{xxtt}+\frac{w_s^2\, r^2}{4gG\mu}\,
v_{tttt}=0\\
I_m\, \theta_{tt}+E\Gamma\, \theta_{xxxx}-GJ\, \theta_{xx}-\frac{H_w\, b^2}{2}\, (\theta_{xx}+\psi_{xx})-\frac{\gamma\, \Gamma}{g}\, \theta_{xxtt}
-\frac{b^2}{4}\, \Big(H_1(t)+H_2(t)\Big)\, (\theta_{xx}+\psi_{xx})\\
\ \ \ +\frac{b}{2}\, \Big(H_1(t)-H_2(t)\Big)\, v_{xx}-\frac{\gamma\, \Lambda}{g}
\psi_{xxtt}+\frac{b\, w}{4H_w}\, \Big(H_2(t)-H_1(t)\Big)+E\Lambda\, \psi_{xxxx}+\frac{w_c\, b^2}{4g}\, \psi_{tt}=0\\
\frac{w_c\, b^2}{4g}\, (\psi_{tt}+\theta_{tt})+\frac{EA\, b^2d^2}{4}\, \psi_{xxxx}-\frac{H_w\, b^2}{2}\, (\psi_{xx}+\theta_{xx})-
\frac{\gamma Ab^2d^2}{4g}\, \psi_{xxtt}-\frac{\gamma\, \Lambda}{g}\theta_{xxtt}+E\Lambda\, \theta_{xxxx}\\
\ \ \ -\frac{b^2}{4}\, \Big(H_1(t)+H_2(t)\Big)\, (\theta_{xx}+\psi_{xx})+\frac{b}{2}\, \Big(H_1(t)-H_2(t)\Big)\, v_{xx}
+\frac{w\, b}{4H_w}\, \Big(H_2(t)-H_1(t)\Big)=0\ .
\end{array}\right.$$
We will not explain here what is the meaning of all the constants involved, it would take several pages... Some of the constants have a clear meaning,
for the interpretation of the remaining ones, we refer to \cite{abdel}. Let us just mention that $H_1$ and $H_2$ represent the vibrational horizontal
components of the cable tension and depend on $v$, $\theta$, $\psi$, and their first derivatives, see \cite[(3)]{abdel}.
We wrote these equations in order to convince the reader that the behavior of the bridge is modeled by terribly complicated equations
and by no means one should make use of ${\cal LHL}$. After making such huge effort, Abdel-Ghaffar simplifies the problem by neglecting the cross section
deformation, the shear deformation and rotatory inertia; he obtains a coupled nonlinear vertical-torsional system of two equations in the two unknowns
functions $v$ and $\theta$. These equations are finally linearised, by neglecting $H_1$ and $H_2$ which are considered small when compared
with the initial tension $H_w$. Then the coupling effect disappears and equations \eq{eqqq} are recovered, see \cite[(34)-(35)]{abdel}. What a pity,
an accurate modeling ended up with a linearisation! But there was no choice... how can one imagine to get any kind of information from the above system?\par
Summarising, after the previously described pioneering models from \cite{bleich,melan,navier2,rocard,tac2} there has not been much work among engineers about
alternative differential equations; the attention has turned to improving performances through design factors,
see e.g.\ \cite{hhs}, or on how to solve structural problems rather than how to understand them more deeply. In this respect, from \cite[p.2]{mckmonth}
we quote a personal discussion between McKenna and a distinguished civil engineer who said
\begin{center}
\begin{minipage}{162mm}
{\em ... having found obvious and effective physical ways of avoiding the problem, engineers will not give too much attention to the mathematical solution
of this fascinating puzzle ...}
\end{minipage}
\end{center}

Only modeling modern footbridges has attracted some interest from a theoretical point of view. As already mentioned, pedestrian bridges are extremely
flexible and display elastic behaviors similar to suspension bridges, although the oscillations are of different kind.
In this respect, we would like to mention an interesting discussion with Diana \cite{diana}. He explained that when a suspension bridge is attacked
by wind its starts oscillating, but soon afterwards the wind itself modifies its behavior according to the bridge oscillation; so, the wind
amplifies the oscillations by blowing synchronously. A qualitative description of this phenomenon was already attempted by Rocard \cite[p.135]{rocard}:
\begin{center}
\begin{minipage}{162mm}
{\em ... it is physically certain and confirmed by ordinary experience, although the effect is known only qualitatively, that a bridge vibrating with an
appreciable amplitude completely imposes its own frequency on the vortices of its wake. It appears as if in some way the bridge itself discharges the vortices
into the fluid with a constant phase relationship with its own oscillation... .}
\end{minipage}
\end{center}
This reminds the above described behavior of footbridges where {\em pedestrians fall
spontaneously into step with the vibrations}: in both cases, external forces synchronise their effect and amplify the oscillations of the bridge.
This is one of the reasons why self-excited oscillations appear in suspension and pedestrian bridges.\par
In \cite{bodgi} a simple 1D model was proposed in order to describe the crowd-flow phenomena occurring when
pedestrians walk on a flexible footbridge. The resulting equation \cite[(2)]{bodgi} reads
\neweq{pedestrian}
\left(m_s(x)+m_p(x,t)\right)u_{tt}+\delta(x)u_t+\gamma(x)u_{xxxx}=g(x,t)
\endeq
where $x$ is the coordinate along the beam axis, $t$ the time, $u=u(x,t)$ the lateral displacement,
$m_s(x)$ is the mass per unit length of the beam, $m_p(x,t)$ the linear mass of pedestrians,
$\delta(x)$ the viscous damping coefficient, $\gamma(x)$ the stiffness per unit length, $g(x,t)$ the pedestrian lateral force
per unit length. In view of the superlinear behavior for large displacements observed for the London Millennium Bridge, see Section \ref{story},
we wonder if instead of a linear model one should consider a lateral force also depending on the displacement, $g=g(x,t,u)$, being
superlinear with respect to $u$.

\begin{problem} {\em Study \eq{pedestrian} modified as follows
$$
u_{tt}+\delta u_t+\gamma u_{xxxx}+f(u)=g(x,t)\qquad(x\in\R\, ,\ t>0)
$$
where $\delta>0$, $\gamma>0$ and $f(s)=s+\eps s^3$ for some $\eps>0$ small. One could first consider the Cauchy problem
$$u(x,0)=u_0(x)\ ,\quad u_t(x,0)=u_1(x)\quad(x\in\R)$$
with $g\equiv0$. Then one could seek traveling waves such as $u(x,t)=w(x-ct)$ which solve the ODE
$$
\gamma w''''(\tau)+c^2w''(\tau)+\delta c w'(\tau)+f(w(\tau))=0\qquad(x-ct=\tau\in\R).
$$
Finally, one could also try to find properties of solutions in a bounded interval $x\in(0,L)$.}\endproof\end{problem}

Scanlan-Tomko \cite{scantom} introduce a model in which the torsional angle $\theta$ of the roadway section satisfies the equation
\neweq{scann}
I\, [\theta''(t)+2\zeta_\theta\omega_\theta \theta'(t)+\omega_\theta^2\theta(t)]=A\theta'(t)+B\theta(t)\ ,
\endeq
where $I$, $\zeta_\theta$, $\omega_\theta$ are, respectively, associated inertia, damping ratio, and natural frequency. The r.h.s.\ of \eq{scann} represents
the aerodynamic force and was postulated to depend linearly on both $\theta'$ and $\theta$ with the positive constants $A$ and $B$ depending
on several parameters of the bridge. Since \eq{scann} may be seen as a two-variables first order linear system, it fails to fulfill both the requirements
of (GP). Hence, \eq{scann} is not suitable to describe the disordered behavior of a bridge. And indeed, elementary calculus shows that if $A$ is sufficiently
large, then solutions to \eq{scann} are positive exponentials times trigonometric functions which do not exhibit a sudden appearance of self-excited
oscillations, they merely blow up in infinite time. In order to have a more reliable description of the bridge, in Section \ref{blup} we consider
the fourth order nonlinear ODE $w''''+kw''+f(w)=0$ ($k\in\R$). We will see that solutions to this equation
blow up in finite time with self-excited oscillations appearing suddenly, without any intermediate stage.\par
That linearisation yields wrong models is also the opinion of McKenna \cite[p.4]{mckmonth} who comments \eq{scann} by writing
\begin{center}
\begin{minipage}{162mm}
{\em This is the point at which the discussion of torsional oscillation starts in the engineering literature.}
\end{minipage}
\end{center}
He claims that the problem is in fact nonlinear and that \eq{scann} is obtained after an incorrect linearisation. McKenna concludes by noticing that
\begin{center}
\begin{minipage}{162mm}
{\em Even in recent engineering literature ... this same mistake is reproduced.}
\end{minipage}
\end{center}
The mistake claimed by McKenna is that the equations are often linearised by taking $\sin\theta=\theta$ and $\cos\theta=1$ also for large amplitude torsional
oscillations $\theta$. The corresponding equation then becomes linear and the main torsional phenomenon disappears. Avoiding
this rude approximation, but considering the cables and hangers as linear springs obeying ${\cal LHL}$, McKenna reaches an uncoupled second order
system for the functions representing the vertical displacement $y$ of the barycenter $B$ of the cross section of the roadway and the deflection from
horizontal $\theta$, see Figure \ref{9}. Here, $2\ell$ denotes the width of the roadway whereas $C_1$ and $C_2$ denote the two lateral hangers
which have opposite extension behaviors.
\begin{figure}[ht]
\begin{center}
{\includegraphics[height=45mm, width=118mm]{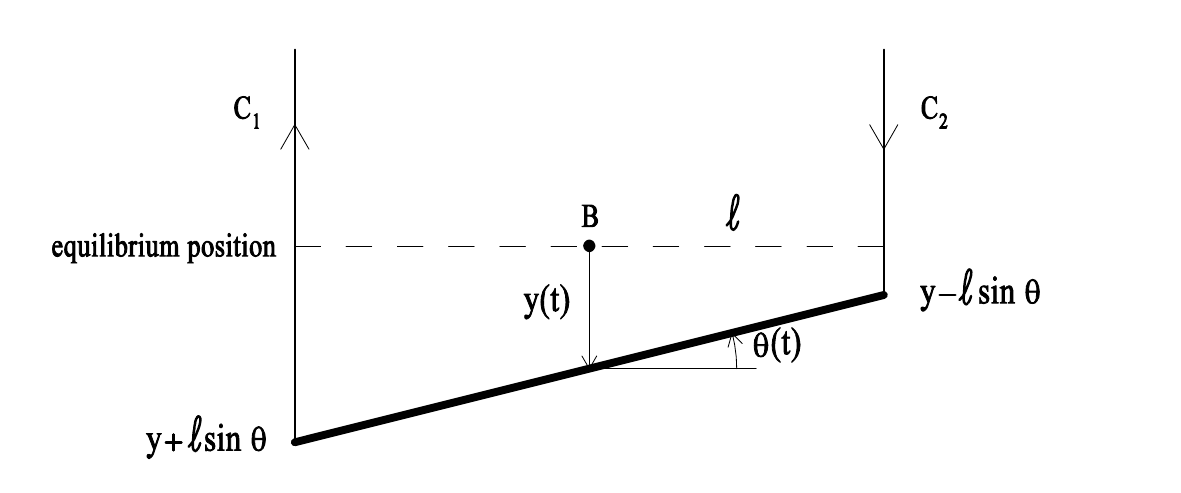}}
\caption{Vertical displacement and deflection of the cross section of the roadway.}\label{9}
\end{center}
\end{figure}

McKenna-Tuama \cite{mckO} suggest a slightly different model. They write:
\begin{center}
\begin{minipage}{162mm}
{\em ... there should be some torsional forcing. Otherwise, there would be no input of energy to overcome the natural
damping of the system ... we expect the bridge to behave like a stiff spring, with a restoring force that becomes somewhat superlinear.}
\end{minipage}
\end{center}
We completely agree with this, see the conclusions in Section \ref{conclusions}. McKenna-Tuama end up with the following coupled second order system
\neweq{coupled}
\frac{m\ell^2}{3}\, \theta''=\ell\cos\theta\, \Big(f(y-\ell\sin\theta)-f(y+\ell\sin\theta)\Big)\ ,\quad
m\, y''=-\Big(f(y-\ell\sin\theta)+f(y+\ell\sin\theta)\Big)\ ,
\endeq
see again Figure \ref{9}. The delicate point is the choice of the superlinearity $f$ which
\cite{mckO} take first as $f(s)=(s+1)^+-1$ and then as $f(s)=e^s-1$ in order to maintain the asymptotically
linear behavior as $s\to0$. Using \eq{coupled}, \cite{mckmonth,mckO} were able to numerically replicate the phenomenon observed at the Tacoma Bridge, namely the
sudden transition from vertical oscillations to torsional oscillations. They found that if the
vertical motion was sufficiently large to induce brief slackening of the hangers, then numerical results highlighted a rapid transition to a torsional motion.
Nevertheless, the physicists Green-Unruh \cite{green} believe that the hangers were not slack during the Tacoma Bridge oscillation. If this were true, then the
piecewise linear forcing term $f$ becomes totally linear. Moreover, by commenting the results in \cite{mckmonth,mckO}, McKenna-Moore \cite[p.460]{mckmoore}
write that
\begin{center}
\begin{minipage}{162mm}
{\em ...the range of parameters over which the transition from vertical to torsional motion was observed was physically unreasonable ... the restoring
force due to the cables was oversimplified ... it was necessary to impose small torsional forcing}.
\end{minipage}
\end{center}

Summarising, \eq{coupled} seems to be the first model able to reproduce the behavior of the Tacoma Bridge but it appears to need some
improvements. First, one should avoid the possibility of a linear behavior of the hangers, the nonlinearity should appear before possible slackening
of the hangers. Second, the restoring force and the parameters involved should be chosen carefully.

\begin{problem} {\em Try a doubly superlinear term $f$ in \eq{coupled}. For instance, take $f(s)=s+\eps s^3$ with $\eps>0$ small, so that
\eq{coupled} becomes
\neweq{mia}
\frac{m\ell^2}{3}\, \theta''+2\ell^2\cos\theta\sin\theta\, \Big(1+3\eps y^2+\eps\ell^2\sin^2\theta\Big)=0\ ,\quad
m\, y''+2\Big(1+3\eps\ell^2\sin^2\theta\Big)y+2\eps y^3=0\ .
\endeq
It appears challenging to determine some features of the solution $(y,\theta)$ to
\eq{mia} and also to perform numerical experiments to see what kind of oscillations are displayed by the solutions.}\endproof\end{problem}

System \eq{coupled} is a $2\times2$ system which should be considered as a nonlinear fourth order model; therefore, it fulfills the necessary conditions
of the general principle (GP). Another fourth order differential equation was suggested in \cite{lzmck,McKennaWalter,mck4} as a one-dimensional model
for a suspension bridge, namely a beam of length $L$ suspended by hangers. When the hangers are stretched there is a restoring force which is
proportional to the amount of stretching, according to ${\cal LHL}$. But when the beam moves in the opposite direction, there is no restoring force
exerted on it. Under suitable boundary conditions, if $u(x,t)$ denotes the vertical displacement of the beam in the downward direction at position
$x$ and time $t$, the following nonlinear beam equation is derived
\neweq{beam}
u_{tt}+u_{xxxx}+\gamma u^+=W(x,t)\, ,\qquad x\in(0,L)\, ,\quad t>0\, ,
\endeq
where $u^+=\max\{u,0\}$, $\gamma u^+$ represents the force due to the cables and hangers which are considered as a linear spring with a one-sided
restoring force, and $W$ represents the forcing term acting on the bridge, including its own weight per unit
length, the wind, the traffic loads, or other external sources. After some normalisation, by seeking traveling waves $u(x,t)=1+w(x-ct)$ to \eq{beam}
and putting $k=c^2>0$, McKenna-Walter \cite{mck4} reach the following ODE
\neweq{maineq}
w''''(\tau)+kw''(\tau)+f(w(\tau))=0\qquad(x-ct=\tau\in\R)
\endeq
where $k\in(0,2)$ and $f(s)=(s+1)^+-1$. Subsequently, in order to maintain the same behavior but with a smooth nonlinearity, Chen-McKenna \cite{chenmck}
suggest to consider \eq{maineq} with $f(s)=e^s-1$. For later discussion, we notice that both these nonlinearities satisfy
\neweq{f}
f\in {\rm Lip}_{{\rm loc}}(\R)\,,\quad f(s)\,s>0 \quad \forall s\in \R\setminus\{0\}.
\endeq
Hence, when $W\equiv0$, \eq{beam} is just a special case of the more general semilinear fourth order wave equation
\neweq{beam2}
u_{tt}+u_{xxxx}+f(u)=0\, ,\qquad x\in(0,L)\, ,\quad t>0\, ,
\endeq
where the natural assumptions on $f$ are \eq{f} plus further conditions, according to the model considered. Traveling waves to \eq{beam2} solve \eq{maineq}
with $k=c^2$ being the squared velocity of the wave. Recently, for $f(s)=(s+1)^+-1$ and its variants, Benci-Fortunato \cite{benci} proved the existence of
special solutions to \eq{maineq} deduced by solitons of the beam equation \eq{beam2}.

\begin{problem} {\em It could be interesting to insert into the wave-type equation \eq{beam2} the term corresponding to the beam elongation, that is,
$$\int_0^L\Big(\sqrt{1+u_x(x,t)^2}-1\Big)\, dx.$$
This would lead to a quasilinear equation such as
$$u_{tt}+u_{xxxx}-\left(\frac{u_x}{\sqrt{1+u_x^2}}\right)_x+f(u)=0$$
with $f$ satisfying \eq{f}. What can be said about this equation? Does it admit oscillating solutions in a suitable sense? One should first consider
the case of an unbounded beam ($x\in\R$) and then the case of a bounded beam ($x\in(0,L)$) complemented with some boundary conditions.}\endproof\end{problem}

Motivated by the fact that it appears unnatural to ignore the motion of the main sustaining cable, a slightly more sophisticated and complicated string-beam
model was suggested by Lazer-McKenna \cite{mck1}. They treat the cable as a vibrating string, coupled with the vibrating beam of the roadway
by piecewise linear springs that have a given spring constant $k$ if expanded, but no restoring force if compressed. The sustaining
cable is subject to some forcing term such as the wind or the motions in the towers. This leads to the system
$$\left\{\begin{array}{ll}
v_{tt}-c_1v_{xx}+\delta_1v_t-k_1(u-v)^+=f(x,t)\, ,\qquad x\in(0,L)\, ,\quad t>0\, ,\\
u_{tt}+c_2u_{xxxx}+\delta_2u_t+k_2(u-v)^+=W_0\, ,\qquad x\in(0,L)\, ,\quad t>0\, ,
\end{array}\right.$$
where $v$ is the displacement from equilibrium of the cable and $u$ is the displacement of the beam, both measured in the downwards direction.
The constants $c_1$ and $c_2$ represent the relative strengths of the cables and roadway respectively, whereas $k_1$ and $k_2$ are the spring
constants and satisfy $k_2\ll k_1$. The two damping terms can possibly be set to $0$, while $f$ and $W_0$ are the forcing terms. We also refer
to \cite{ahmed} for a study of the same problem in a rigorous functional analytic setting.\par
Since the Tacoma Bridge collapse was mainly due to a wide torsional motion of the bridge, see \cite{tacoma}, the bridge cannot be considered as a one
dimensional beam. In this respect, Rocard \cite[p.148]{rocard} states
\begin{center}
\begin{minipage}{162mm}
{\em Conventional suspension bridges are fundamentally unstable in the wind because the coupling effect introduced between bending and torsion by the
aerodynamic forces of the lift.}
\end{minipage}
\end{center}
Hence, if some model wishes to display instability of bridges, it should necessarily take into account more degrees of freedom than just a beam.
In fact, to be exhaustive one should consider vertical oscillations $y$ of the roadway, its torsional angle $\theta$, and coupling with
the two sustaining cables $u$ and $v$. This model was suggested by Matas-O\v cen\'a\v sek \cite{matas} who consider the hangers as linear springs
and obtain a system of four equations; three of them are second order wave-type equations, the last one is again a fourth order equation such as
$$m\, y_{tt}+k\, y_{xxxx}+\delta\, y_t+E_1(y-u-\ell\sin\theta)+E_2(y-v+\ell\sin\theta)=W(x)+f(x,t)\ ;$$
we refer to $(SB_4)$ in \cite{drabek} for an interpretation of the parameters involved.\par
In our opinion, any model which describes the bridge as a one dimensional beam is too simplistic, unless the model takes somehow into account the possible
appearance of a torsional motion. In \cite{gazpav} it was suggested to maintain the one dimensional model provided one also allows displacements below the
equilibrium position and these displacements replace the deflection from horizontal of the roadway of the bridge; in other words,
\renewcommand{\arraystretch}{1.1}
\neweq{w0}
\begin{array}{c}
\mbox{the unknown function $w$ represents the upwards vertical displacement when $w>0$}\\
\mbox{and the deflection from horizontal, computed in a suitable unity measure, when $w<0$.}
\end{array}
\endeq
\renewcommand{\arraystretch}{1.5}In this setting, instead of \eq{beam} one should consider the more general semilinear fourth order wave
equation \eq{beam2} with $f$
satisfying \eq{f} plus further conditions which make $f(s)$ superlinear and unbounded when both $s\to\pm\infty$; hence, ${\cal LHL}$ is dropped
by allowing $f$ to be as close as one may wish to a linear function but eventually superlinear for large displacements.
The superlinearity assumption is justified both by the observations in Section \ref{story} and by the fact that more the position of the bridge is far
from the horizontal equilibrium position, more the action of the wind becomes relevant because the wind hits transversally the roadway of the bridge.
If ever the bridge would reach the limit vertical position, in case the roadway is torsionally rotated of a right angle, the wind would hit it
orthogonally, that is, with full power.\par\medskip
In this section we listed a number of attempts to model bridges mechanics by means of differential equations. The sources for this list
are very heterogeneous. However, except for some possible small damping term, none of them
contains odd derivatives. Moreover, none of them is acknowledged by the scientific community to perfectly describe the complex behavior of bridges.
Some of them fail to satisfy the requirements of (GP) and, in our opinion, must be accordingly modified. Some others seem to better describe
the oscillating behavior of bridges but still need some improvements.

\section{Blow up oscillating solutions to some fourth order differential equations}\label{blup}

If the trivial solution to some dynamical system is unstable one may hope to magnify self-excitement phenomena through finite time blow up. In this section we
survey and discuss several results about solutions to \eq{maineq} which blow up in finite time. Let us rewrite the equation with a different time variable, namely
\neweq{maineq2}
w''''(t)+kw''(t)+f(w(t))=0\qquad(t\in\R)\ .
\endeq
We first recall the following results proved in \cite{bfgk}:

\begin{theorem}\label{global}
Let $k\in \R$ and assume that $f$ satisfies \eqref{f}.\par
$(i)$ If a local solution $w$ to \eqref{maineq2} blows up at some finite $R\in\R$, then
\neweq{pazzo}
\liminf_{t\to R}w(t)=-\infty\qquad\mbox{and}\qquad\limsup_{t\to R}w(t)=+\infty\, .
\endeq

$(ii)$ If $f$ also satisfies
\neweq{ff3}
\limsup_{s\to+\infty}\frac{f(s)}{s}<+\infty\qquad\mbox{or}\qquad\limsup_{s\to-\infty}\frac{f(s)}{s}<+\infty,
\endeq
then any local solution to \eqref{maineq2} exists for all $t\in\R$.
\end{theorem}

If both the conditions in \eq{ff3} are satisfied then global existence follows from classical theory of ODE's; but \eq{ff3} merely requires
that $f$ is ``one-sided at most linear'' so that statement $(ii)$ is far from being trivial and, as shown in \cite{gazpav}, it does not hold for
equations of order at most 3. On the other hand, Theorem \ref{global} $(i)$ states that, under the sole assumption \eq{f}, the only way that
finite time blow up can occur is with ``wide and thinning oscillations'' of the solution $w$; again, in \cite{gazpav} it was shown that this kind of
blow up is a phenomenon typical of at least fourth order problems such as \eq{maineq2} since it does not occur in related lower order equations.
Note that assumption \eq{ff3} includes, in particular, the cases where $f$ is either concave or convex.\par
Theorem \ref{global} does not guarantee that the blow up described by \eq{pazzo} indeed occurs. For this reason, we assume further that
\neweq{fmono}
f\in {\rm Lip}_{{\rm loc}}(\R)\cap C^2(\R\setminus\{0\})\ ,\quad f'(s)\ge0\quad\forall s\in\R\ ,\quad\liminf_{s\to\pm\infty}|f''(s)|>0
\endeq
and the growth conditions
\neweq{f2}
\exists p>q\ge1,\ \alpha\ge0,\ 0<\rho\le \beta,\quad\mbox{s.t.}\quad\rho|s|^{p+1}\le f(s)s\le\alpha|s|^{q+1}+\beta|s|^{p+1}\quad\forall s\in\R\ .
\endeq
Notice that \eq{fmono}-\eq{f2} strengthen \eq{f}. In \cite{gazpav3} the following sufficient conditions for the finite
time blow up of local solutions to \eq{maineq2} has been proved.

\begin{theorem}\label{blowup}
Let $k\le0$, $p>q\ge1$, $\alpha\ge0$, and assume that $f$ satisfies \eqref{fmono} and \eqref{f2}.
Assume that $w=w(t)$ is a local solution to \eqref{maineq2} in a neighborhood of $t=0$ which satisfies
\neweq{tech}
w'(0)w''(0)-w(0)w'''(0)-kw(0)w'(0)>0\, .
\endeq
Then, $w$ blows up in finite time for $t>0$, that is, there exists $R\in(0,+\infty)$ such that \eqref{pazzo} holds.
\end{theorem}

Since self-excited oscillations such as \eq{pazzo} should be expected in any equation attempting to model suspension bridges, linear models should be avoided.
Unfortunately, even if the solutions to \eq{maineq2} display these oscillations, they cannot be prevented since they arise suddenly after a long time of apparent
calm. In Figure \ref{duemila}, we display the plot of a solution to \eq{maineq2}.
\begin{figure}[ht]
\begin{center}
{\includegraphics[height=33mm, width=51mm]{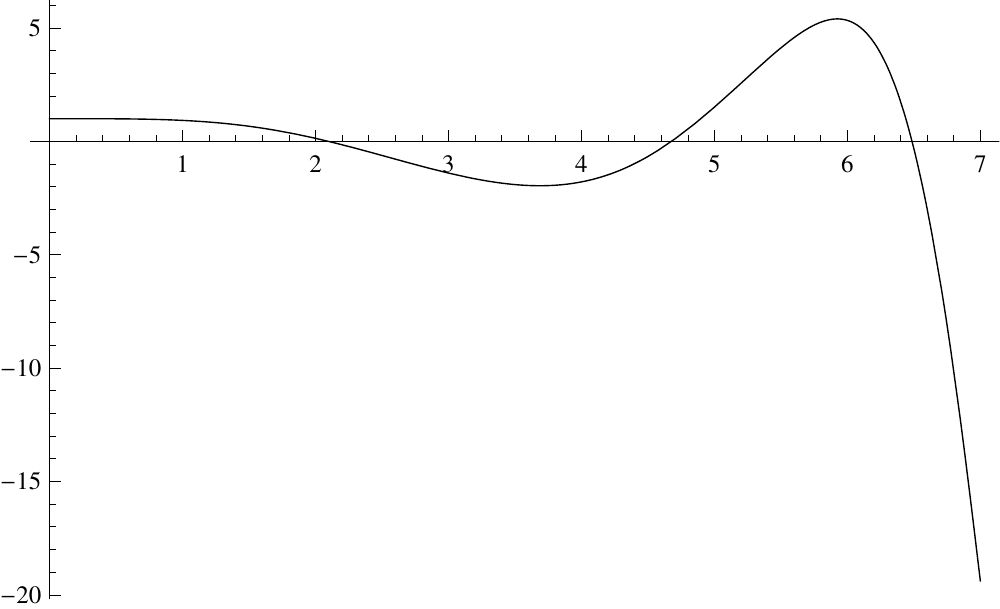}}\quad{\includegraphics[height=33mm, width=51mm]{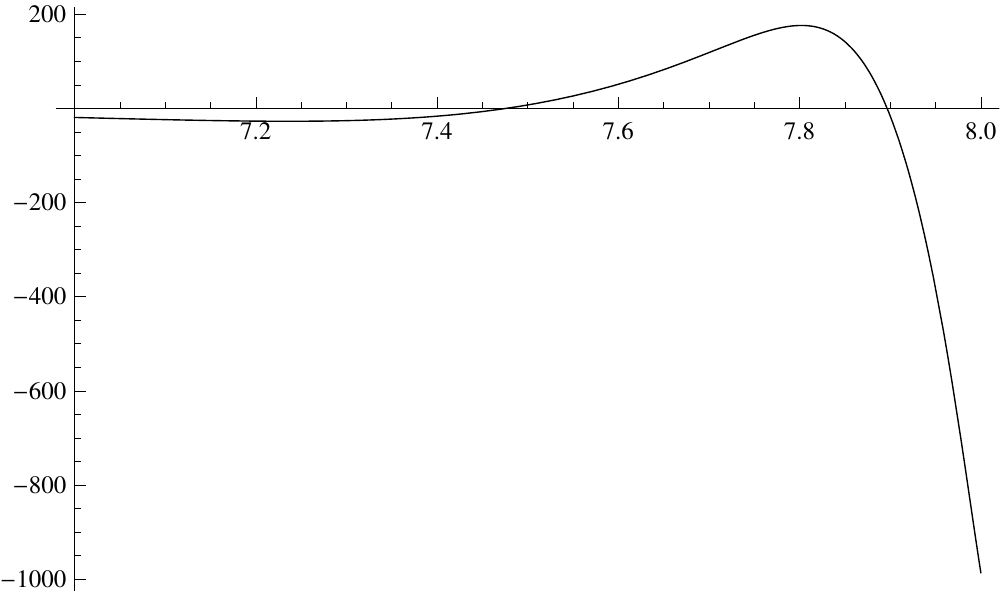}}\quad
{\includegraphics[height=33mm, width=51mm]{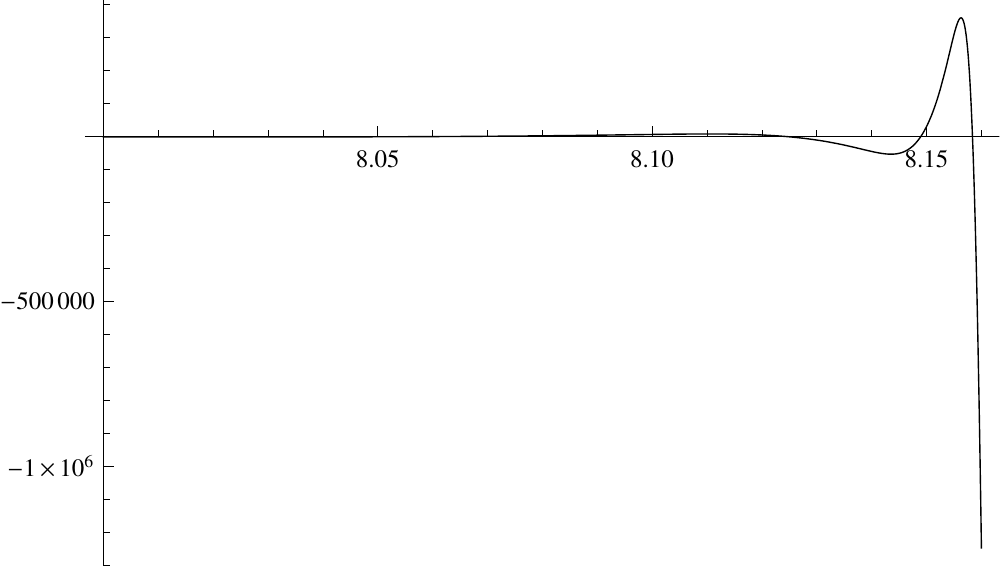}}
\caption{Solution to \eq{maineq2} for $k=3$, $[w(0),w'(0),w''(0),w'''(0)]=[1,0,0,0]$, $f(s)=s+s^3$. The three intervals are $t\in[0,7]$, $t\in[7,8]$, $t\in[8,8.16]$}\label{duemila}
\end{center}
\end{figure}
It can be observed that the solution has oscillations with increasing amplitude and rapidly decreasing ``nonlinear frequency"; numerically, the blow up seems to occur
at $t=8.164$. Even more impressive appears the plot in Figure \ref{mille}.
\begin{figure}[ht]
\begin{center}
{\includegraphics[height=33mm, width=51mm]{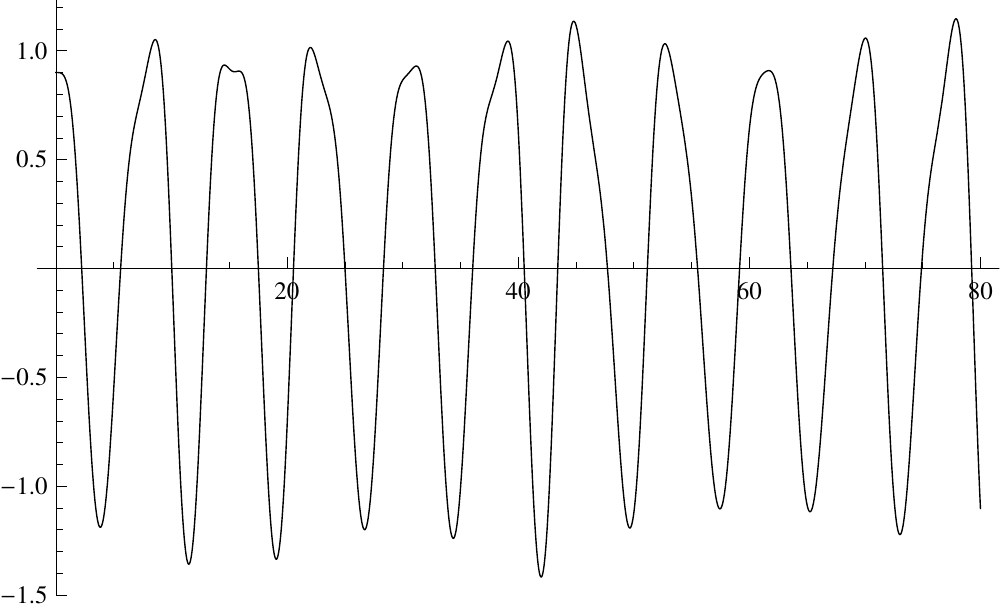}}\quad{\includegraphics[height=33mm, width=51mm]{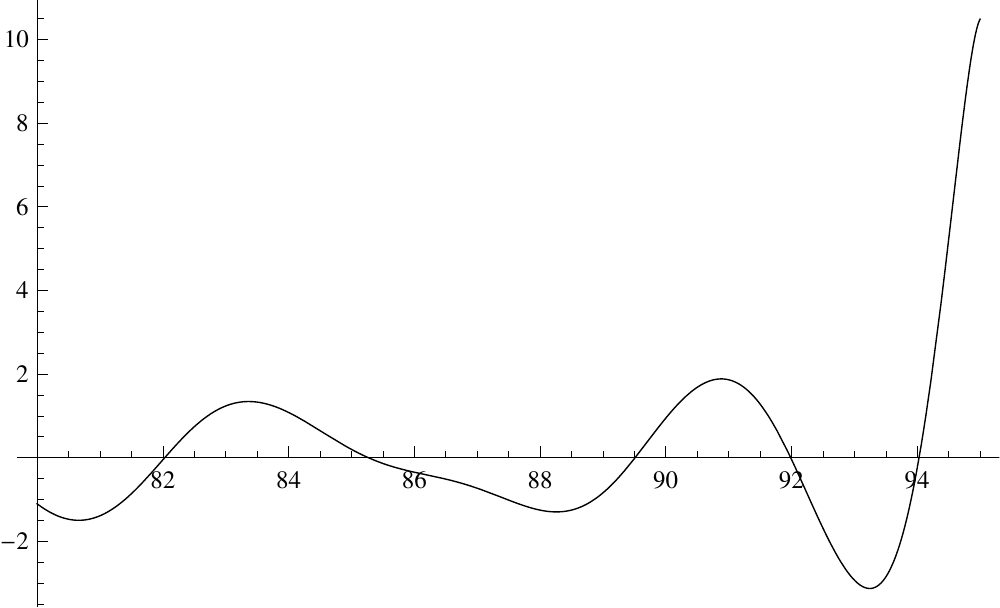}}\quad
{\includegraphics[height=33mm, width=51mm]{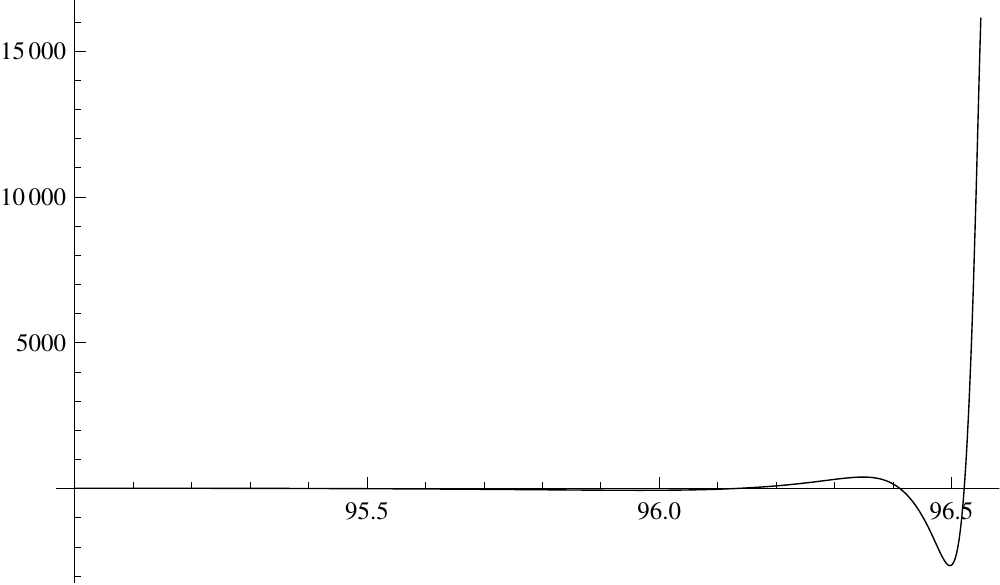}}
\caption{Solution to \eq{maineq2} for $k=3.6$, $[w(0),w'(0),w''(0),w'''(0)]=[0.9,0,0,0]$, $f(s)=s+s^3$.
The three intervals are $t\in[0,80]$, $t\in[80,95]$, $t\in[95,96.55]$}\label{mille}
\end{center}
\end{figure}
Here the solution has ``almost regular'' oscillations between $-1$ and $+1$ for $t\in[0,80]$. Then the amplitude of oscillations nearly doubles in the interval
$[80,93]$ and, suddenly, it violently amplifies after $t=96.5$ until the blow up which seems to occur only slightly later at $t=96.59$. We also refer
to \cite{gazpav,gazpav2,gazpav3} for further plots.\par
We refer to \cite{gazpav,gazpav3} for numerical results and plots of solutions to \eq{maineq2} with nonlinearities $f=f(s)$ having different growths as
$s\to\pm\infty$. In such case, the solution still blows up according to \eq{pazzo} but, although its ``limsup'' and ``liminf'' are respectively $+\infty$
and $-\infty$, the divergence occurs at different rates. We represent this qualitative behavior in Figure \ref{blow}.
\begin{figure}[ht]
\begin{center}
{\includegraphics[height=5cm, width=14cm]{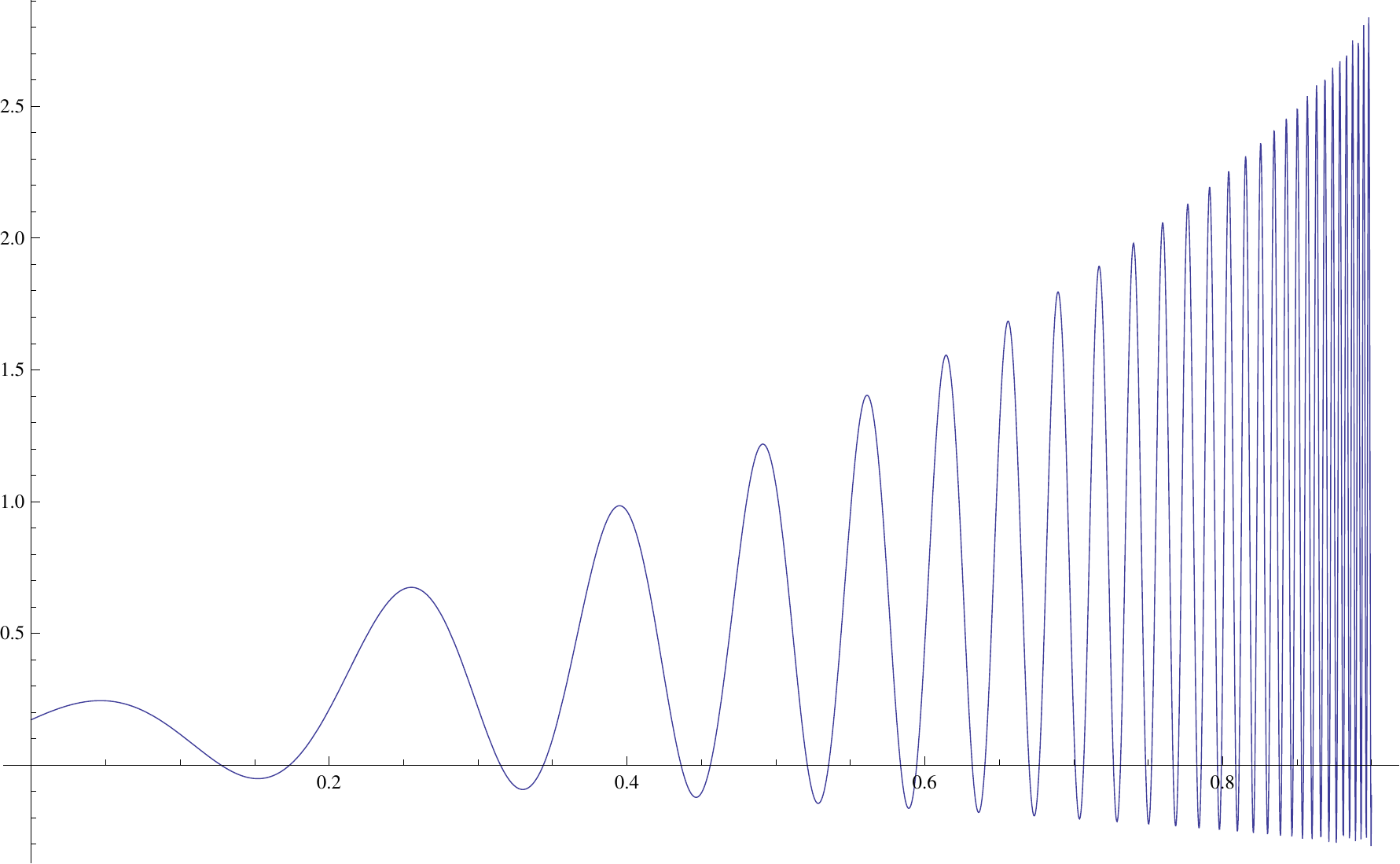}}
\caption{Qualitative blow up for solutions to \eq{maineq2} when $f(s)$ is superlinear at different rates as $s\to\pm\infty$.}\label{blow}
\end{center}
\end{figure}

Traveling waves to \eq{beam2} which propagate at some velocity $c>0$, depending on the elasticity of the material of the beam, solve
\eq{maineq2} with $k=c^2>0$. Further numerical results obtained in \cite{gazpav,gazpav3} suggest that a statement similar to Theorem
\ref{blowup} also holds for $k>0$ and, as expected, that the blow up time $R$ is decreasing with respect to the initial height $w(0)$
and increasing with respect to $k$. Since $k=c^2$ and $c$ represents the velocity of the traveling wave, this means that the time of blow up
is an increasing function of $k$. In turn, since the velocity of the traveling wave depends on the elasticity of the material
used to construct the bridge (larger $c$ means less elastic), this tells us that more the bridge is stiff more it will survive to
exterior forces such as the wind and/or traffic loads.

\begin{problem} {\em Prove Theorem \ref{blowup} when $k>0$. This would allow to show that traveling waves to \eq{beam2} blow up in finite time.
Numerical results in \cite{gazpav,gazpav3} suggest that a result similar to Theorem \ref{blowup} also holds for $k>0$.}\endproof\end{problem}

\begin{problem} {\em Prove that the blow up time of solutions to \eq{maineq2} depends increasingly with respect to $k\in\R$.
The interest of an analytical proof of this fact relies on the important role played by $k$ within the model.}\endproof\end{problem}

\begin{problem} {\em The blow up time $R$ of solutions to \eq{maineq2} is the expectation of life of the oscillating bridge. Provide an estimate of
$R$ in terms of $f$ and of the initial data.}\endproof\end{problem}

\begin{problem} {\em Condition \eq{f2} is a superlinearity assumption which requires that $f$ is bounded both from above and below by the same power $p>1$.
Prove Theorem \ref{blowup} for more general kinds of superlinear functions $f$.}\endproof\end{problem}

\begin{problem} {\em Can assumption \eq{tech} be relaxed? Of course, it cannot be completely removed since the trivial solution $w(t)\equiv0$ is globally
defined, that is, $R=+\infty$. Numerical experiments in \cite{gazpav,gazpav3} could not detect any nontrivial global solution
to \eq{maineq2}.}\endproof\end{problem}

\begin{problem} {\em Study \eq{maineq2} with a damping term: $w''''(t)+kw''(t)+\delta w'(t)+f(w(t))=0$ for some $\delta>0$. Study the competition between
the damping term $\delta w'$ and the nonlinear self-exciting term $f(w)$.}\endproof\end{problem}

Note that Theorems \ref{global} and \ref{blowup} ensure that there exists an increasing sequence $\{z_j\}_{j\in\N}$ such that:\par
$(i)$ $z_j\nearrow R$ as $j\to\infty$;\par
$(ii)$ $w(z_j)=0$ and $w$ has constant sign in $(z_j,z_{j+1})$ for all $j\in\N$.\par
It is also interesting to compare the rate of blow up of the displacement and of the acceleration on these intervals. By slightly modifying the proof of
\cite[Theorem 3]{gazpav3} one can obtain the following result which holds for any $k\in\R$.

\begin{theorem}\label{asymptotics}
Let $k\in \R$, $p>q\ge1$, $\alpha\ge0$, and assume that $f$ satisfies \eqref{fmono} and \eqref{f2}.
Assume that $w=w(t)$ is a local solution to
$$
w''''(t)+kw''(t)+f(w(t))=0\qquad(t\in\R)
$$
which blows up in finite time as $t\nearrow R<+\infty$. Denote by $\{z_j\}$ the increasing sequence of zeros
of $w$ such that $z_j\nearrow R$ as $j\to+\infty$. Then
\neweq{estimate}
\int_{z_j}^{z_{j+1}}w(t)^2\, dt\ \ll\ \int_{z_j}^{z_{j+1}}w''(t)^2\, dt\ ,\qquad
\int_{z_j}^{z_{j+1}}w'(t)^2\, dt\ \ll\ \int_{z_j}^{z_{j+1}}w''(t)^2\, dt
\endeq
as $j\to\infty$. Here, $g(j)\ll\psi(j)$ means that $g(j)/\psi(j)\to0$ as $j\to\infty$.
\end{theorem}

The estimate \eq{estimate}, clearly due to the superlinear term, has a simple interpretation in terms of comparison between blowing up energies,
see Section \ref{energies}.

\begin{remark}\label{pde} {\em Equation \eq{maineq2} also arises in several different contexts, see the book by Peletier-Troy \cite{pt} where
one can find some other physical models, a survey of existing results, and further references. Moreover, besides \eq{beam2},
\eq{maineq2} may also be fruitfully used to study some other partial differential equations. For instance, one can consider nonlinear elliptic equations such as
$$
\Delta^2u+e^u=\frac{1}{|x|^4}\qquad\mbox{in }\R^4\setminus\{0\}\ ,
$$
\neweq{critical}
\Delta^2 u+|u|^{8/(n-4)}u=0\mbox{ in }\R^n\ (n\ge5),\qquad\Delta\Big(|x|^2\Delta u\Big)+|x|^2|u|^{8/(n-2)}u=0\mbox{ in }\R^n\ (n\ge3);
\endeq
it is known (see, e.g.\ \cite{gazgruswe}) that the Green function for some fourth order elliptic problems displays oscillations, differently from
second order problems. Furthermore, one can also consider the semilinear parabolic equation
$$u_t+\Delta ^2u=|u|^{p-1}u\mbox{ in }\mathbb{R}_{+}^{n+1}\ ,\qquad u(x,0)=u_0(x)\mbox{ in }\mathbb{R}^{n}$$
where $p>1+4/n$ and $u_0$ satisfies suitable assumptions. It is shown in \cite{fgg,gg2} that the linear biharmonic heat operator has an
``eventual local positivity'' property: for positive initial data $u_0$ the solution to the linear problem with no source is eventually
positive on compact subsets of $\R^n$ but negativity can appear at any time far away from the origin. This phenomenon is due to the
sign changing properties, with infinite oscillations, of the biharmonic heat kernels.
We also refer to \cite{bfgk,gazpav3} for some results about the above equations and for the explanation of how they can be reduced to \eq{maineq2} and, hence,
how they display self-excited oscillations.\endproof}
\end{remark}

\begin{problem} {\em For any $q>0$ and parameters $a,b,k\in\R$, $c\ge0$, study the equation
\neweq{subcrit}
w''''(t)+aw'''(t)+kw''(t)+bw'(t)+cw(t)+|w(t)|^qw(t)=0\qquad(t\in\R)\ .
\endeq
Any reader who is familiar with the second order Sobolev space $H^2$ recognises the critical exponent in the first equation
in \eq{critical}. In view of Liouville-type results in \cite{ambrosio} when $q\le8/(n-4)$, it would be interesting to study the equation
$\Delta^2 u+|u|^qu=0$ with the same technique. The radial form of this equation may be written as \eq{maineq2} only when $q=8/(n-4)$ since for other values
of $q$ the transformation in \cite{GG} gives rise to the appearance of first and third order derivatives as in \eq{subcrit}: this motivates \eq{subcrit}.
The values of the parameters corresponding to the equation $\Delta^2 u+|u|^qu=0$ can be found in \cite{GG}.}\endproof\end{problem}

Our target is now to reproduce the self-excited oscillations found in Theorem \ref{blowup} in a suitable second order system.
Replace $\sin\theta\cong\theta$ and $\cos\theta\cong1$, and put $x=\ell\theta$. After these transformations, the McKenna system \eq{coupled} reads
\neweq{truesystem}
\left\{\begin{array}{ll}
x''+\omega^2 f(y+x)-\omega^2 f(y-x)=0\\
y''+f(y+x)+f(y-x)=0\ .
\end{array}\right.\endeq

We further modify \eq{truesystem}; for suitable values of the parameters $\beta$ and $\delta$, we consider the system
\neweq{miosystxy}
\left\{\begin{array}{ll}
x''-f(y-x)+\beta(y+x)=0\\
y''-f(y-x)+\delta(y+x)=0
\end{array}\right.
\endeq
which differs from \eq{truesystem} in two respects: the minus sign in front of $f(y-x)$ in the second equation and the other restoring
force $f(y+x)$ being replaced by a linear term. To \eq{miosystxy} we associate the initial value problem
\neweq{cauchy}
x(0)=x_0\, ,\ x'(0)=x_1\, ,\ y(0)=y_0\, ,\ y'(0)=y_1\ .
\endeq
The following statement holds.

\begin{theorem}\label{oscill}
Assume that $\beta<\delta\le-\beta$ (so that $\beta<0$). Assume also that $f(s)=\sigma s+cs^2+ds^3$ with $d>0$ and $c^2\le2d\sigma$.
Let $(x_0,y_0,x_1,y_1)\in\R^4$ satisfy
\neweq{initial}
(3\beta-\delta)x_0y_1+(3\delta-\beta)x_1y_0>(\beta+\delta)(x_0x_1+y_0y_1)\ .
\endeq
If $(x,y)$ is a local solution to \eqref{miosystxy}-\eqref{cauchy} in a neighborhood of $t=0$, then $(x,y)$ blows up in finite time for $t>0$ with
self-excited oscillations, that is, there exists $R\in(0,+\infty)$ such that
$$\liminf_{t\to R}x(t)=\liminf_{t\to R}y(t)=-\infty\qquad\mbox{and}\qquad\limsup_{t\to R}x(t)=\limsup_{t\to R}y(t)=+\infty\, .$$
\end{theorem}
\begin{proof} After performing the change of variables \eq{change}, system \eq{miosystxy} becomes
$$w''+(\delta-\beta)z=0\ ,\qquad z''-2f(w)+(\beta+\delta)z=0$$
which may be rewritten as a single fourth order equation
\neweq{mia4}
w''''(t)+(\beta+\delta)w''(t)+2(\delta-\beta)f(w(t))=0\ .
\endeq
Assumption \eq{initial} reads
$$w'(0)w''(0)-w(0)w'''(0)-(\beta+\delta)w(0)w'(0)>0\, .$$
Furthermore, in view of the above assumptions, $f$ satisfies \eq{fmono}-\eq{f2} with $\rho=d/2$, $p=3$, $\alpha=2\sigma$, $q=1$, $\beta=3d$.
Whence, Theorem \ref{asymptotics} states that $w$ blows up in finite time for $t>0$ and that there exists $R\in(0,+\infty)$ such that
\neweq{puzzo}
\liminf_{t\to R}w(t)=-\infty\qquad\mbox{and}\qquad\limsup_{t\to R}w(t)=+\infty\, .
\endeq

Next, we remark that \eq{mia4} admits a first integral, namely
\begin{eqnarray}
E(t) &:=& \frac{\beta+\delta}{2}\,w'(t)^2+w'(t)w'''(t)+2(\delta-\beta)F(w(t))-\frac{1}{2}\,w''(t)^2 \notag \\
\ &=& \frac{\beta+\delta}{2}\,w'(t)^2+(\beta-\delta)w'(t)z'(t)+2(\delta-\beta)F(w(t))-\frac{(\beta-\delta)^2}{2}\,z(t)^2\equiv\overline{E}\ , \label{E}
\end{eqnarray}
for some constant $\overline{E}$. By \eq{puzzo} there exists an increasing sequence $m_j\to R$ of local maxima of $w$ such that
$$z(m_j)=\frac{w''(m_j)}{\beta-\delta}\ge0\ ,\quad w'(m_j)=0\ ,\quad w(m_j)\to+\infty\mbox{ as }j\to\infty\ .$$
By plugging $m_j$ into the first integral \eq{E} we obtain
$$\overline{E}=E(m_j)=2(\delta-\beta)F(w(m_j))-\frac{(\beta-\delta)^2}{2}\,z(m_j)^2$$
which proves that $z(m_j)\to+\infty$ as $j\to+\infty$. We may proceed similarly in order to show that $z(\mu_j)\to-\infty$ on a sequence $\{\mu_j\}$ of local
minima of $w$. Therefore, we have
$$
\liminf_{t\to R}z(t)=-\infty\qquad\mbox{and}\qquad\limsup_{t\to R}z(t)=+\infty\, .
$$

Assume for contradiction that there exists $K\in\R$ such that $x(t)\le K$ for all $t<R$. Then, recalling \eq{change}, on the above sequence $\{m_j\}$ of local
maxima for $w$, we would have $y(m_j)-K\ge y(m_j)-x(m_j)=w(m_j)\to+\infty$ which is incompatible with \eq{E} since
$$2(\delta-\beta)F(y(m_j)-x(m_j))-\frac{(\beta-\delta)^2}{2}\, (y(m_j)+x(m_j))^2\equiv\overline{E}$$
and $F$ has growth of order 4 with respect to its divergent argument. Similarly, by arguing on the sequence $\{\mu_j\}$, we rule out the possibility that there exists
$K\in\R$ such that $x(t)\ge K$ for ll $t<R$. Finally, by changing the role of $x$ and $y$ we find that also $y(t)$ is unbounded both from above and below
as $t\to R$. This completes the proof.\end{proof}\smallskip

\begin{remark} {\em Numerical results in \cite{gazpav3} suggest that the assumption $\delta\le-\beta$ is not necessary to obtain \eq{puzzo}.
So, most probably, Theorem \ref{oscill} and the results of this section hold true also without this assumption.\endproof}\end{remark}

A special case of function $f$ satisfying the assumptions of Theorem \ref{oscill} is $f_\eps(s)=s+\eps s^3$ for any $\eps>0$.
We wish to study the situation when the problem tends to become linear, that is, when $\eps\to0$. Plugging such $f_\eps$ into \eq{miosystxy} gives the system
\neweq{fe}
\left\{\begin{array}{ll}
x''+(\beta+1)x+(\beta-1)y+\eps(x-y)^3=0\\
y''+(\delta+1)x+(\delta-1)y+\eps(x-y)^3=0
\end{array}\right.
\endeq
so that the limit linear problem obtained for $\eps=0$ reads
\neweq{f0}
\left\{\begin{array}{ll}
x''+(\beta+1)x+(\beta-1)y=0\\
y''+(\delta+1)x+(\delta-1)y=0\ .
\end{array}\right.
\endeq
The theory of linear systems tells us that the shape of the solutions to \eq{f0} depends on the signs of the parameters
$$A=\beta+\delta\ ,\quad B=2(\delta-\beta)\ ,\quad \Delta=(\beta+\delta)^2+8(\beta-\delta)\ .$$
Under the same assumptions of Theorem \ref{oscill}, for \eq{f0} we have $A\le0$ and $B>0$ but the sign of $\Delta$ is not known a priori and three different cases may occur.\par
$\bullet$ If $\Delta<0$ (a case including also $A=0$), then we have exponentials times trigonometric functions so either we have self-excited oscillations
which increase amplitude as $t\to\infty$ or we have damped oscillations which tend to vanish as $t\to\infty$.
Consider the case $\delta=-\beta=1$ and $(x_0,y_0,x_1,y_1)=(1,0,1,-1)$, then \eq{initial} is fulfilled and Theorem \ref{oscill} yields

\begin{corollary}
For any $\eps>0$ there exists $R_\eps>0$ such that the solution $(x^\eps,y^\eps)$ to the Cauchy problem
\neweq{feps}
\left\{\begin{array}{ll}
x''-2y+\eps(x-y)^3=0\\
y''+2x+\eps(x-y)^3=0\\
x(0)=1,\ y(0)=0,\ x'(0)=1,\ y'(0)=-1
\end{array}\right.
\endeq
blows up as $t\to R_\eps$ and satisfies
$$
\liminf_{t\to R_\eps}x^\eps(t)=\liminf_{t\to R_\eps}y^\eps(t)=-\infty\qquad\mbox{and}\qquad\limsup_{t\to R_\eps}x^\eps(t)=\limsup_{t\to R_\eps}y^\eps(t)=+\infty\, .
$$
\end{corollary}\smallskip

A natural conjecture, supported by numerical experiments, is that $R_\eps\to\infty$ as $\eps\to0$.
For several $\eps>0$, we plotted the solution to \eq{feps} and the pictures all looked like Figure \ref{plot1}.
\begin{figure}[ht]
\begin{center}
{\includegraphics[height=38mm, width=50mm]{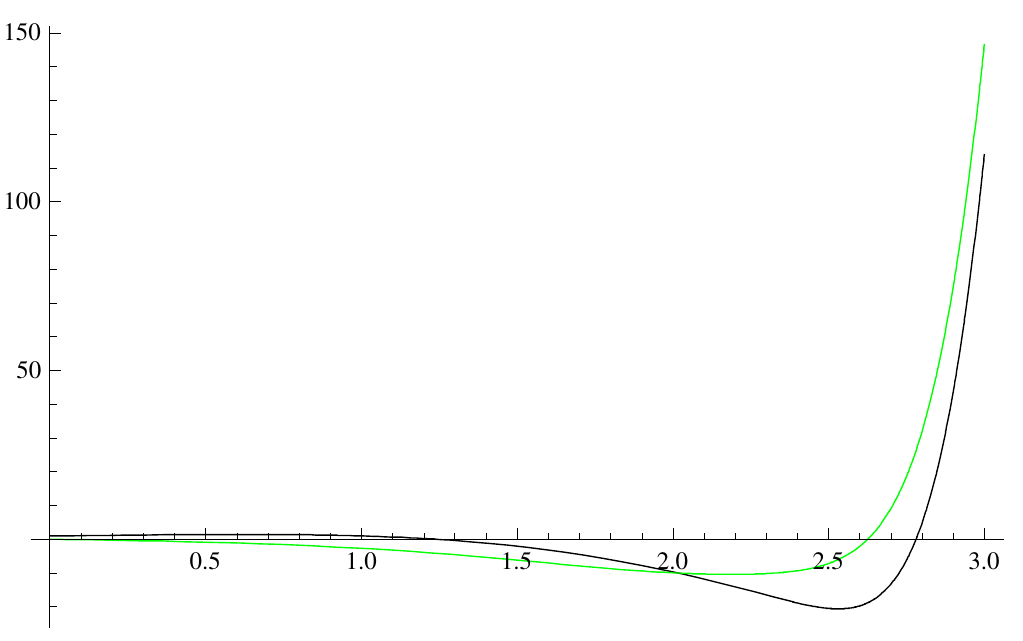}\quad \includegraphics[height=38mm, width=50mm]{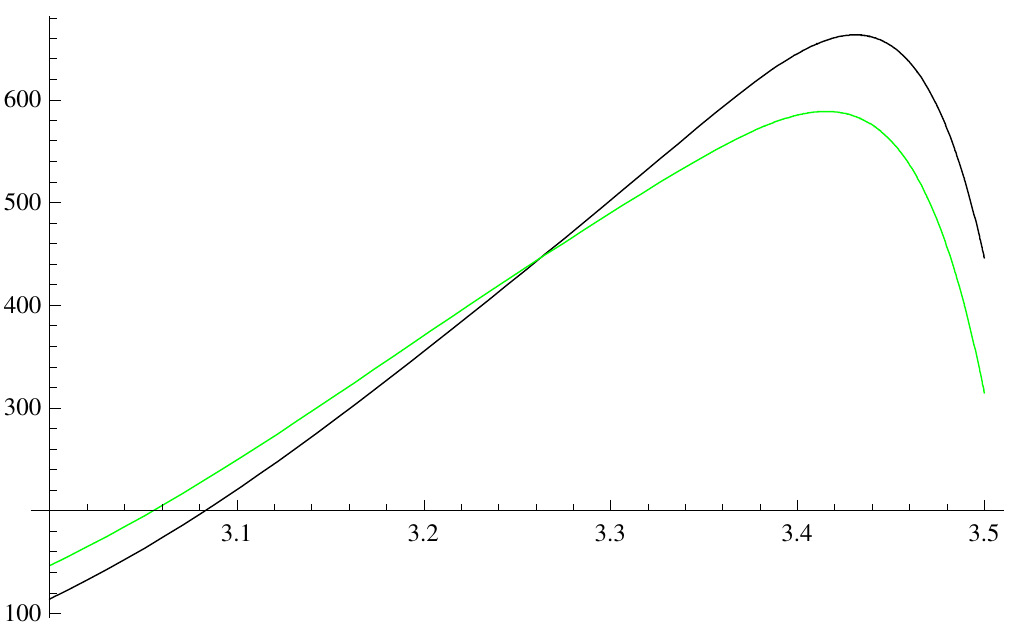}\quad
\includegraphics[height=38mm, width=50mm]{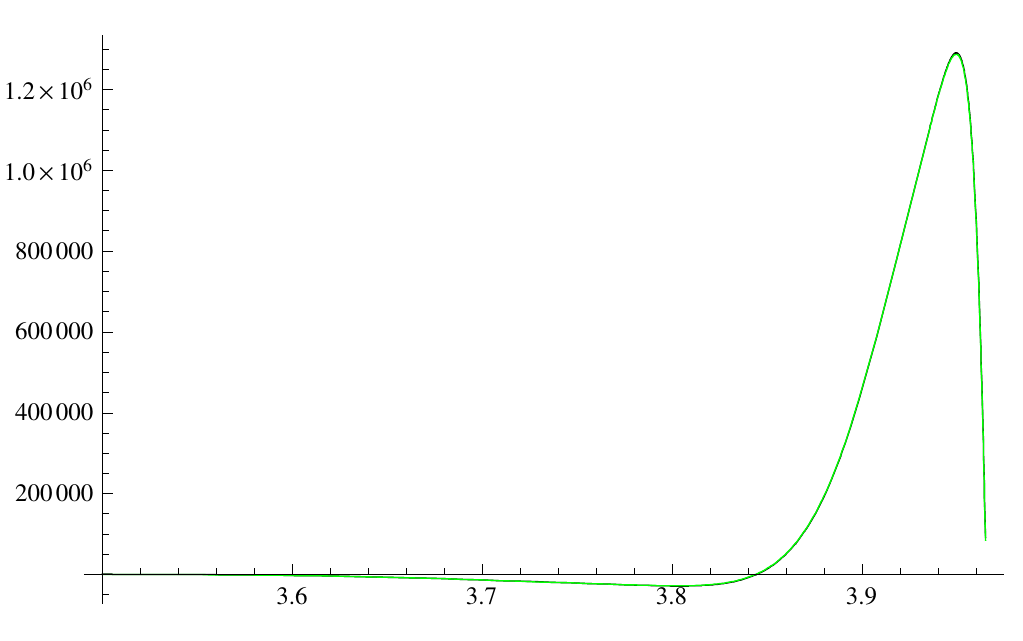}}
\caption{The solution $x^\eps$ (black) and $y^\eps$ (green) to \eq{feps} for $\eps=0.1$.}\label{plot1}
\end{center}
\end{figure}
When $\eps=0.1$ the blow up seems to occur at $R_\eps=4.041$. Notice that $x^\eps$ and $y^\eps$ ``tend to become the same'', in the third picture
they are indistinguishable. After some time, when wide oscillations amplifies, $x^\eps$ and $y^\eps$ move almost synchronously. When $\eps=0$, the solution
to \eq{feps} is explicitly given by $x^0(t)=e^t\cos(t)$ and $y^0(t)=-e^t\sin(t)$, thereby displaying oscillations blowing up in infinite time similar to
those visible in \eq{scann}.\par
If we replace the Cauchy problem in \eq{feps} with
$$x(0)=1,\ y(0)=0,\ x'(0)=-1,\ y'(0)=1$$
then \eq{initial} is not fulfilled. However, for any $\eps>0$ that we tested, the corresponding numerical solutions looked like in Figure \ref{plot1}.
In this case, the limit problem with $\eps=0$ admits as solutions $x^0(t)=e^{-t}\cos(t)$ and $y^0(t)=e^{-t}\sin(t)$ which do exhibit oscillations but,
now, strongly damped.\par
Let us also consider the two remaining limit systems which, however, do not display oscillations.\par
$\bullet$ If $\Delta=0$, since $A\le0$, there are no trigonometric functions in the limit case \eq{f0}.\par
$\bullet$ If $\Delta>0$, then necessarily $A<0$ since $B>0$, and hence only exponential functions are involved: the solution to \eq{f0} may blow up
in infinite time or vanish at infinity.\par\smallskip
The above results explain why we believe that \eq{scann} is not suitable to display self-excited oscillations as the ones which appeared for the TNB.
Since it has only two degrees of freedom, it fails to consider both vertical and torsional oscillations which, on the contrary, are visible in the
McKenna-type system \eq{miosystxy}. We have seen in Theorem \ref{oscill} that destructive self-excited oscillations may blow up in finite time, something very
similar to what may be observed in \cite{tacoma}. Hence, \eq{miosystxy} shows more realistic self-excited oscillations than \eq{scann}.\par
Although the blow up occurs at $t=4.04$, the solution plotted in Figure \ref{plot1} is relatively small until $t=3.98$.
This, together with the behavior displayed in Figures \ref{duemila} and \ref{mille}, allows us to conclude that
\begin{center}
\mbox{\bf in nonlinear systems, self-excited oscillations appear suddenly, without any intermediate stage.}
\end{center}
The material presented in this section also enables us to conclude that
\begin{center}
\mbox{\bf the linear case cannot be seen as a limit situation of the nonlinear case}
\end{center}
since the behavior of the solution to \eq{f0} depends on $\beta$, $\delta$, and on the initial conditions, while nothing can be deduced from the sequence of
solutions $(x^\eps,y^\eps)$ to problem \eq{fe} as $\eps\to0$ because these solutions all behave similarly independently of $\beta$ and $\delta$. Furthermore,
the solutions to the limit problem \eq{f0} may or may not exhibit oscillations and if they do, these oscillations may be both of increasing amplitude or of vanishing
amplitude as $t\to+\infty$. All this shows that linearisation may give misleading and unpredictable answers.\par\smallskip
In this section we have seen that the blow up of solutions to \eq{maineq2} and to \eq{miosystxy} occurs with wide oscillations after a long time of apparent calm.
Hence, the solution does not display any visible behavior which may warn some imminent danger. The reason is that second derivatives of the solution to
\eq{maineq2} blow up at a higher rate, see Theorem \ref{asymptotics}, and second derivatives are not visible by simply looking at the graph.
Wide oscillations after a long time of apparent calm suggest that some hidden energy is present in the system.
Finally, we have seen that self-excited oscillating blow up also appears for a wide class of superlinear fourth order differential equations, including PDE's.

\section{Affording an explanation in terms of energies}\label{afford}

\subsection{Energies involved}\label{energies}

The most important tools to describe any structure are the energies which appear. A precise description of all the energies involved would lead to
perfect models and would give all the information to make correct projects. Unfortunately, bridges, as well as many other structures, do not allow simple
characterisations of all the energies present in the structure and, maybe, not all possible existing energies have been detected up to nowadays.
Hence, it appears impossible to make a precise list of all the energies involved in the complex behavior of a bridge.\par
The kinetic energy is the simplest energy to be described. If $v(x,t)$ denotes the vertical displacement at $x\in\Omega$ and at $t>0$, then
the total kinetic energy at time $t$ is given by
$$\frac{m}{2}\int_\Omega v_t(x,t)^2\, dx$$
where $m$ is the mass and $\Omega$ can be either a segment (beam model) or a thin rectangle (plate model). This energy gives rise to
the term $mv_{tt}$ in the corresponding Euler-Lagrange equation, see Section \ref{models}.\par
Then one should consider potential energy, which is more complicated. From \cite[pp.75-76]{bleich}, we quote
\begin{center}
\begin{minipage}{162mm}
{\em The potential energy is stored partly in the stiffening frame in the form of elastic energy due to bending and partly in the cable in the form
of elastic stress-strain energy and in the form of an increased gravity potential.}
\end{minipage}
\end{center}
Hence, an important role is played by stored energy. Part of the stored energy is potential energy which is quite simple to determine: in order to avoid
confusion, in the sequel we call potential energy only the energy due to gravity which, in the case of a bridge, is computed in terms of the vertical
displacement $v$. However, the dominating part of the stored energy in a bridge is its elastic energy.\par
The distinction between elastic and potential stored energies, which in our opinion appears essential,
is not highlighted with enough care in \cite{bleich} nor in any subsequent treatise of suspension bridges.
A further criticism about \cite{bleich} is that it often makes use of ${\cal LHL}$, see \cite[p.214]{bleich}. Apart these two weak points,
\cite{bleich} makes a quite careful quantitative analysis of the energies involved. In particular, concerning the elastic energy, the contribution
of each component of the bridge is taken into account in \cite{bleich}: the chords (p.145), the diagonals (p.146), the cables (p.147), the towers
(pp.164-168), as well as quantitative design factors (pp.98-103).\par
A detailed energy method is also introduced at p.74, as a practical tool
to determine the modes of vibrations and natural frequencies of suspension bridges: the energies considered are expressed in terms of the amplitude
of the oscillation $\eta=\eta(x)$ and therefore, they do not depend on time.
As already mentioned, the nonlocal term in \eq{eqqq} represents the increment of energy due to the external wind during a period of time.
Recalling that $v(x,t)=\eta(x)\sin(\omega t)$, \cite[p.28]{bleich} represents the net energy input per cycle by
\neweq{dissipation}
A:=\frac{w^2}{H_w^2}\, \frac{EA}{L}\int_0^L \eta(z)\, dz-C\int_0^L \eta(z)^2\, dz
\endeq
where $L$ is the length of the beam and $C>0$ is a constant depending on the frequency of oscillation and on the damping coefficient, so that the second
term is a quantum of energy being dissipated as heat: mechanical hysteresis, solid friction damping, aerodynamic damping, etc. It is explained in
Figure 13 in \cite[p.33]{bleich} that
\begin{center}
\begin{minipage}{162mm}
{\em the kinetic energy will continue to build up and therefore the amplitude will continue to increase until $A=0$.}
\end{minipage}
\end{center}
Hence, the larger is the input of energy $\int_0^L \eta$ due to the wind, the larger needs to be the displacement $v$ before the kinetic energy will stop
to build up. This is related to \cite[pp.241-242]{bleich}, where an attempt is made
\begin{center}
\begin{minipage}{162mm}
{\em to approach by rational analysis the problem proper of self-excitation of vibrations in truss-stiffened suspension bridges. ...
The theory discloses the peculiar mechanism of catastrophic self-excitation in such bridges...}
\end{minipage}
\end{center}

The word ``self-excitation'' suggests behaviors similar to \eq{pazzo}. As shown in \cite{gazpav3}, the oscillating blow up of solutions
described by \eq{pazzo} occurs in many fourth order differential equations, including PDE's, see also Remark \ref{pde}, whereas it does not occur in
lower order equations. But these oscillations, and the energy generating them, are somehow hidden also in fourth order equations; let us explain
qualitatively what we mean by this. Engineers usually say that {\em the wind feeds into the structure an increment of energy} (see \cite[p.28]{bleich})
and that {\em the bridge eats energy} but we think it is more appropriate to say that {\bf the bridge ruminates energy}. That is, first the
bridge stores the energy due to prolonged external sources. Part of this stored energy is indeed dissipated (eaten) by the structural damping
of the bridge. From \cite[p.211]{bleich}, we quote
\begin{center}
\begin{minipage}{162mm}
{\em Damping is dissipation of energy imparted to a vibrating structure by an exciting force, whereby a portion of the external energy is transformed
into molecular energy.}
\end{minipage}
\end{center}
Every bridge has its own damping capacity defined as the ratio between the energy dissipated in one cycle of oscillation and the maximum energy
of that cycle. The damping capacity of a bridge depends on several components such as elastic hysteresis of the structural material and friction
between different components of the structure, see \cite[p.212]{bleich}. A second part of the stored energy becomes potential energy if the bridge
is above its equilibrium position. The remaining part of the stored energy, namely the part exceeding the damping capacity plus the potential
energy, is stored into inner elastic energy; only when this stored elastic energy reaches a critical threshold (saturation),
the bridge starts ``ruminating'' energy and gives rise to torsional or more complicated oscillations.\par
When \eq{pazzo} occurs, the estimate \eq{estimate} shows that $|w''(t)|$ blows up at a higher rate when compared to $|w(t)|$ and $|w'(t)|$.
Although any student is able to see if a function or its first derivative are large just by looking at the graph, most people are unable to see if the
second derivative is large. Roughly speaking, the term $\int w''(t)^2$ measures the elastic energy, the term $\int w'(t)^2$ measures the kinetic
energy, whereas $\int w(t)^2$ is a measure of the potential energy due to gravity. Hence, \eq{estimate} states that the elastic energy has a higher rate
of blow up when compared to the kinetic and potential energies; equivalently, we can say that both the potential energy, described by $|w|$,
and the kinetic energy, described by $|w'|$, are negligible with respect to the elastic energy, described by $|w''|$.
But since large $|w''(t)|$ cannot be easily detected, the bridge may have large elastic energy, and hence large total energy, without revealing it.
Since $|w(t)|$ and $|w'(t)|$ blow up later than $|w''(t)|$, the total energy can be very large without being visible; this is what we mean by
hidden elastic energy. This interpretation well agrees with the numerical results described in Section \ref{blup} which show that blow up
in finite time for \eq{maineq2} occurs after a long waiting time of apparent calm and sudden wide oscillations. Since the apparent calm suggests low
energy whereas wide oscillations suggest high energy, this means that some hidden energy is indeed present in the bridge. And the stored elastic
energy, in all of its forms, seems the right candidate to be the hidden energy. Summarising, the large elastic energy $|w''(t)|$ hides the blow
up of $|w(t)|$ for some time. Then, with some delay but suddenly, also $|w(t)|$ becomes large. If one could find a simple way to measure $|w''(t)|$
which is an approximation of the elastic energy or, even better, $|w''(t)|/\sqrt{1+w'(t)^2}$ which is the mean curvature, then one would have some time to
prevent possible collapses.\par
A flavor of what we call hidden energy was already present in \cite{bleich} where the energy storage capacity of a bridge is often discussed, see
(p.34, p.104, p.160, p.164) for the storage capacity of the different vibrating components of the bridge.
Moreover, the displayed comment just before \eq{coupled} shows that McKenna-Tuama \cite{mckO} also had the feeling that some energy could be hidden.

\subsection{Energy balance}\label{energybalance}

As far as we are aware, the first attempt for a precise quantitative energy balance in a beam representing a suspension bridge was made
in \cite[Chapter VII]{tac2}. Although all the computations are performed with precise values of the constants, in our opinion the analysis there
is not complete since it does not distinguish between different kinds of potential energies; what is called potential energy is just the energy
stored in bending a differential length of the beam.\par
A better attempt is made in \cite[p.107]{bleich} where the plot displays the behavior of the stored energies: the potential energy due to gravity
and the elastic energies of the cables and of the stiffening frame. Moreover, the important notion of flutter speed is first used.
Rocard \cite[p.185]{rocard} attributes to Bleich \cite{bleichsolo}
\begin{center}
\begin{minipage}{162mm}
{\em ... to have pointed out the connection with the flutter speed of aircraft wings ... He distinguishes clearly between flutter and the effect of the staggered
vortices and expresses the opinion that two degrees of freedom (bending and torsion) at least are necessary for oscillations of this kind.}
\end{minipage}
\end{center}
A further comment on \cite{bleichsolo} is given at \cite[p.80]{wake}:
\begin{center}
\begin{minipage}{162mm}
{\em ... Bleich's work ... ultimately opened up a whole new field of study. Wind tunnel tests on thin plates suggested that higher wind velocities increased
the frequency of vertical oscillation while decreasing that of torsional oscillation.}
\end{minipage}
\end{center}
The conclusion is that when the two frequencies corresponded, a flutter critical velocity was reached, as manifested in a potentially catastrophic coupled
oscillation. In order to define the flutter speed, \cite[pp.246-247]{bleich} assumes that the bridge is subject to a natural steady state oscillating motion;
the flutter speed is then defined by:
\begin{center}
\begin{minipage}{162mm}
{\em With increasing wind speed the external force necessary to maintain the motion at first increases and then decreases until a point is reached where
the air forces alone sustain a constant amplitude of the oscillation. The corresponding velocity is called the critical velocity or flutter speed.}
\end{minipage}
\end{center}
The importance of the flutter speed is then described by
\begin{center}
\begin{minipage}{162mm}
{\em Below the critical velocity $V_c$ an exciting force is necessary to maintain a steady-state motion; above the critical velocity the direction of the
force must be reversed (damping force) to maintain the steady-state motion. In absence of such a damping force the slightest increase of the velocity
above $V_c$ causes augmentation of the amplitude.}
\end{minipage}
\end{center}
This means that self-excited oscillations appear as soon as the flutter speed is exceeded.\par
Also Rocard devotes a large part of \cite[Chapter VI]{rocard} to
\begin{center}
\begin{minipage}{162mm}
{\em ... predict and delimit the range of wind speeds that inevitably produce and maintain vibrations of restricted amplitudes.}
\end{minipage}
\end{center}
This task is reached by a careful study of the natural frequencies of the structure. Moreover, Rocard aims to
\begin{center}
\begin{minipage}{162mm}
{\em ... calculate the really critical speed of wind beyond which oscillatory instability is bound to arise and will always cause fracture.}
\end{minipage}
\end{center}
The flutter speed $V_c$ for a bridge without damping is computed on \cite[p.163]{rocard} and reads
\neweq{speedflutter}
V_c^2=\frac{2r^2\ell^2}{2r^2+\ell^2}\, \frac{\omega_T^2-\omega_B^2}{\alpha}
\endeq
where $2\ell$ denotes the width of the roadway, see Figure \ref{9}, $r$ is the radius of gyration, $\omega_B$ and $\omega_T$ are the lowest modes
circular frequencies of the bridge in bending and torsion respectively, $\alpha$ is the mass of air in a unit cube divided by the mass of steel and concrete
assembled within the same unit length of the bridge; usually, $r\approx\ell/\sqrt{2}$ and $\alpha\approx0.02$. More complicated formulas for the flutter speed
are obtained in presence of damping factors. Moreover, Rocard \cite[p.158]{rocard} shows that, for the original Tacoma Bridge, \eq{speedflutter} yields $V_c=47$mph
while the bridge collapsed under the action of a wind whose speed was $V=42$mph; he concludes that his computations are quite reliable.\par
In pedestrian bridges, the counterpart of the flutter speed is the {\em critical number of pedestrians}, see the quoted sentence by
Macdonald \cite{macdonald} in Section \ref{story} and also \cite[Section 2.4]{franck}. For this reason, in the sequel we prefer to deal with energies rather
than with velocities: the flutter speed $V_c$ corresponds to a {\bf critical energy threshold} $\overline{E}$ above which the bridge displays self-excited
oscillations. We believe that
\begin{center}
\begin{minipage}{162mm}
{\bf The critical energy threshold, generated by the flutter speed for suspension bridges and by the critical number of pedestrians for footbridges,
is the threshold where the nonlinear behavior of the bridge really appears, due to sufficiently large displacements of the roadway from equilibrium.}
\end{minipage}
\end{center}
The threshold $\overline{E}$ depends on the elasticity of the bridge, namely on the materials used for its construction. This is in accordance with the numerical
results obtained in \cite{gazpav3} where it is shown that the blow up time for solutions to \eq{maineq2} depends increasingly on the parameter $k$.
We refer to Section \ref{howto2} for a more precise definition of $\overline{E}$ and a possible way to determine it.\par
In this section, we attempt a qualitative energy balance involving more kinds of energies. A special role is played by
the elastic energy which should be distinguished form the potential energy that, as we repeatedly said, merely denotes the potential energy due to gravity;
its level zero is taken in correspondence of the equilibrium position of the roadway.\par
Let us first describe what we believe to happen in a single cross section $\Gamma$ of the bridge; let $\E$ denote its total energy.
Let $A$, $B$, and $C$ denote the three positions of the endpoint $P$ of $\Gamma$, as described in Figure \ref{1}
\begin{figure}[ht]
\begin{center}
{\includegraphics[height=25mm, width=120mm]{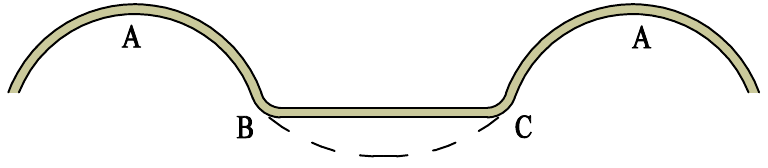}}
\caption{Different positions for the bridge side.}\label{1}
\end{center}
\end{figure}
where the thick grey part denotes the roadway whereas the dotted line displays the behavior of the solution $w$ to \eq{maineq2} when $w<0$,
namely when the deflection from horizontal appears, see \eq{w0}. In what follows, we denote by $A$, $B$, $C$, both the positions in Figure
\ref{1} and the instants of time when they occur for $P$. When $P$ is in its highest position $A$,
$\Gamma$ has maximal potential energy $E_p$ and zero kinetic energy $E_k$: $E_p(A)=\E$, $E_k(A)=0$. In the interval of time $t$ when $P$ goes
from position $A$ to position $B$ the potential and kinetic energies of $\Gamma$ exhibit the well-known behavior with constant sum, see the first picture in Figure \ref{2345}
\begin{figure}[ht]
\begin{center}
{\includegraphics[height=45mm, width=162mm]{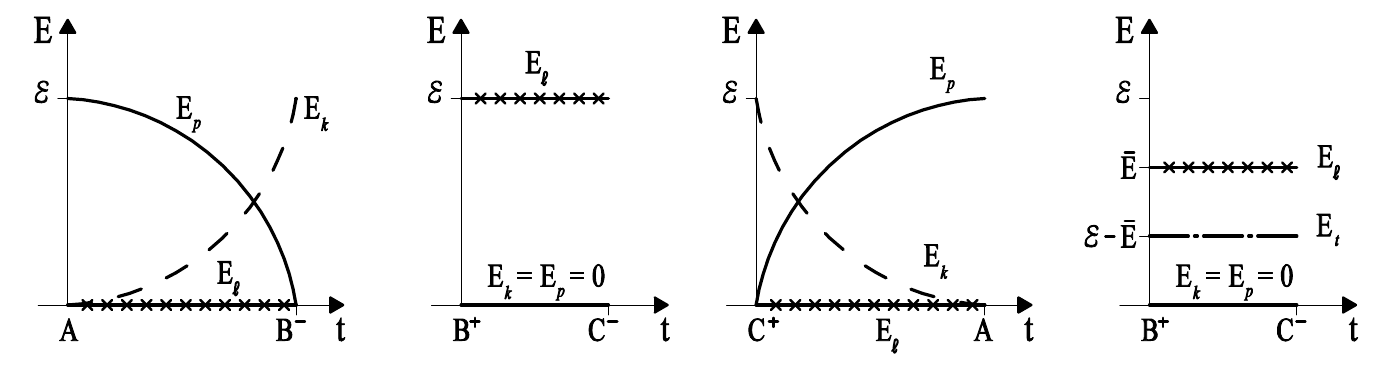}}
\caption{Energy balance for different positions of the bridge.}\label{2345}
\end{center}
\end{figure}
where $E_\ell$ denotes the portion of the stored elastic energy exceeding the structural damping of the bridge: when it is $0$, it means that
there is no stored elastic energy. When $P$ reaches position $B$, corresponding to the maximal elongation of the sustaining hangers, all the
energy has been transformed into kinetic energy: $E_p(B^-)=0$, $E_k(B^-)=\E$.
In this position the falling of $\Gamma$ is violently stopped by the extended hangers by means of a Dirac delta impulse and the existing kinetic
energy of $\Gamma$ is instantaneously stored into elastic energy $E_\ell$: $E_k(B^-)=E_\ell(B^+)$, see the second picture in Figure \ref{2345}.
If the total energy $\E$ is smaller than the critical threshold $\overline{E}$, corresponding to the flutter speed, nothing seems to happen
because the elastic energy is not visible; in this case, after a while it reaches position $C$ and thanks to a further impulse, the elastic energy transforms
back to kinetic energy. Then $P$ starts raising up towards position $A$ and
the sum of the potential and kinetic energies is again constant, see the third picture in Figure \ref{2345}.
In the meanwhile, if the wind keeps blowing or traffic loads generate further negative damping, the total energy $\E$ increases. Hence, after a cycle,
when $P$ is back in position $A$ the total energy $\E$ of $\Gamma$ may have become larger. In turn, $E_k(B^-)$ will also be larger and, after
a certain number of cycles, if the wind velocity is larger than the flutter speed, the total energy exceeds the critical threshold: $\E>\overline{E}$.
In turn, also $E_k(B^-)$ exceeds the critical threshold: $E_k(B^-)>\overline{E}$. When this occurs, the energy splits into two parts: the saturated
elastic energy $E_\ell(B^+)=\overline{E}$ and a torsional elastic energy $E_t(B^+)=\E-\overline{E}$ which immediately gives rise to torsional oscillations,
see the fourth picture in Figure \ref{2345}. As long as $\E>\overline{E}$, when $P$ reaches position $B$ the torsional elastic energy becomes positive
and $P$ ``virtually'' goes from $B$ to $C$ along the dotted line in Figure \ref{1}. In the interval of time when $P$
is between $B^+$ and $C^-$, the torsional elastic energy remains constant and equal to $E_t(C^+)$, see the fourth picture in Figure \ref{2345}.
Only after a further impact, in position $B$, it may vary due to the new impulse. Finally, there exists a second critical threshold: if $\E-\overline{E}$
becomes too large, namely if the total energy $\E$ itself is too large, then the bridge collapses.
\begin{remark}
{\rm With some numerical results at hand, Lazer-McKenna \cite[p.565]{mck1} attempt to explain the Tacoma collapse with the following comment:
\begin{center}
\begin{minipage}{162mm}
{\em An impact, due to either an unusual strong gust of wind, or to a minor structural failure, provided sufficient energy to
send the bridge from one-dimensional to torsional orbits.}\end{minipage}
\end{center}
We believe that what they call {\em an unusual impact} is, in fact, a cyclic impulse for the transition between positions $B^-$ and $B^+$.\endproof}
\end{remark}

\begin{problem} {\em The above energy balance should become quantitative. An exact way to compute all the energies involved should be determined.
Of course, the potential and kinetic energy are straightforward. But the elastic energy needs deeper analysis.}\endproof\end{problem}

Let us now consider the entire bridge which we model as a rectangular plate $\Omega=(0,L)\times(-\ell,\ell)\subset\R^2$. For all $x_1\in(0,L)$ let $\E_{x_1}$ denote
the total energy of the cross section $\Gamma_{x_1}=\{x_1\}\times(-\ell,\ell)$, computed following the above explained method. Then the total energy of the plate
$\Omega$ is given by
\neweq{allsections}
\E_\Omega=\int_0^L\E_{x_1}\, dx_1\ .
\endeq

For simplicity we have here neglected the stretching energy which is a kind of ``interaction energy between cross sections''. If one wishes to consider also this energy,
one usually assumes that the elastic force is proportional to the increase of surface. Then the stretching energy of the horizontal plate $\Omega$ whose vertical deflection
is $u$ reads
$$\E_S(u)=\int_\Omega\left(\sqrt{1+|\nabla u|^2}-1\right) \, dx_1dx_2$$
and, after multiplication by a suitable coefficient, it should be added to $\E_\Omega$. For small deformations $u$ the asymptotic expansion leads to the usual
Dirichlet integral $\frac12 \int_\Omega|\nabla u|^2$ and, in turn, to the appearance of the second order term $\Delta u$ in the corresponding Euler-Lagrange equation.
Clearly, the bridge is better described with the addition of the stretching energy but, at least to have a description of qualitative behaviors, we may neglect it
in a first simplified model.\par\medskip
The just described elastic phenomenon may also be seen in a much simpler model. Imagine there is a ball at some given height above
an horizontal plane $P$, see position $A$ in Figure \ref{10}.
\begin{figure}[ht]
\begin{center}
{\includegraphics[height=45mm, width=153mm]{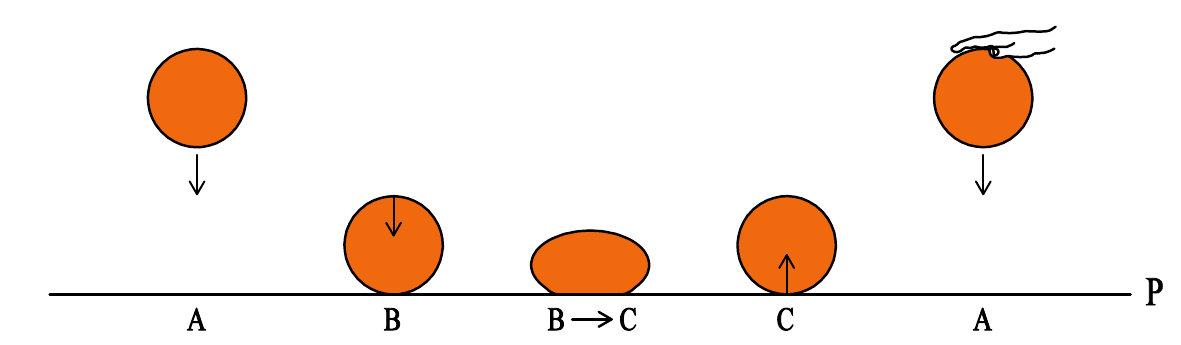}}
\caption{A suspension bridge is like a bouncing ball.}\label{10}
\end{center}
\end{figure}
The ball is falling down until it reaches the position tangent to the plane as in position $B$.
Then there is some positive time where the ball touches $P$; the reason is that it is squeezed and deformed, although probably less than illustrated
in the $B\to C$ picture. But, of course, a very soft ball may have an important deformation. Just after the impact, the ball stores elastic energy which
is hardly visible. After some time, the ball recovers its initial spherical
shape and is ready to bounce up, see position $C$. When it is back in position $A$ it may store further energy, for instance with a hand pushing it
downwards. For these reasons, we believe that there is some resemblance between bouncing balls and oscillating bridges.

\subsection{Oscillating modes in suspension bridges: seeking the correct boundary conditions}\label{modes}

Smith-Vincent \cite[Section I.2]{tac2} analyse the different forms of motion of a suspension bridge and write
\begin{center}
\begin{minipage}{162mm}
{\em The natural modes of vibration of a suspension bridge can be classified as vertical and torsional. In pure vertical modes all points on
a given cross section move vertically the same amount and in phase... The amount of this vertical motion varies along the longitudinal axis of the bridge
as a modified sine curve.}
\end{minipage}
\end{center}
Then, concerning torsional motions, they write
\begin{center}
\begin{minipage}{162mm}
{\em In pure torsional modes each cross section rotates about an axis which is parallel to the longitudinal axis of the bridge and is in the same
vertical plane as the centerline of the roadway. Corresponding points on opposite sides of the centerline of the roadway move equal distances but
in opposite directions.}
\end{minipage}
\end{center}
Moreover, Smith-Vincent also analyse small oscillations:
\begin{center}
\begin{minipage}{162mm}
{\em For small torsional amplitudes the movement of any point is essentially vertical, and the wave form or variation of amplitude
along a line parallel to the longitudinal centerline of the bridge ... is the same as for a corresponding pure vertical mode.}
\end{minipage}
\end{center}

With these remarks at hand, in this section we try to set up a reliable eigenvalue problem. We consider the roadway bridge as a
long narrow rectangular thin plate, simply supported on its short sides. So, let $\Omega=(0,L)\times(-\ell,\ell)\subset\R^2$ where $L$ is
the length of the bridge and $2\ell$ is its width; a realistic assumption is that $2\ell\ll L$.\par
As already mentioned in Section \ref{elasticity}, the choice of the boundary conditions is delicate since it depends on the physical model considered. We first
recall that the boundary conditions $u=\Delta u=0$ are the so-called Navier boundary conditions, see Figure \ref{navierbc} which is taken from \cite[p.96]{navier}.
\begin{figure}[ht]
\begin{center}
{\includegraphics[height=25mm, width=120mm]{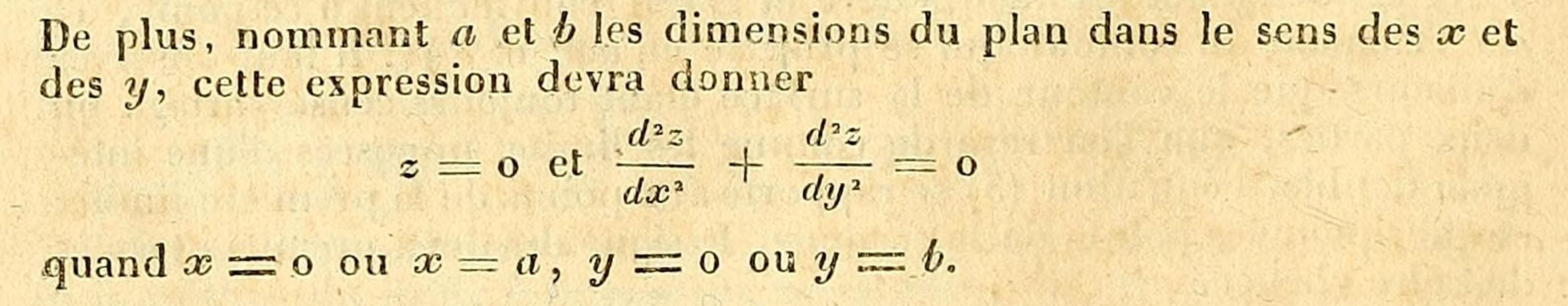}}
\caption{First appearance of Navier boundary conditions.}\label{navierbc}
\end{center}
\end{figure}
On flat parts of the boundary where no curvature is present, they describe simply supported plates, see e.g.\ \cite{gazgruswe}. When $x_1$ is
fixed, either $x_1=0$ or $x_1=L$, these conditions reduce to $u=u_{x_1x_1}=0$. And precisely on these two sides, the roadway $\Omega$ is assumed to be simply
supported; this is uniformly accepted in any of the models we met. The delicate point is the determination of the boundary conditions on the other sides.\par
In order to get into the problem, we start by dealing with the linear Kirchhoff-Love theory described in Section \ref{elasticity}.
In view of \eq{energy-gs}, the energy of the vertical deformation $u$ of a rectangular plate $\Omega=(0,L)\times(-\ell,\ell)$ subject to a load $f=f(x_1,x_2)$ is given by
\neweq{energy-f}
\mathbb{E}(u)=\int_{\Omega }\left(\frac{1}{2}\left( \Delta u\right) ^{2}+(\sigma-1)\det(D^2u)-f\, u\right) \, dx_1dx_2
\endeq
and the corresponding Euler-Lagrange equation reads $\Delta^2u=f$ in $\Omega$. For a fully simply supported plate, that is $u=u_{x_1x_1}=0$ on the vertical sides and
$u=u_{x_2x_2}=0$ on the horizontal sides, this problem has been solved by Navier \cite{navier} in 1823, see also \cite[Section 2.1]{mansfield}. But Cauchy
\cite{cauchy} criticised the work by Navier by claiming that
\begin{center}
\begin{minipage}{162mm}
{\em ... Navier ... avait consid\'er\'e deux esp\`eces de forces produites, les unes par la dilatation ou la contraction, les autres par la flexion de
ce m\^eme plan. ... Il me parut que ces deux esp\`eces de forces pouvaient \^etre r\'eduites \`a une seule... .}
\end{minipage}
\end{center}
Did Cauchy already have in mind the difference/analogy between bending, stretching, and torsion? In any case, since the bridge is not a fully simply
supported plate, different boundary
conditions should be considered on the horizontal sides. The load problem on the rectangle $\Omega$ with only the vertical sides being simply supported was considered by
L\'evy \cite{levy}, Zanaboni \cite{zanaboni}, and Nadai \cite{nadai}, see also \cite[Section 2.2]{mansfield} for the analysis of different kinds of boundary conditions on the
remaining two sides $x_2=\pm\ell$. Let us also mention the more recent monograph \cite[Chapter 3]{ventsel} for a very clear description of bending of rectangular plates.\par
A first natural possibility is to consider the horizontal sides to be free. If no physical constraint is present on the horizontal sides, then the
boundary conditions there become (see e.g.\ \cite[(2.40)]{ventsel})
$$
u_{x_2x_2}(x_1,\pm\ell)+\sigma u_{x_1x_1}(x_1,\pm\ell)=0\, ,\quad u_{x_2x_2x_2}(x_1,\pm\ell)+(2-\sigma)u_{x_1x_1x_2}(x_1,\pm\ell)=0\, ,\quad x_1\in(0,L)\, .
$$
Unfortunately, some physical constraints are present on the two horizontal sides, both because of the action of the hangers and because the cross section of the
bridge may be considered to be rigid. Our purpose is to describe the oscillating modes of the plate $\Omega$ under the most possible realistic boundary conditions; we
suggest here some new conditions to be required on the horizontal sides $x_2=\pm\ell$. Hopefully, they should allow to emphasise both vertical and torsional oscillations.\par
First of all, note that if the cross section of the roadway is rigid and behaves as in Figure \ref{9} then each cross section has constant
deflection from horizontal, that is, it rotates around its barycenter $B$: Prof.\ Farquharson, the man escaping in \cite{tacoma}, ran following
the middle line of the roadway precisely in order to avoid torsional oscillations.
Denoting by $u=u(x_1,x_2)$ the vertical displacement of the roadway, this amounts to say that
\neweq{small}
u_{x_2}(x_1,x_2,t)=\Psi(x_1,t)\qquad(x\in\Omega\, ,\ t>0)
\endeq
for some $\Psi$. If we translate this constraint on the vertical sides $x_2=\pm\ell$ of the plate $\Omega$, we obtain
$$\begin{array}{cc}
u_{x_2x_2}(x_1,x_2)=0\quad x\in(0,L)\times\{-\ell,\ell\}\, ,\\
2\ell u_{x_2}(x_1,-\ell)=2\ell u_{x_2}(x_1,\ell)=u(x_1,\ell)-u(x_1,-\ell)\quad x_1\in(0,L)\, .
\end{array}$$
Indeed, by \eq{small} we have $u_{x_2x_2}\equiv0$ in $\Omega$, which justifies both these conditions.
Taking all the above boundary conditions into account, we are led to consider the following eigenvalue problem
\neweq{eigen1}
\left\{\begin{array}{ll}
\Delta^2 u=\lambda u\quad & x=(x_1,x_2)\in\Omega\, ,\\
u(x_1,x_2)=u_{x_1x_1}(x_1,x_2)=0\quad & x\in\{0,L\}\times(-\ell,\ell)\, ,\\
u_{x_2x_2}(x_1,x_2)=0\quad & x\in(0,L)\times\{-\ell,\ell\}\, ,\\
2\ell u_{x_2}(x_1,-\ell)=2\ell u_{x_2}(x_1,\ell)=u(x_1,\ell)-u(x_1,-\ell)\quad & x_1\in(0,L)\, .
\end{array}\right.
\endeq
This is a nonlocal problem which combines boundary conditions on different parts of $\partial\Omega$.
The oscillating modes of $\Omega$ are the eigenfunctions to \eq{eigen1}. By separating variables, we find the two families of eigenfunctions
\neweq{autofunzione}
\sin\left(\frac{m\pi}{L}x_1\right)\ ,\qquad x_2\sin\left(\frac{m\pi}{L}x_1\right)\qquad(m\in\N\setminus\{0\})\ .
\endeq
The first family describes pure vertical oscillations whereas the second family describes pure torsional oscillations with no vertical displacement
of the middle line of the roadway. For fixed $m$, both these eigenfunctions correspond to the eigenvalue
$$\lambda_m=\frac{m^4\pi^4}{L^4}\ .$$
Although the ``interesting'' eigenfunctions to \eq{eigen1} are the ones in \eq{autofunzione}, further eigenfunctions might exist and are expected to have the form
\neweq{formeigen}
\psi_m(x_2)\, \sin\left(\frac{m\pi}{L}x_1\right)\qquad(m\in\N\setminus\{0\})
\endeq
for some $\psi_m\in C^4(-\ell,ell)$ satisfying a suitable linear fourth order ODE, see \cite{fergaz}.

\begin{problem} {\em Determine all the eigenvalues and eigenfunctions to \eq{eigen1}. Is $\lambda=0$ an eigenvalue? Which subspace of $H^2(\Omega)$
is spanned by these eigenfunctions?}\endproof\end{problem}

It would be interesting to find out if the corresponding loaded plate problem has an equilibrium.

\begin{problem} {\em For any $f\in L^2(\Omega)$ study existence and uniqueness of a function $u\in H^4(\Omega)$ satisfying $\Delta^2 u=f$
in $\Omega$ and \eq{eigen1}$_2$-\eq{eigen1}$_3$-\eq{eigen1}$_4$. Try first some particular forms of $f$ as in \cite[Sections 2.2, 2.2.2]{mansfield}
and then general $f=f(x_1,x_2)$. Is there any reasonable weak formulation for this problem?}\endproof\end{problem}

Problem \eq{eigen1} may turn out to be quite complicated from a mathematical point of view: it is not a variational problem and standard elliptic regularity does not apply.
So, let us suggest an alternative model which seems to fit slightly better in well-known frameworks and also admits \eq{autofunzione} as eigenfunctions.
Consider the eigenvalue problem
\neweq{eigen2}
\left\{\begin{array}{ll}
\Delta^2 u=\lambda u\quad & x=(x_1,x_2)\in\Omega\, ,\\
u(x_1,x_2)=u_{x_1x_1}(x_1,x_2)=0\quad & x\in\{0,L\}\times(-\ell,\ell)\, ,\\
u_{x_2x_2}(x_1,x_2)=u_{x_2x_2x_2}(x_1,x_2)=0\quad & x\in(0,L)\times\{-\ell,\ell\}\, .
\end{array}\right.
\endeq
Here, the condition on the third normal derivative replaces \eq{eigen1}$_4$. This condition somehow ``forces $u_{x_2x_2}$ to remain zero" which is
precisely what happens in the bridge. It is straightforward to verify that \eq{autofunzione} are eigenfunctions to \eq{eigen2} so that similar
problems arise.

\begin{problem} {\em Determine all the eigenvalues and eigenfunctions to \eq{eigen2}. Which subspace of $H^2(\Omega)$ is spanned by these
eigenfunctions? For any $f\in L^2(\Omega)$ study existence and uniqueness of a function $u\in H^4(\Omega)$ satisfying $\Delta^2 u=f$
in $\Omega$ and \eq{eigen2}$_2$-\eq{eigen2}$_3$.}\endproof\end{problem}

A further alternative eigenvalue problem is also of some interest. A possible additional simplification in the model would be to assume that
\neweq{mezzeria}
u(x_1,0,t)\simeq0\qquad\mbox{for all }x_1\in(0,L)\, ,\ t>0\, ,
\endeq
namely that the center line of the roadway has small vertical oscillations. If on one hand this seems realistic in view of \cite{tacoma},
on the other hand this would preclude the appearance of wide vertical oscillations on the center line. In the whole, we believe that, qualitatively,
the behavior of the plate will not change too much. By assuming \eq{small} and that equality holds in \eq{mezzeria}, for all $x_1\in(0,L)$
and $t>0$ we obtain
$$u_{x_2}(x_1,-\ell,t)=u_{x_2}(x_1,\ell,t)=\frac{u(x_1,\ell,t)-u(x_1,-\ell,t)}{2\ell}\ ,\quad u(x_1,\ell,t)=-u(x_1,-\ell,t)\, .$$
By putting these together and decoupling equations on $x_2=\pm\ell$ we arrive at the Robin conditions
$$\ell u_{x_2}(x_1,-\ell,t)+u(x_1,-\ell,t)=0\, ,\ \ell u_{x_2}(x_1,\ell,t)-u(x_1,\ell,t)=0\quad(x_1\in(0,L)\, ,\ t>0)\, .$$
We believe that these boundary conditions may help to obtain oscillation properties also of the solutions to equations derived without assuming \eq{mezzeria}.

\begin{problem} {\em Determine the eigenvalues $\lambda$ and the properties of the eigenfunctions to the following problem
$$\left\{\begin{array}{ll}
\Delta^2 u=\lambda u\quad & x\in\Omega\, ,\\
u(x_1,x_2)=u_{x_1x_1}(x_1,x_2)=0\quad & (x_1,x_2)\in\{0,L\}\times(-\ell,\ell)\, ,\\
u_{x_2x_2}(x_1,x_2)=0\quad & (x_1,x_2)\in(0,L)\times\{-\ell,\ell\}\, ,\\
\ell u_{x_2}(x_1,-\ell)+u(x_1,-\ell)=0\quad & x_1\in(0,L)\, ,\\
\ell u_{x_2}(x_1,\ell)-u(x_1,\ell)=0\quad & x_1\in(0,L)\, .\\
\end{array}\right.$$
This would give an idea of what kind of oscillations should be expected in the below model equation \eq{truebeam}.}\endproof\end{problem}

In this section we set up several eigenvalue problems for thin rectangular plates simply supported on two opposite sides. There is no evidence on which could be
the best boundary conditions on the remaining sides. Once these are determined, it could be of great interest to have both theoretical and numerical information
on the the behavior of eigenvalues and eigenfunctions. The conditions should be sought in order to perfectly describe the oscillating modes of suspension bridges.
In a work in preparation \cite{fergaz} we tackle these problems.

\subsection{Seeking the critical energy threshold}\label{howto2}

The title of this section should not deceive. We will not give a precise method how to determine the energy threshold which gives rise to torsional oscillations
in a plate. We do have an idea how to proceed but several steps are necessary before reaching the final goal.\par
Consider the plate $\Omega=(0,L)\times(-\ell,\ell)$ and the incomplete eigenvalue problem
\neweq{incomplete}
\left\{\begin{array}{ll}
\Delta^2 u=\lambda u\quad & x\in\Omega\, ,\\
u(x_1,x_2)=u_{x_1x_1}(x_1,x_2)=0\quad & (x_1,x_2)\in\{0,L\}\times(-\ell,\ell)\, .
\end{array}\right.
\endeq
Problem incomplete lacks conditions on the remaining sided $x_2=\pm\ell$ and, as mentioned in the previous section, it is not clear which boundary conditions
should be added there.\par
In order to explain which could be the method to determine the critical energy threshold, consider the simple case where the plate is square,
$\Omega=(0,\pi)\times(-\frac{\pi}{2},\frac{\pi}{2})$, and let us complete \eq{incomplete} with the ``simplest'' boundary conditions, namely the Navier boundary
conditions which represent a fully simply supported plate. As already mentioned these are certainly not the correct boundary conditions for a bridge but they are quite helpful
to describe the method we are going to suggest. For different and more realistic boundary conditions, we refer to the paper in preparation \cite{fergaz}.
So, consider the problem
\neweq{irrational}
\left\{\begin{array}{ll}
\Delta^2 u=\lambda u\quad & x\in\Omega\, ,\\
u(x_1,x_2)=u_{x_1x_1}(x_1,x_2)=0\quad & (x_1,x_2)\in\{0,\pi\}\times(-\frac{\pi}{2},\frac{\pi}{2})\, ,\\
u(x_1,x_2)=u_{x_2x_2}(x_1,x_2)=0\quad & (x_1,x_2)\in(0,\pi)\times\{-\frac{\pi}{2},\frac{\pi}{2}\}\ .
\end{array}\right.
\endeq
It is readily seen that, for instance, $\lambda=625$ is an eigenvalue for \eq{irrational} and that there are 4 linearly independent corresponding eigenfunctions
\neweq{exxample}
\{\sin(24x_1)\cos(7x_2),\, \sin(20x_1)\cos(15x_2),\, \sin(15x_1)\cos(20x_2),\, \sin(7x_1)\cos(24x_2)\}\ .
\endeq
It is well-known that similar facts hold for the second order eigenvalue problem $-\Delta u=\lambda u$ on the square, so what we are discussing is not surprising.
What we want to emphasise here is that, associated to the same eigenvalue $\lambda=625$, we have 4 different kinds of vibrations in the $x_1$-direction and
each one of these vibrations has its own counterpart in the $x_2$-direction corresponding to torsional oscillations. We believe that this will be true
for any boundary conditions on $x_2=\pm\ell$ completing \eq{incomplete} and for any values of $L$ and $\ell$. We refer again to Figure \ref{patterns} for
the patterns of some vibrating plates.\par
Consider now a general plate $\Omega=(0,L)\times(-\ell,\ell)$ and let $f\in L^2(\Omega)$; in view of \cite[Section 2.2]{mansfield}, we expect the solution to the problem
\neweq{withsource}
\left\{\begin{array}{ll}
\Delta^2 u=f\quad & x\in\Omega\, ,\\
u(x_1,x_2)=u_{x_1x_1}(x_1,x_2)=0\quad & (x_1,x_2)\in\{0,\pi\}\times(-\frac{\pi}{2},\frac{\pi}{2})\, ,\\
\mbox{other boundary conditions}\quad & (x_1,x_2)\in(0,\pi)\times\{-\frac{\pi}{2},\frac{\pi}{2}\}\ ,
\end{array}\right.
\endeq
to be of the kind
$$u(x_1,x_2)=\sum_{m=1}^\infty \psi_m(x_2)\sin\left(\frac{m\pi}{L}x_1\right)\qquad(x_1,x_2)\in\Omega$$
for some functions $\psi_m$ depending on the Fourier coefficients of $f$. Since we have in mind small $\ell$, we can formally expand $\psi_m$ in Taylor polynomials and obtain
$$\psi_m(x_2)=\psi_m(0)+\psi_m'(0)x_2+o(x_2)\qquad\mbox{as }x_2\to0\ .$$
Hence, $u$ may approximately be written as a combination of the functions in \eq{autofunzione}:
$$u(x_1,x_2)\approx\sum_{m=1}^\infty[a_m+b_mx_2]\sin\left(\frac{m\pi}{L}x_1\right)\qquad(x_1,x_2)\in\Omega$$
where $a_m=\psi_m(0)$ and $b_m=\psi'_m(0)$. If instead of a stationary problem such as \eq{withsource}, $u=u(x_1,x_2,t)$ satisfies an evolution problem with the
same boundary conditions, then also its coefficients depend on time:
\neweq{fourier}
u(x_1,x_2,t)\approx\sum_{m=1}^\infty \Big(a_m(t)+b_m(t)x_2\Big)\sin\left(\frac{m\pi}{L}x_1\right)\qquad(x_1,x_2)\in\Omega\, ,\ t>0\ .
\endeq

Let now $\E(t)$ denote the instantaneous total energy of the bridge, as determined in \eq{allsections}. What follows
is not precise, it is a qualitative attempt to describe combined vertical and torsional oscillations. In particular, due to the restoring cables and hangers,
the sine functions in \eq{fourier} should be modified in order to display different behaviors for positive and negative arguments. Moreover, we call ``small'' any quantity
which is less than unity and ``almost zero'' (in symbols $\cong0$) any quantity which has a smaller order of magnitude when compared with small quantities.
Finally, in order to avoid delicate sign arguments, we will often refer to $a_m^2$ and $b_m^2$ instead of $a_m$ and $b_m$.\par\smallskip
$\bullet$ {\bf Small energy.} As long as $\E(t)$ is small one may not even see oscillations, but if somebody stands on the bridge he might be
able to feel oscillations. For instance, standing on the sidewalk of a bridge, one can feel the oscillations created by a car going through the
roadway but the oscillations will not be visible to somebody watching the roadway from a point outside the bridge. For small energies
$\E(t)$ only small oscillations appear and the corresponding solution \eq{fourier} has small coefficients $a_m(t)$ while $b_m(t)\cong0$. More precisely,
\neweq{ambm}
\left\{\begin{array}{ll}
\forall\eps>0\ \exists\delta>0\quad\mbox{such that}\quad\E(t)<\delta\ \Longrightarrow\ a_m(t)^2<\eps\ \forall m\, ,\\
\exists\gamma>0\quad\mbox{such that}\quad\E(t)<\gamma\ \Longrightarrow\ b_m(t)\cong0\ \forall m\, .
\end{array}\right.
\endeq
The reason of the second of \eq{ambm} is that even small variations of the $b_m$'s correspond to a large variation of the total energy $\E$
because the huge masses of the cross sections would rotate along the large length $L$ of the roadway. On the other hand, the first
of \eq{ambm} may be strengthened by assuming that also some of the $a_m$'s are almost zero for small $\E$; in particular, we expect that this
happens for large $m$ since these coefficients correspond to higher eigenvalues
\neweq{vanishlargea}
\forall\overline{m}\in\N\setminus\{0\}\quad \exists\E_{\overline{m}}>0
\quad\mbox{such that}\quad\E(t)<\E_{\overline{m}}\ \Longrightarrow\ a_m(t)\cong0\quad \forall m>\overline{m}\ .
\endeq
To better understand this point, let us compute the elongation $\Gamma_m$ due to the $m$-th mode:
\neweq{elongation}
\Gamma_m(t):=\int_0^L\bigg(\sqrt{1+\frac{m^2\pi^2}{L^2}\, a_m(t)^2\, \cos^2\left(\frac{m\pi}{L}x_1\right)}\, -\, 1\bigg)\, dx_1\ ;
\endeq
this describes the stretching elastic energy since it is the difference between the length of the roadway deformed by one single mode
and the length of the roadway at rest. Due to the coefficient $\frac{m^2\pi^2}{L^2}$, it is clear that if $a_m^2\equiv a_{m+1}^2$ then
$\Gamma_m(t)<\Gamma_{m+1}(t)$. This is the reason why \eq{vanishlargea} holds.\par\smallskip
$\bullet$ {\bf Increasing energy.} According to \eq{vanishlargea}, as long as $\E(t)<\E_{\overline{m}}$ one has
$a_m(t)\cong0$ for all $m>\overline{m}$. If $\E(t)$ increases but remains smaller than $\E_{\overline{m}}$, then the coefficients $a_m(t)^2$
for $m=1,...,\overline{m}$ also increase. But they cannot increase to infinity since \eq{elongation} implies that the length of the
roadway would also increase to infinity. So, when the total energy $\E(t)$ reaches the threshold $\E_{\overline{m}}$ the superlinear elastic structure
of the bridge forces the solution \eq{fourier} to add one mode, so that $a_{\overline{m}+1}(t)\not\cong0$. Hence, the number of modes $\not\cong0$ is
a nondecreasing function of $\E$.\par\smallskip
$\bullet$ {\bf Critical energy threshold.} What is described above is purely theoretical, but the bridge has several physical constraints.
Of course, it cannot be stretched to infinity, it will break down much before. In particular, the number of active modes cannot increase
to infinity. The elastic properties of the bridge determine a critical (maximal) number of possible active modes, say $\mu$. If the energy
is distributed on the $\mu$ coefficients $a_1$,...,$a_\mu$, and if it increases up to $\E_\mu$, once again the superlinear elastic structure
of the bridge forces the solution \eq{fourier} to change mode, but this time from the $a_m$ to the $b_m$; due to \eq{elongation},
further stretching of the roadway would require much more energy than switching oscillations on torsional modes.
The switch is due to an impulse caused by the instantaneous stopping of the falling roadway imposed by the sustaining cables and the elongated hangers.
And which torsional modes will be activated depends on which coupled modes have the same eigenvalue; as an example, consider \eq{exxample} which, roughly speaking,
says that the motion may change from $24$ to $7$ oscillations in the $x_1$-direction with a consequent change of oscillation also in the $x_2$-direction.\par\smallskip
$\bullet$ {\bf Summarising...} Let $u$ in \eq{fourier} describe the vertical displacement of the roadway. The bridge has
several characteristic values which depend on its elastic structure.\par
$\diamondsuit$\quad An integer number $\mu\in\N$ such that $a_m(t)\cong0$ and $b_m(t)\cong0$ for all $m>\mu$, independently of the value of $\E(t)$.\par
$\diamondsuit$\quad $\mu$ different energy ``increasing modes thresholds'' $E_1,...,E_\mu$.\par
$\diamondsuit$\quad The critical energy threshold $\overline{E}=E_\mu$.\par
Assume that $\E(0)=0$, in which case $u(x_1,x_2,0)=0$, and that $t\mapsto\E(t)$ is increasing. As long as $\E(t)\le E_1$ we have $a_m\cong0$
for all $m\ge2$ and $b_m\cong0$ for all $m\ge1$; moreover, $t\mapsto a_1(t)^2$ is increasing. When $\E(t)$ reaches and exceeds $E_1$ there is a first
switch: the function $a_2^2$ starts being positive while, as long as $\E(t)\le E_2$, we still have
$a_m\cong0$ for all $m\ge3$ and $b_m\cong0$ for all $m\ge1$. And so on, until $\E(t)=E_\mu=\overline{E}$. At this point, also because of an impulse,
the energy forces the solution to have a nonzero coefficient $b_1$ rather than a nonzero coefficient $a_{\mu+1}$. The impulse forces $u$ to lower the
number of modes for which $a_m\not\cong0$. For instance, the observation by Farquharson ({\em The motion, which a moment before had involved nine or
ten waves, had shifted to two}) quoted in Section \ref{story} shows that, for the Tacoma Bridge, there was a change such as
\neweq{change}
\Big(a_m\cong0\ \forall m\ge11,\ b_m\cong0\ \forall m\ge1\Big)\ \longrightarrow\ \Big(a_m\cong0\ \forall m\ge3,\ b_m\cong0\ \forall m\ge2\Big)\ .
\endeq

In order to complete the material in this section, two major problems are still to be solved.\par
- Find the correct boundary conditions on $x_2=\pm\ell$.\par
- Find at which energy levels the ``transfer of energy between modes'' occurs, see \eq{change}.\par
Both these problems are addressed in a forthcoming paper \cite{fergaz}.

\section{Conclusions and future perspectives}\label{conclusions}

So far, we observed phenomena displayed by real structures, we discussed models, and we recalled some theoretical results.
In this section we take advantage from all this work, we summarise all the phenomena and remarks previously discussed, and we reach several conclusions.

\subsection{A new mathematical model for suspension bridges}\label{newmodel}

We suggest here a new mathematical model for the description of oscillations in suspension bridges. We expect the solution to
the corresponding equation to display both self-excited oscillations and instantaneous switch between vertical and torsional oscillations.
Moreover, the critical energy threshold (corresponding to the flutter speed) appears in the equation.\par
Let $\Omega=(0,L)\times(-\ell,\ell)\subset\R^2$ where $L$ represents the length of the bridge and $2\ell$ represents the width of the roadway.
Assume that $2\ell\ll L$ and consider the initial-boundary value problem
\neweq{truebeam}
\left\{\begin{array}{ll}
u_{tt}+\Delta^2 u+\delta u_t+f(u)=\varphi(x,t)\ & x=(x_1,x_2)\in\Omega,\ t>0,\\
u(x_1,x_2,t)=u_{x_1x_1}(x_1,x_2,t)=0\ & x\in\{0,L\}\times(-\ell,\ell),\ t>0,\\
u_{x_2x_2}(x_1,x_2,t)=0\ & x\in(0,L)\times\{-\ell,\ell\},\ t>0,\\
u_{x_2}(x_1,-\ell,t)=u_{x_2}(x_1,\ell,t)\ & x_1\in(0,L),\ t>0,\\
u_t(x_1,-\ell,t)+u(x_1,-\ell,t)=E(t)\, [u_t(x_1,\ell,t)+u(x_1,\ell,t)]\ & x_1\in(0,L),\ t>0,\\
u(x,0)=u_0(x)\ & x\in\Omega,\\
u_t(x,0)=u_1(x)\ & x\in\Omega.
\end{array}\right.
\endeq

Here, $u=u(x,t)$ represents the vertical displacement of the plate, $u_0(x)$ is its initial position while $u_1(x)$ is its initial vertical velocity.
Before discussing the other terms and conditions, let us remark that if one wishes to have smooth solutions, a compatibility condition between
boundary and initial conditions is needed:
\neweq{cc}
u_1(x_1,-\ell)+u_0(x_1,-\ell)=E(0)\, [u_1(x_1,\ell)+u_0(x_1,\ell)]\qquad\forall x_1\in(0,L)\ .
\endeq
The function $\varphi$ represents an external source, such as the wind, which is responsible for the variation of the total energy $\E(t)$ inserted in the
structure which, in turn, can be determined by
\neweq{gust}
\E(t)=\int_\Omega \varphi(x,t)^2\, dx\ .
\endeq
Of course, it is much simpler to compute $\E(t)$ in terms of the known source $\varphi$ rather than in terms of the unknown solution $u$
and its derivatives. The function $E$ is then defined by
\neweq{EE}
E(t)=\left\{\begin{array}{ll}
1\quad & \mbox{if }\E(t)\le\overline{E}\\
-1\quad & \mbox{if }\E(t)>\overline{E}
\end{array}\right.
\endeq
where $\overline{E}>0$ is the critical energy threshold defined in Section \ref{energybalance}; hence, the function $E(t)$ is a discontinuous
nonlocal term which switches to $\pm1$ according to whether the total energy is smaller/larger than the critical threshold.
So, \eq{truebeam}$_5$ is a dynamic boundary condition involving the total energy of the bridge.
When $\E(t)\le\overline{E}$ there is no torsional energy $E_t$ and the motion tends to become of pure vertical-type, that is, with $u_{x_2}\cong0$:
to see this, note that in this case \eq{truebeam}$_5$ may be written as
$$\frac{\partial}{\partial t}\Big\{[u(x_1,\ell,t)-u(x_1,-\ell,t)]e^t\Big\}=0\qquad\forall x_1\in(0,L)\ .$$
This means that as long as $\E(t)\le\overline{E}$, the map $t\mapsto|u(x_1,\ell,t)-u(x_1,-\ell,t)|$ decreases so that the two opposite endpoints
of any cross section tend to have the same vertical displacement and to move synchronously as in a pure vertical motion.
When $\E(t)>\overline{E}$ condition \eq{truebeam}$_5$ may be written as
$$\frac{\partial}{\partial t}\Big\{[u(x_1,\ell,t)+u(x_1,-\ell,t)]e^t\Big\}=0\qquad\forall x_1\in(0,L)\ .$$
This means that as long as $\E(t)>\overline{E}$, the map $t\mapsto|u(x_1,\ell,t)+u(x_1,-\ell,t)|$ decreases so that the two opposite endpoints
of any cross section tend to have zero average and to move asynchronously as in a pure torsional motion, that is, with
$u(x_1,0,t)\cong\frac12 [u(x_1,\ell,t)+u(x_1,-\ell,t)]\cong0$.\par
Note that in \eq{truebeam} the jump of $E(t)$ from/to $\pm1$ occurs simultaneously and instantaneously along all the points located on the sides
of the roadway; hence, either all or none of the cross sections have some torsional motion, in agreement with what has been observed for the
Tacoma Bridge, see \cite{Tacoma1} and also \cite[pp.50-51]{wake}. The form \eq{gust} and the switching criterion for $E(t)$ in \eq{EE} mean that
problem \eq{truebeam} models a situation
where if a gust of wind is sufficiently strong then, instantaneously, a torsional motion appears. One could also consider the case where
$$\E(t)\simeq\int_0^t\int_\Omega \varphi(x,\tau)^2\, dx\, d\tau\ $$
which would model a situation where if the wind blows for too long then at some critical time, instantaneously, a torsional motion appears.
However, the problem with \eq{gust} is much simpler because it is local in time.\par
The differential operator in \eq{truebeam} is derived according to the linear Kirchhoff-Love model for a thin plate, see Section \ref{elasticity}.
We have neglected a strong distinction between the bending and stretching elastic energies which are,
however, quite different in long narrow plates, see \cite[Section 8.3]{mansfield} and previous work by Cox \cite{cox}; if one wishes to make some
corrections, one should add a further nonlinear term $g(\nabla u,D^2u)$ and the equation would become quasilinear, see Problem \ref{gDu}.
But, as already mentioned in Section \ref{elasticity}, we follow here a compromise and merely consider a semilinear problem. Concerning the nonlinearity $f(u)$,
some superlinearity should be required. For instance, $f(u)=u+\eps u^3$ with $\eps>0$ small could be a possible choice; alternatively, one could take
$f(u)=a(e^{bu}-1)$ as in \cite{mckO} for some $a,b>0$. In the first case the hangers are sought as ideal springs and gravity is somehow neglected, in the second
case more relevance is given to gravity and to the possibility of slackening hangers. Finally, $\delta u_t$ is a damping term which represents the positive structural
damping of the structure; its role should be to weaken the effect of the nonlinear term $f(u)$, see Problem \ref{competition}.\par
As far as we are aware, there is no standard theory for problems like \eq{truebeam}. It is a nonlocal problem since it links behaviors on different
parts of the boundary and involves the function $E(t)$ in \eq{EE}. It also has dynamic boundary conditions which are usually delicate to handle.

\begin{problem} {\em Prove that if $\varphi(x,t)\equiv0$ then \eq{truebeam} only admits the trivial solution $u\equiv0$. The standard
trick of multiplying the equation in \eq{truebeam} by $u$ or $u_t$ and integrating over $\Omega$ does not allow to get rid of all the boundary terms.
Note that, in this case, $E(t)\equiv1$.}\endproof\end{problem}

\begin{problem}\label{ill} {\em Study existence and uniqueness results for \eq{truebeam}; prove continuous dependence results with respect to the
data $\varphi$, $u_0$, $u_1$, and with respect to possible perturbations of $f$. Of course, the delicate conditions to deal with are \eq{truebeam}$_4$
and \eq{truebeam}$_5$. If problem \eq{truebeam} were ill-posed, what must be changed in order to have a well-posed problem?}\endproof\end{problem}

\begin{problem}\label{competition} {\em Study \eq{truebeam} with no damping, that is, $\delta=0$: does the solution display oscillations such as \eq{pazzo}
when $t$ tends to some finite blow up instant? Then study the competition between the damping term $\delta u_t$ and the self-exciting term $f(u)$: for
a given $f$ is it true that if $\delta$ is sufficiently large then the solution $u$ is global in time? We believe that the answer is negative and that the
only effect of the damping term is to delay the blow up time.}\endproof\end{problem}

\begin{problem} {\em Determine the regularity of the solutions $u$ to \eq{truebeam} and study the importance of the compatibility
condition \eq{cc}.}\endproof\end{problem}

\begin{problem}\label{gDu} {\em Insert into the equation \eq{truebeam} a correction term for the elastic energies, something like
$$g(\nabla u,D^2u)=-\left(\frac{u_{x_1}}{\sqrt{1+u_{x_1}^2}}\right)_{x_1}-\gamma\left(\frac{u_{x_1}}{\sqrt{1+u_{x_1}^2}}\right)_{x_2}$$
with $\gamma>0$ small. Then prove existence, uniqueness and continuous dependence results.}\endproof\end{problem}

An important tool to study \eq{truebeam} would be the eigenvalues and eigenfunctions of the corresponding stationary problem.
In view of the dynamic boundary conditions \eq{truebeam}$_5$, a slightly simpler model could be considered, see \eq{eigen1} or \eq{eigen2}
and subsequent discussion in Section \ref{modes}. We have no feeling on what could be the better choice...

\subsection{A possible explanation of the Tacoma collapse}\label{possibleTacoma}

Hopefully, this paper sheds some further light on oscillating bridges. We have emphasised the necessity of models fulfilling the requirements of (GP) since,
otherwise, the solution will not display the phenomena visible in real bridges. In particular, any equation aiming to model the complex behavior of bridges should
contain information on at least two possible kinds of oscillations: this target may be achieved either by considering a PDE, or by considering coupled systems of ODE's,
or by linking the two oscillations within a unique function solving a suitable ODE. A further contribution of this paper is the remark that there might be some hidden
elastic energy in the bridge and that there is no simple way to detect it. Not only this energy is larger than the kinetic and potential energy but also it may give rise,
almost instantaneously, to self-excited oscillations.\par
We now put together all these observations in order to afford an explanation of the Tacoma collapse.
As we shall see, our explanation turns out to agree with all of them.\par\medskip\noindent
\textsf{On November 7, 1940, for some time before 10:00 AM, the bridge was oscillating as it did many other times before. The wind was apparently more
violent than usual and, moreover, it continued for a long time. The oscillations of the bridge were completely similar to those displayed in the interval
$(0,80)$ of the plot in Figure \ref{mille}. Since the energy involved was quite large, also the oscillations were quite large. The roadway was far
from its equilibrium position and, consequently, the restoring force due to the sustaining cables and to the hangers did not obey ${\cal LHL}$. The oscillations
were governed by a fairly complicated differential equation such as \eq{truebeam} which, however, may be approximated by \eq{maineq2} after assuming \eq{w0},
since this equation possesses the main features of many fourth order differential equations, both ODE's and PDE's. It is not clear which superlinear
function $f$ would better describe the restoring nonlinear elastic forces, but any function $f=f(s)$ asymptotically linear as $s\to0$ and superlinear
as $|s|\to\infty$ generates the same qualitative behavior of solutions, see Theorem \ref{blowup}.
As the wind was keeping on blowing, the total energy $\E$ in the bridge was increasing; the bridge started ruminating energy and, in particular,
its stored elastic energy $E_\ell$ was also increasing. Unfortunately, nobody knew how to measure $E_\ell$ because, otherwise, \eq{estimate}
would have taught in some advance that sudden wider vertical oscillations and subsequent torsional oscillations would have appeared. After each cycle, when
the cross section of the bridge violently reached position $B$ in Figure \ref{1}, a delta Dirac mass increased the internal elastic energy $E_\ell$.
As the wind was continuously blowing, after some cycles the elastic energy became larger than the critical energy threshold $\overline{E}$ of
the bridge, see Section \ref{energybalance}. This threshold may be computed provided one knows the elastic structure of the bridge, see Section
\ref{howto2}. As soon as $E_\ell>\overline{E}$, the function $E(t)$ in \eq{EE} switched from $+1$ to $-1$,
a torsional elastic energy $E_t$ appeared and gave rise, almost instantaneously, to a torsional motion.
As described in \eq{change}, due to the impulse, the energy switched to the first torsional mode $b_1$ rather than to a further
vertical mode $a_{11}$; so, the impulse forced $u$ to lower the number of modes
for which $a_m\not\cong0$ and the motion, which a moment before had involved nine or ten waves, shifted to two. At that moment,
the graph of the function $w=w(t)$, describing the bridge according to \eq{w0}, reached time $t=95$ in the plot in Figure \ref{mille}.
Oscillations immediately went out of control and after some more oscillations the bridge collapsed.}\par\smallskip

One should compare this description with the original one from the Report \cite{Tacoma1}, see also \cite[pp.26-29]{tac1} and \cite[Chapter 4]{wake}.

\subsection{What about future bridges?}\label{howplan}

Equation \eq{maineq2} is a simple prototype equation for the description of self-excited oscillations. None of the previously existing
mathematical models ever displayed this phenomenon which is also visible in oscillating bridges. The reason is not that the behavior of the bridge is
too complicated to be described by a differential equation but mainly because they fail to satisfy (GP); this prevents
the appearance of oscillations and therefore the projects based on the corresponding equation may contain huge errors.
In order to avoid bad surprises as in the past, many projects nowadays include stiffening trusses or strong dampers. This has the great advantage
to maintain the bridge much closer to its equilibrium position and to justify ${\cal LHL}$. But this also has disadvantages, see \cite{kawada2}.
First of all, they create an artificial stiffness which can give rise to the appearances of cracks in the more elastic structure of the bridge.
Second, damping effects and stiffening trusses significantly increase the weight and the cost of the whole structure.
Moreover, in extreme conditions, they may become useless: under huge external solicitations the bridge would again be too far from its equilibrium
position and would violate ${\cal LHL}$. So, hopefully, one should find alternative solutions, see again \cite{kawada2}.\par
One can act both on the structure and on the model. In order to increase the flutter speed, some suggestions on how to modify the design were made
by Rocard \cite[pp.169-173]{rocard}: he suggests how to modify the components of the bridge in order to raise the right hand side of \eq{speedflutter}.
More recently, some attempts to improve bridges performances can be found in \cite{hhs} where, in particular, a careful analysis of the role played by
the hangers is made. But much work has still to be done; from \cite[p.1624]{hhs}, we quote
\begin{center}
\begin{minipage}{162mm}
{\em Research on the robustness of suspension bridges is at the very beginning.}
\end{minipage}
\end{center}

From a theoretical point of view, one should first determine a suitable equation satisfying (GP).
Our own suggestion is to consider \eq{truebeam} where one should choose a reliable nonlinearity $f$ and add coefficients to the other terms,
according to the expected technical features of the final version of the bridge: its length, its width, its weight, the materials used for
the construction, the expected external solicitations, the structural damping... Since we believe that the solution to this equation may display blow up,
which means a strong instability, a crucial role is played by all the constants which appear in the equation. Hence, a careful measurement of these
parameters is necessary. Moreover, a sensitivity analysis for the continuous dependence of the solution on the parameters should be performed.
Once the most reliable nonlinearity and parameters are chosen, the so obtained equation should be tested numerically to see if the solution
displays dangerous phenomena. In particular, one should try to estimate, at least numerically, the critical energy threshold and the possible blow up time. Also
purely theoretical estimates are of great interest; in general, these are difficult to obtain but even if they are not very precise they can be of some help.\par
In a ``perfect model'' for a suspension bridge, one should also take into account the role of the sustaining cables and of the towers. Each cable links all
the hangers on the same side of the bridge, its shape is a catenary of given length and the hangers cannot all elongate at the same time. The towers link
the two cables and, in turn, all the hangers on both sides of the roadway. Cables and towers are further elastic components of the structure which, of course,
modify considerably the model, its oscillating modes, its total elastic energy, etc. In this paper we have not introduced these components but, in the next
future, this would be desirable.\par
An analysis of the more flexible parts of the roadway should also be performed; basically, this consists in measuring the ``instantaneous local energy'' defined by
\neweq{localenergy}
{\bf E}(u(t),\omega)=\int_\omega\left[\left(\frac{|\Delta u(t)|^2}{2}+(\sigma-1)\det(D^2u(t))\right)+\frac{u_t(t)^2}{2}+F(u(t))\right]\, dx_1dx_2
\endeq
for solutions $u=u(t)$ to \eq{truebeam}, for any $t\in(0,T)$, and for any subregion $\omega\subset\Omega$ of given area. In \eq{localenergy} we recognize the first
term to be as in the Kirchhoff-Love model, see \eq{energy-gs}; moreover, $F(s)=\int_0^s f(\sigma)d\sigma$.

\begin{problem} {\em Let $\Omega=(0,L)\times(-\ell,\ell)$, consider problem \eq{truebeam} and let $u=u(t)$ denote its solution provided it exists and is
unique, see Problem \ref{ill}. For given lengths $a<L$ and $b<2\ell$ consider the set $\Re$ of rectangles entirely contained in $\Omega$ and whose
sides have lengths $a$ (horizontal) and $b$ (vertical). Let ${\bf E}$ be as in \eq{localenergy} and consider the maximisation problem
$$\max_{\omega\in\Re}\ {\bf E}(u(t),\omega)\ .$$
Using standard tools from calculus of variations, prove that there exists an optimal rectangle and study its dependence on $t\in(0,T)$; the natural
conjecture is that, at least as $t\to T$, it is the ``middle rectangle'' $(\frac{L-a}{2},\frac{L+a}{2})\times(-\frac{b}{2},\frac{b}{2})$. Then one should find out
if there exists some optimal ratio $a/b$. Finally, it would be extremely useful to find the dependence of the energy ${\bf E}(u(t),\omega)$
on the measure $ab$ of the rectangle $\omega$; this would allow to minimise costs for reinforcing the plate. We do not expect analytical tools to be able
to locate optimal rectangles nor to give exact answers to the above problems so that a good numerical procedure could be of great help.}\endproof\end{problem}

In \cite[Chapter IV]{tac2} an attempt to estimate the impact of stiffening trusses is made, although only one kind of truss design is considered.
In order to determine the best way to display the truss, one should solve the following simplified problems from calculus of variations.
A first step is to consider the linear model.

\begin{problem}\label{pisa} {\em Assume that the rectangular plate $\Omega=(0,L)\times(-\ell,\ell)$ is simply supported on all its sides and that it is
submitted to a constant load $f\equiv1$. In the linear theory by Kirchhoff-Love model, see \eq{energy-f}, its elastic energy is given by
$$E_0(\Omega)=\ -\min_{u\in H^2\cap H^1_0(\Omega)}\ \int_{\Omega }\left(\frac{1}{2}\left(\Delta u\right)^{2}+(\sigma-1)\det(D^2u)-u\right)\, dx_1dx_2\ .$$
Here, $H^2\cap H^1_0(\Omega)$ denotes the usual Hilbertian Sobolev space which, since we are in the plane, is embedded into $C^{0,\alpha}(\overline{\Omega})$.
The unique minimiser $u$ solves the corresponding Euler-Lagrange equation which reads
$$\Delta^2u=1\mbox{ in }\Omega\, ,\quad u=\Delta u=0\mbox{ on }\partial\Omega$$
and which may be reduced to a system involving the torsional rigidity of $\Omega$:
$$-\Delta u=v\, ,\ -\Delta v=1\mbox{ in }\Omega\, ,\quad u=v=0\mbox{ on }\partial\Omega\, .$$
Let $\lambda>0$ and denote by $\Gamma_\lambda$ the set of connected curves $\gamma$ contained in $\overline{\Omega}$, such that $\gamma\cap\partial\Omega\neq\emptyset$,
and whose length is $\lambda$: the curves $\gamma$ represent the stiffening truss to be put below the roadway. For any $\gamma\in\Gamma_\lambda$ the elastic energy of the
reinforced plate $\Omega\setminus\gamma$ is given by
$$E_\gamma(\Omega)=\ -\min_{u\in H^2\cap H^1_0(\Omega\setminus\gamma)}\ \int_{\Omega }\left(\frac{1}{2}\left(\Delta u\right)^{2}+(\sigma-1)\det(D^2u)-u\right)\, dx_1dx_2\ ,$$
and this energy should be minimised among all possible $\gamma\in\Gamma_\lambda$:
$$\min_{\gamma\in\Gamma_\lambda}\ E_\gamma(\Omega)\ .$$
Is there an optimal $\gamma_\lambda$ for any $\lambda>0$? When we asked this question to Buttazzo \cite{buttazzo} and Santambrogio \cite{filippo},
we received positive answers; their optimism is justified by the connectedness assumption and by their previous work \cite{oudet,tilli} which also gives
hints on how the line $\gamma_\lambda$ should look like. In fact, for a realistic model, one should further require that all the endpoints of $\gamma_\lambda$ lie
on the boundary $\partial\Omega$. Finally, since the stiffening truss has a cost $C>0$ per unit length, one should also solve the minimisation problem
$$\min_{\lambda\ge0}\ \left\{C\lambda+\min_{\gamma\in\Gamma_\lambda}\ E_\gamma(\Omega)\right\}\, ;$$
if $C$ is sufficiently small, we believe that the minimum exists.}\endproof\end{problem}

Problem \ref{pisa} is just a simplified version of the ``true problem'' which... should be completed!

\begin{problem}\label{upsilon1} {\em Let $\Omega=(0,L)\times(-\ell,\ell)$ and fix some $\lambda>0$. Denote by $\Gamma_\lambda$ the set of connected
curves contained in $\overline{\Omega}$ whose length is $\lambda$ and whose endpoints belong to $\partial\Omega$. Study the minimisation problem
$$\min_{\gamma\in\Gamma_\lambda}\ \left|\min_{u\in{\mathcal H}(\Omega\setminus\gamma)}\
\int_{\Omega}\left(\frac{1}{2}\left(\Delta u\right)^{2}+(\sigma-1)\det(D^2u)-u\right)\, dx_1dx_2\right|\ ,$$
where
$$
{\mathcal H}(\Omega\setminus\gamma)=\{u\in H^2(\Omega\setminus\gamma);\ \mbox{\eq{eigen2}$_2$ holds},
\ \mbox{+ something on }x_2=\pm\ell\mbox{ and on }\gamma\}\ .
$$
First of all, instead of ``something'', one should find the correct conditions on $x_2=\pm\ell$ and on $\gamma$. This could also suggest to modify the energy function
to be minimised with an additional boundary integral. The questions are similar. Is there an optimal $\gamma_\lambda$
for any $\lambda>0$? Is there an optimal $\lambda>0$ if one also takes into account the cost?}\endproof\end{problem}

Then one should try to solve the same problems with variable loads.

\begin{problem} {\em Solve Problems \ref{pisa} and \ref{upsilon1} with nonconstant loads $f\in L^2(\Omega)$, so that the minimum problem becomes
$$\min_{\gamma\in\Gamma_\lambda}\ \left|\min_{u\in{\mathcal H}(\Omega\setminus\gamma)}\
\int_{\Omega}\left(\frac{1}{2}\left(\Delta u\right)^{2}+(\sigma-1)\det(D^2u)-\, fu\right)\, dx_1dx_2\right|\ .$$
What happens if $f\not\in L^2(\Omega)$? For instance, if $f$ is a delta Dirac mass concentrated at some point $x_0\in\Omega$.}\endproof\end{problem}
\par\bigskip\noindent
{\bf Acknowledgment.} The author is grateful to his colleague Pier Giorgio Malerba, a structural engineer at the Politecnico of Milan,
for several interesting and stimulating discussions.

\end{document}